\documentclass[11pt]{amsart}
\usepackage[left=2.45cm,right=2.45cm,top=2.6cm,bottom=2.6cm]{geometry}
\usepackage{graphicx, color}
\usepackage{amscd}
\usepackage{amsmath,empheq}
\usepackage{amsfonts}
\usepackage{amssymb}
\usepackage{mathrsfs}
\usepackage[utf8]{inputenc}
\usepackage[all]{xy}

\usepackage{color,enumitem,graphicx}
\usepackage[colorlinks=true,urlcolor=blue, citecolor=red,linkcolor=blue,
linktocpage,pdfpagelabels,
bookmarksnumbered,bookmarksopen]{hyperref}

\usepackage[hyperpageref]{backref}
\usepackage[english]{babel}

\usepackage{cite}

\newtheorem{theorem}{Theorem}[section]
\newtheorem{lemma}[theorem]{Lemma}
\newtheorem{corollary}[theorem]{Corollary}
\newtheorem{proposition}[theorem]{Proposition}

\newtheorem{remark}[theorem]{Remark}
\newtheorem{definition}[theorem]{Definition}
\numberwithin{equation}{section} 
\makeindex
\newcommand{\RN}{\mathbb R^N}

\newcommand{\s}{\section}

\newcommand{\R}{\mathbb R}

\newcommand{\bt}{\begin{theorem}}
\newcommand{\et}{\end{theorem}}
\newcommand{\bl}{\begin{lemma}}
\newcommand{\el}{\end{lemma}}
\newcommand{\bd}{\begin{definition}}
\newcommand{\ed}{\end{definition}}
\newcommand{\bc}{\begin{corollary}}
\newcommand{\ec}{\end{corollary}}
\newcommand{\bp}{\begin{proof}}
\newcommand{\ep}{\end{proof}}
\newcommand{\bx}{\begin{example}}
\newcommand{\ex}{\end{example}}
\newcommand{\bi}{\begin{exercise}}
\newcommand{\ei}{\end{exercise}}
\newcommand{\bo}{\begin{prop}}
\newcommand{\eo}{\end{prop}}
\newcommand{\br}{\begin{remark}}
\newcommand{\er}{\end{remark}}
\newcommand{\be}{\begin{equation}}
\newcommand{\ee}{\end{equation}}
\newcommand{\ba}{\begin{align}}
\newcommand{\ea}{\end{align}}
\newcommand{\bn}{\begin{enumerate}}
\newcommand{\en}{\end{enumerate}}
\newcommand{\bg}{\begin{align*}}
\newcommand{\bcs}{\begin{cases}}
\newcommand{\ecs}{\end{cases}}

\newcommand{\bean}{\begin{eqnarray*}}
\newcommand{\eean}{\end{eqnarray*}}

\begin{document}

\title[{Normalized solutions for a fractional Schr\"{o}dinger-Poisson system} ]{Normalized solutions for  a fractional \\  Schr\"{o}dinger-Poisson system with critical growth}

\author[X. He]{Xiaoming He}
  \address[X. He]{College of Science
     \newline \indent
    Minzu University of China
    \newline \indent
    Beijing 100081, China}
\email{xmhe923@muc.edu.cn}

 \author[Y. Meng]{Yuxi Meng}
 \address[Y. Meng]{School of Mathematics and Statistics
     \newline \indent
   Beijing Institute of Technology
   \newline \indent
  Beijing 100081, China}
 \email{yxmeng@bit.edu.cn}

 \author[M. Squassina]{Marco Squassina}
 \address[M. Squassina]{Universit\`{a} Cattolica del Sacro Cuore
     \newline \indent
    Dipartimento di Matematica e Fisica
     \newline \indent Via della Garzetta 48, 25133, Brescia, Italy}
 \email{marco.squassina@unicatt.it}




\thanks{ This work is  supported by the National Natural Science
    Foundation of China  (121714971, 11771468, 11971027).
    Marco  Squassina  is  member  of  Gruppo  Nazionale  per
    l'Analisi  Matematica,  la Probabilita  e  le  loro  Applicazioni  (GNAMPA)  of  the  Istituto  Nazionale  di  Alta  Matematica  (INdAM)}

\begin{abstract}
In this paper, we study the fractional critical
Schr\"{o}dinger-Poisson system   \[\begin{cases} (-\Delta)^su
+\lambda\phi
u= \alpha u+\mu|u|^{q-2}u+|u|^{2^*_s-2}u,&~~ \mbox{in}~\R^3,\\
(-\Delta)^t\phi=u^2,&~~ \mbox{in}~\R^3,\end{cases}
\]  having
prescribed mass
\[\int_{\R^3} |u|^2dx=a^2,\] where   $ s, t \in (0, 1)$ satisfies $2s+2t >
3, q\in(2,2^*_s), a>0$ and $\lambda,\mu>0$   parameters and
$\alpha\in\R$ is an undetermined parameter.  Under   the
$L^2$-subcritical perturbation $q\in (2, 2+\frac{4s}{3})$, we
derive the existence of  multiple normalized solutions
  by means of the truncation technique, concentration-compactness principle and the genus
  theory. For the $L^2$-supercritical perturbation $q\in (2+\frac{4s}{3}, 2^*_s)$,
by applying the constrain variational methods and the mountain pass
theorem, we show the existence of positive normalized ground state
solutions.
\end{abstract}


\maketitle

\begin{center}
    \begin{minipage}{8cm}
        \small
        \tableofcontents
    \end{minipage}
\end{center}

\smallskip
\smallskip

 \section{Introduction}


\par In the last decade,  the following  time-dependent
 fractional Schr\"{o}dinger-Poisson system
 \be\label{e1.1}
\begin{cases}\displaystyle
 i\frac{\partial\Psi}{\partial \tau}=(-\Delta)^s\Psi+\lambda\phi\Psi-f(x,|\Psi|),& x\in \R^3, \\
(-\Delta )^{t}\phi=|\Psi|^2, &x\in\R^3,\end{cases} \ee  has
attracted much attention, where $\Psi:
\R\times\R^3\rightarrow\mathbb{C}, s,t\in(0,1), \lambda\in\R.$  It
is well-known that, the first equation in \eqref{e1.1} was used by
Laskin (see \cite{Lask1,Lask2}) to extend the Feynman path integral,
from Brownian-like to L\'{e}vy-like quantum mechanical paths.  This
class of fractional Schr\"{o}dinger equations with a repulsive
nonlocal Coulombic potential can be approximated by the Hartree-Fock
equations to describe a
 quantum mechanical system of many particles; see, for example, \cite{CHKL,LL,Longhi}, and \cite{MRS,NPV} for more applied backgrounds on the fractional Laplacian.

\par When we look for standing wave solutions to  \eqref{e1.1}, namely to
solutions of the form $(\Psi(\tau, x) =e^{-i\alpha
\tau}u(x),\phi(x)), \alpha\in\R,$ then the function $(u(x),\phi(x))$
solves the equation
 \be\label{e1.2}
\begin{cases}\displaystyle
 (-\Delta)^su+\lambda\phi u=\alpha u+f(x,u),  &x \in \R^{3},\\
                (-\Delta )^t\phi=u^2, &x \in \R^{3}.\end{cases} \ee
Here $(-\Delta)^s$ is a nonlocal operator defined by
$$(-\Delta)^s u(x)= C_{s}~ \mbox{P.V.} \int_{ \mathbb{R}^{3}}\frac{u(x)-u(y)}{|x-y|^{3+2s}}dy,~~~x\in\R^3,~~ s\in(0,1), $$
and P.V. stands for the Cauchy principal value on the integral, and
$C_s$ is a suitable normalization constant.
\par

We note that, when $\alpha\in\R$ is a fixed real number, there was a
lot of attention in recent years on the system \eqref{e1.2}
 for the existence and multiplicity of ground state solutions, bound
state solutions and concentrating  solutions, see for examples
 \cite{Wu,YZZ1,YYZ,ZDS} and  references therein. Especially,
 Zhang, do \'{O} and Squassina  \cite{ZDS} considered the existence
and asymptotical behaviors of positive solutions  as
$\lambda\rightarrow 0^+$, for the fractional Schr\"{o}dinger-Poisson
system
$$
\begin{cases} (-\Delta)^s
u+\lambda \phi  u=g(u),&x\in\R^3, \vspace{0.1cm}\\
(-\Delta)^t \phi=\lambda u^2,&x\in\R^3,
\end{cases}
$$
where $\lambda>0$ and $g$ may be subcritical or critical growth
satisfying the Berestycki-Lions conditions. In \cite{Teng},  Teng
studied the existence of a nontrivial ground state solution for  the
nonlinear fractional Schr\"{o}dinger-Poisson system with critical
Sobolev exponent
$$
\begin{cases} (-\Delta)^s
u+V(x)u+\phi  u=\mu|u|^{q-1}u+|u|^{2^*_s-2}u,&x\in\R^3, \vspace{0.1cm}\\
(-\Delta)^t \phi= u^2,&x\in\R^3,
\end{cases}
$$
where $\mu\in\R^+$ is a parameter, $1<q<2^*_s-1,~s,t\in(0,1)$ with
$2s+2t>3.$ The potential $V$ satisfies some suitable hypotheses. By
the monotonicity trick, concentration-compactness principe and a global
compactness Lemma, the author establishes the existence of ground
state solutions. Formally, system \eqref{e1.1} with $s=t=1$ can be
regarded as the following classical  Schr\"{o}dinger-Poisson system
\[\begin{cases} -\Delta u
+\lambda\phi
u= f(x,u),&~~ \mbox{in}~\R^3,\\
-\Delta\phi=u^2,&~~ \mbox{in}~\R^3,\end{cases}
\]which appears in semiconductor theory \cite{MRS} and also describes the interaction
of a charged particle with the electrostatic field in quantum
mechanics.  The literature on the Schr\"{o}dinger-Poisson system in
presence of a pure power nonlinearity is very rich, we refer to
\cite{Wu,YYZ,YZZ2} and references therein.

\par Alternatively,   from a physical point of view, it is interesting to find
solutions of  \eqref{e1.2}  with prescribed $L^2$-norms, $\alpha$
appearing as Lagrange multiplier. Solutions of this type are often
referred to as normalized solutions.  The occurrence of the
$L^2$-constraint renders several methods developed to deal with
variational problems without constraints useless, and the
 $L^2$-constraint induces a new critical exponent, the $L^2$-critical
exponent given by
 \[\bar{q} := 2 + \frac{4s}{3},\]
  and the number
$\bar{q}$ can keep the mass invariant by the law of conservation of
mass. Precisely for this reason, $2 + \frac{4s}{3}$ is called
 $L^2$-critical exponent or mass critical exponent, which is the
threshold exponent for many dynamical properties such as global
existence, blow-up, stability or instability of ground states. In
particular, it strongly influences the geometrical structure of the
corresponding functional. Meanwhile, the appearance of the
$L^2$-constraint makes some classical methods, used to prove the
boundedness of any Palais-Smale sequence for the unconstrained
problem, difficult to implement.    In \cite{LT}, Li and  Teng
proved the existence of normalized solutions to the following
fractional Schr\"{o}dinger-Poisson system:

\be\label{e1.3} \begin{cases} (-\Delta)^su +\phi u= \lambda u+f(u),&~~ \mbox{in}~\R^3, \\
(-\Delta)^t\phi=u^2, &~~ \mbox{in}~\R^3,\\
\displaystyle\int_{\R^N}|u|^2dx=a^2,&
\end{cases} \ee
  where $s \in (0,1),
2s+2t>3, \lambda\in\R $ and $f\in C^1(\R,\R)$ satisfies some general
conditions which contain the case $f(u)\sim |u|^{q-2}u$ with
$q\in(\frac{4s+2t}{s+t},2+\frac{4s}{3})\cup (2+\frac{4s}{3},2^*_s)$,
i.e., the nonlinearity $f$ is $L^2$-mass subcritical or $L^2$-mass
supercritical growth, but is Sobolev subcritical growth. In
\cite{YZZ1},  Yang,    Zhao,  and  Zhao     showed    the existence
of infinitely many solutions $(u, \lambda)$  to \eqref{e1.3} with
 subcritical nonlinearity $\mu|u|^{q-2}u$, by using the cohomological
index theory.

\par We note that, when  $ s = t=1,$   problem   \eqref{e1.3}, are related  to the the following  equation

\be\label{e1.4} \begin{cases} -\Delta u +\lambda u-\gamma (|x|^{-1}\ast|u|^2) u= a|u|^{p-2}u,&~~ \mbox{in}~\R^3,\\
\displaystyle\int_{\R^N}|u|^2dx=c^2,~~u\in H^1(\R^3).& \end{cases} \ee  Recently,
 Jeanjean and Trung Le in \cite{JLe2} studied the existence of
normalized solutions for \eqref{e1.4} when $\gamma>0$ and $a>0$,
both in the Sobolev subcritical case $p\in (10/3, 6)$ and in the
Sobolev critical case $p=6,$ they  showed that there exists a
$c_1>0$ such that, for any $c\in(0,
 c_1),$ \eqref{e1.4} admits two solutions $u^+_c$ and $u^-_c$ which can be characterized
respectively as a local minima and as a mountain pass critical point
of the associated energy functional restricted to the norm
constraint. While in the case $\gamma<0, a>0$ and $p=6$ the authors
showed that \eqref{e1.4} does not admit positive solutions.
Bellazzini, Jeanjean and Luo \cite{BJL}  proved  that for $c > 0$
sufficiently small, there exists a critical point which minimizes
with prescribed $L^2$-norms. In \cite{JL}, Jeanjean and Luo studied
the existence  of minimizers for with $L^2$-norm for \eqref{e1.4},
and they
   expressed  a threshold value of $c> 0$ separating existence and
nonexistence of minimizers. In \cite{WQ}, Wang and Qian established
the existence of ground state and infinitely many radial solutions
to \eqref{e1.4} with $a|u|^{p-2}u$ replaced by a general subcritical
nonlinearity $af(u)$, by constructing a particular bounded
Palais-Smale sequence when $\gamma< 0, a > 0.$   In \cite{LZ}, Li
and Zhang  studied  the existence of positive normalized ground
state solutions for a class of   Schr\"{o}dinger-Popp-Podolsky
system. For more results  on  the existence and no-existence of
normalized solutions  of Schr\"{o}dinger-Poisson systems,  we refer
to \cite{BJS,BS,BS1,BS2,HLW,JL,JLe2,Luo,Y,YZZ1} and references
therein.

 \par   After the above bibliography review we have found
only two papers \cite{LT,YZZ1} considering the normalized solutions
for the fractional Schr\"{o}dinger-Poisson system by the prescribed
mass approaches with the nonlinearity $f(u)$, being Sobolev
subcritical growth.

A natural question arises: How to obtain solutions to system
\eqref{e1.3} in presence of the nonlinear term
$f(u)=\mu|u|^{q-2}u+|u|^{2^*_s-2}u$, combining the Sobolev critical
term with a subcritical perturbation?

 The main contribution of this paper is to give an affirmative answer to this question and  fill this
 gap. To be specific, in the present paper  we aim to study the following   fractional Schr\"{o}dinger-Poisson
system
 \be\label{e1.5} \begin{cases} (-\Delta)^su +\lambda\phi u= \alpha u+\mu|u|^{q-2}u+|u|^{2^*_s-2}u,&~~ \mbox{in}~\R^3, \\
(-\Delta)^t\phi=u^2,&~~ \mbox{in}~\R^3, \end{cases} \ee  having  prescribed
 $L^2$-norm \be\label{e1.6} \int_{\R^3} |u|^2dx=a^2, \ee where   $s,
t \in (0, 1)$ satisfies $2s+2t > 3, q\in(2,2^*_s)$  and
$\alpha\in\R$ is an undetermined parameter, $\mu,\lambda>0$ are
parameters.  For this purpose, applying the reduction argument
introduced in \cite{ZDS},   system
 \eqref{e1.5} is equivalent to the following single equation

\be\label{e1.7}  (-\Delta)^su +\lambda\phi^t_u u= \alpha
u+\mu|u|^{q-2}u+|u|^{2^*_s-2}u,~~~x\in\R^3,\ee where
\[\phi^t_u(x)=c_t\int_{\R^3}\frac{|u(y)|^2}{|x-y|^{3-2t}}dy,~~~\mbox{and}~~~c_t:=\frac{\Gamma(\frac{3}{2}-2t)}{\pi^32^{2t}\Gamma(t)}.\]
We shall look for solutions to \eqref{e1.5}-\eqref{e1.6},  as a
critical points   of the action functional
\[
I_{\mu}(u)=\frac{1}{2}\int_{\R^3}|(-\Delta )^\frac{s}{2}
u|^2dx+\frac{\lambda}{4}
\int_{\R^3}\phi^t_u|u|^2dx-\frac{\mu}{q}\int_{\R^3}|u|^qdx
-\frac{1}{2^*_s}\int_{\R^3}|u|^{2^*_s}dx,
\]
restricted on the
set
\[S_a=\left\{u\in H^s(\R^3):~\int_{\R^3}|u|^2dx=a^2\right\},
\]
with $\alpha$ being the Lagrange multipliers, Clearly,   each
critical point $u_a\in  S_a$ of $I_{\mu}|_{S_a}$, corresponds a
Lagrange multiplier $\alpha\in\R$ such that $(u_a, \alpha)$ solves
\eqref{e1.7}. In particular, if $u_a\in  S_a$ is a minimizer of
problem

$$
m(a):=\inf_{u\in S_a}I_{\mu}(u),
$$
then there exists $\alpha\in\R$ as a Lagrange multiplier and then
$(u_a, \alpha)$ is a weak solution of  \eqref{e1.7}. As far as we
know, there is no result about the existence of normalized solutions
for Schr\"{o}dinger-Poisson system with a critical term in the
current literature. For this aim, we shall focus our attention on
the existence, asymptotic and multiplicity of normalized solutions
for problem \eqref{e1.5}- \eqref{e1.6}.

\vskip0.2in

 \section{The main results}

\par In this section we formulate the main results. We first deal
with the existence of
 multiple normalized ground state solutions in the $L^2$-subcritical case: $q\in
(2,2+\frac{4s}{3})$. Secondly, we are  concerned with the existence
and asymptotic behavior of positive normalized ground state
solutions of Schr\"{o}dinger-Poisson  system \eqref{e1.7} in the
$L^2$-supercritical case: $q\in (2+\frac{4s}{3}, 2^*_s)$.

To state the main results, for $\delta_{q,s}=3(q-2)/2qs$, we introduce the following constants:
\be\label{e2.1}
  D_1:=2^{-\frac{q\delta_{q,s}-2}{2^*_s-2}}S^{\frac{3(2^*_s-q)}{2s(2^*_s-2)}};
  \ee
\be\label{e2.2}
 D_2:=D(s,t)^{-1}S^{\frac{3[(2^*_s-2)-q(1-\delta_{q,s})]}{2s(2^*_s-2)}},
\ee where \be\label{e2.3}D(s,t):=\left(\frac{(3-2t)\lambda
 \Gamma_t}{2s}\right)^{\frac{(q\delta_{q,s}-2)s}{s2^*_{s}+2t-3}},\ee
 and $\Gamma_t$ is given in \eqref{e3.3}.

\par The  first result is concerned with the multiplicity of normalized solutions for the $L^2$-subcritical perturbation, which   can be
formulated as

\bt\label{Theorem 2.1} Let $\mu,\lambda, a>0$, and   $q\in
(2,2+\frac{4s}{3})$. Then, for a given $k\in \mathbb{N}$, there
exists $\beta>0$ independent of $k$ and $\mu^*_k>0$ large, such that
problem \eqref{e1.5}-\eqref{e1.6} possesses at least $k$ couples
$(u_j,\alpha_j)\in H^s(\R^3)\times \R$ of weak solutions for
$\mu>\mu_k$ and \be\label{e2.4}a\in
\left(0,\left(\frac{\beta}{\mu}\right)^\frac{1}{q(1-\delta_{q,s})}\right)
\ee with $\int_{\R^3}|u_j|^2dx=a^2$, $\alpha_j<0$ for all
$j=1,\cdots,k$.\et

\par The second result of this paper is concerned with the existence and asymptotical behavior of normalized
solutions for the $L^2$-supercritical perturbation when the parameters
$\lambda,\mu>0$ are suitably small.

\smallskip
\bt\label{Theorem 2.2} Let $q\in (2+\frac{4s}{3}, 2^*_s)$, assume
that $\mu,a>0$ satisfy  the following inequality
 \be\label{e2.5}
 \mu {\delta_{q,s}}\max\left\{a^{q(1-\delta_{q,s})},a^{\frac{(q-2)2t+2s(2^*_s-4)}{s2^*_{s}+2t-3}}\right\}
  <\min\{D_1, D_2\},
  \ee
where $\delta_{q,s}=3(q-2)/2qs.$ Then, there exists $\Lambda^*>0$
such that for $0<\lambda<\Lambda^*$, problem
\eqref{e1.5}-\eqref{e1.6} possesses a positive normalized ground
state solution $u_{\alpha}\in H^s(\R^3)$ for some $\alpha<0$. \et

\par Finally, we present an existence result of normalized
solutions under the $L^2$-supercritical perturbation, when
parameter $\mu>0$ is large.

\vskip0.1in

\bt\label{Theorem 2.3} If $2+\frac{4s}{3}<q<2^*_s $, there exists
$\mu^\ast=\mu^\ast(a)>0$ large,  such that as $\mu>\mu^\ast$,
problem
 \eqref{e1.5}-\eqref{e1.6} possesses a couple $(u_a,\alpha)\in H^{s}(\R^3) \times \R$
of weak solutions with $\int_{\R^3}|u_a|^2dx=a^2$, $\alpha<0$. \et
\vskip0.1in

\par \noindent {\bf Remark 2.1.} (i) { Theorems \ref{Theorem 2.1}-\ref{Theorem 2.3}
improve and complement the main results in \cite{Teng,ZDS} in the
sense that, we are concerned with the normalized solutions.

\par  (ii) Our studies
improve and fill in gaps of the main works of \cite{LT,ST,YZZ1},
since we consider the existence of normalized solutions to
\eqref{e1.5}-\eqref{e1.6} with Sobolev critical growth. }

 \vskip0.1in

\subsection{Remarks on the proofs}
We give some comments on the proof for the main results
above. Since the   critical terms $|u|^{2^*_s-2}u$  is
  $L^2$-supercritical, the functional $I_{\mu}$ is always unbounded from below
on $S_a,$  and this makes it  difficulty to deal  with existence of
normalized solutions  on the $L^2$- constraint. One of the main
difficulties that one has to face in such context is the analysis of
the convergence of constrained Palais-Smale sequences: In fact, the
critical growth term in the equation makes the bounded (PS)
sequences possibly not convergent;  moreover,  the   Sobolev critical
term $|u|^{2^*_s-2}u$ and      nonlocal   convolution term
$\lambda\phi^t_u u$, makes  it  more complicated  to estimate the
critical value of mountain pass, and one has to consider how the
interaction between the nonlocal term and the nonlinear term, and
  the energy balance between these competing terms needs to be
controlled through moderate adjustments of  parameter $\lambda>0.$
  Another of   difficulty  is that sequences of approximated Lagrange
multipliers have to be controlled, since $\alpha$ is not prescribed;
and moreover, weak limits of Palais-Smale sequences could leave the
constraint, since the embeddings $H^s(\R^3)\hookrightarrow
L^2(\R^3)$ and also $H^s_{{\rm rad}}(\R^3)\hookrightarrow L^2(\R^3)$ are
not compact.

\par To overcome these difficulties, we employ
Jeanjean's  theory \cite{J} by  showing that the mountain pass
geometry of $I_{\mu}|_{S_a}$ allows to construct a Palais-Smale
sequence of functions satisfying the Pohozaev identity. This gives
boundedness, which is the first step in proving strong
$H^s$-convergence.  As naturally expected, the presence of the
Sobolev critical term   in \eqref{e1.5} further complicates the
study of the convergence of Palais-Smale sequences. To overcome the
loss of compactness caused by the
 critical growth, we shall employ the concentration-compactness principle, mountain pass theorem and
energy estimation   to obtain the existence of normalized ground
states   of \eqref{e1.5}, by showing that, suitably combining some
of the main ideas from \cite{SV,Soave1}, compactness can be restored
  in the present setting.

\par  Finally,  let us sketch the ideas and methods used along this paper to obtain
our main results.  For the $L^2$-subcritical perturbation: $q\in (2,
2+\frac{4s}{3})$, it is difficult to get    the boundedness of the
(PS) sequence by the idea of \cite{J}. To get over this difficulty,
we use the truncation technique; to  restore the loss of compactness
of the (PS) sequence caused by the critical growth, we apply for the
concentration-compactness principle; and to obtain the multiplicity
of normalized solutions of \eqref{e1.5}-\eqref{e1.6}, we employ the
genus theory. For the $L^2$-supercritical perturbation: $q\in
(2+\frac{4s}{3}, 2^*_s),$ we use the Pohozaev manifold and  mountain
pass theorem to prove the existence of positive ground state
solutions for system \eqref{e1.5}-\eqref{e1.6} when $\mu>0$  small.
 While if the parameter $\mu>0$ is large, we employ a
fiber map and the concentration-compactness principle to  prove that
the (PS) sequence is strongly convergent, to obtain a normalized
solution of \eqref{e1.5}-\eqref{e1.6}.

\vskip0.2in

\par
\subsection{Paper outline}
This paper is organized as follows.

\noindent $\bullet$ Section  2 provides the main results, and
 Section 3 presents  some preliminary results that will be used frequently in the sequel.

\noindent $\bullet$ Section 4 presents the multiplicity of
normalized ground state solutions for system
\eqref{e1.5}-\eqref{e1.6} when $q\in (2, 2+\frac{4s}{3})$, and
finish the proof of Theorem \ref{Theorem 2.1}.

\noindent $\bullet$  Section 5 proves the existence of normalized
positive ground state solutions for problem
\eqref{e1.5}-\eqref{e1.6} when $q\in (2+\frac{4s}{3}, 2^*_s),$  and
Theorem \ref{Theorem 2.2} is proved if $\mu,\lambda>0$ are suitably
 small.

\noindent $\bullet$ In  Section 6 we give another existence result
for problem \eqref{e1.5}-\eqref{e1.6} with  $q\in (2+\frac{4s}{3},
2^*_s),$ when the parameter $\mu>0$ is large, and finishes the proof
of Theorem \ref{Theorem 2.3}.

\vskip0.2in
\par \noindent {\bf Notations.} In the sequel of  this paper, we denote by   $C,C_i > 0$
different positive constants whose values may vary from line to line
and are not essential to the problem. We denote by   $L^p =
L^p(\R^3)$ with $1< p \leq\infty$ the Lebesgue space with the
standard norm $\|u\|_p=\left(\int_{\R^3}|u|^pdx\right)^{1/p}.$


\vskip0.2in

  \s{Preliminary stuff}
    In this section, we first  give   the functional space setting, and  sketch  the fractional order Sobolev spaces  \cite{NPV}. We recall that,
for any
 $s\in (0, 1),$   the nature functions space associated with $(-\Delta)^s$
  is  $H:=H^s(\R^3)$ which is  a Hilbert space equipped
with the inner product and norm, respectively given by
\[\left<u,v\right>:=\int_{\R^3}((-\Delta)^{\frac{s}{2}}u(-\Delta)^{\frac{s}{2}}v+uv)dx, ~\|u\|^2_{H}=\left<u,u\right>.\]
\par The   homogeneous fractional Sobolev space
$D^{s,2}(\R^3)$ is defined by
\[D^{s,2}(\R^3)=\left\{u\in L^{2^*_s}(\R^3):\iint_{\R^{6}}\frac{|u(x)-u(y)|^2}{|x-y|^{3+2s}}dxdy<+\infty\right\},\]
a completion of $C_0^\infty(\R^3)$ under the norm
\[\|u\|^2:=\|u\|^2_{D^{s,2}(\R^3)}=\iint_{\R^{6}}\frac{|u(x)-u(y)|^2}{|x-y|^{3+2s}}dxdy,\]where
$2^*_s=6/(3-2s)$ is the critical exponent.  From
Proposition 3.4 and 3.6 in \cite{NPV} we have
$$
\|u\|^2=\|(-\Delta)^{\frac{s}{2}}u\|_2^2=\iint_{\R^{6}}\frac{|u(x)-u(y)|^2}{|x-y|^{3+2s}}dxdy.
$$
The best fractional Sobolev constant $S$ is defined as
\be\label{e3.1} S=\inf_{u\in
D^{s,2}(\R^3),u\neq0}\frac{\|(-\Delta)^{\frac{s}{2}}u\|_2^2}{(\int_{\R^3}|u|^{2^*_s}dx)^{\frac{2}{2^*_s}}}.
 \ee
 The work space $H^s_{rad}(\R^3)$ is defined by
\[
H^s_{rad}(\R^3):=\left\{u\in H^s(\R^3): ~u~ \mbox{is radially
decreasing}\right\}.
\]
Let $\mathbb{H} = H\times\R$ with the scalar product
$\langle\cdot,\cdot\rangle_{H}+\langle\cdot,\cdot\rangle_{\R},$ and
the corresponding norm
$\|(\cdot,\cdot)\|^2_{\mathbb{H}}=\|\cdot,\cdot\|_{H}^2+|\cdot,\cdot|_{\R}^2.$
\vskip0.1in
\par The following two inequalities play an important role in the proof
of our main results.

\begin{proposition}\label{prop 3.1}(Hardy-Littlewood-Sobolev inequality \cite{LL})
Let $l,r>1$ and $0<\mu<N$ be such that
 $\frac{1}{r}+\frac{1}{l}+\frac{\mu}{N}=2, f\in L^r(\RN)$ and $h\in L^l(\RN)$. Then there exists a constant $C(N,\mu,r,l)>0$ such that
$$
\left|\int_{\R^{N}}\int_{\R^{N}}f(x)h(y)|x-y|^{-\mu}dxdy\right|\leq
C(N,\mu,r,l)\|f\|_r\|h\|_l.
$$
\end{proposition}

\par We recall  the fractional Gagliardo-Nirenberg inequality.
\bl\label{Lemma 3.2}(\cite{FLS}) Let $ 0<s <1,$ and $p\in (2,
2^*_s).$ Then there exists a constant $C(p,s) =
S^{-\frac{\delta_{p,s}}{2}}>0$ such that
\be\label{e3.2}\|u\|_p^p\leq
{C}(p,s)\|(-\Delta)^{\frac{s}{2}}u\|_2^{p\delta_{p,s}}\|u\|_2^{p(1-\delta_{p,s})},~~~\forall
u\in H^s(\R^3),\ee where  $\delta_{p,s}=3(p-2)/2ps.$ \el

\bl\label{Lemma 3.3} (Lemma 5.1  \cite{DS}) If $u_n\rightharpoonup
u$ in $H^s_{rad}(\R^3)$, then
\[\int_{\R^3}\phi_{u_n}^tu_n^2dx\rightarrow \int_{\R^3}\phi_{u}u^2dx,\]
and \[\int_{\R^3}\phi_{u_n}^tu_n\varphi dx\rightarrow
\int_{\R^3}\phi^t_{u}u\varphi dx,~~\forall\varphi\in
H^s_{rad}(\R^3).\] \el

\par From Proposition \ref{prop 3.1}, with $l= r = \frac{6}{3+2t}$, then Hardy-Littlewood-Sobolev
inequality implies that: \be\label{e3.3}\int_{\R^3}\phi_u^tu^2dx=
\int_{\R^3}\left(\frac{1}{|x|^{3-2t}}\ast u^2\right)u^2dx\leq
\Gamma_t\|u\|_{\frac{12}{3+2t}}^4.\ee It is easy to enumerate that
\[q\delta_{q,s}\left\{\begin{array}{ll}
<2,&\mbox{if}~~2<q<\bar{q};
\\ =2,&\mbox{if}~~q=\bar{q};\\
>2,&\mbox{if}~~ \bar{q}<q<2^*_s,
\end{array}\right.\]
where $\bar{q}:=2+\frac{4s}{3}$ is the $L^2$-critical exponent.

\par Now,  we introduce the Pohozaev mainfold associated to
\eqref{e1.7}, which can be    derived from \cite{Teng}.

 \begin{proposition}\label{prop 3.4} Let $u\in H^s(\RN)\cap L^\infty(\RN)$ be a weak solution of \eqref{e1.7}, then $u$ satisfies the equality
  $$
  \frac{3-2s}{2}\|u\|^2+\frac{2t+3}{4}\lambda\int_{\R^3}\phi_u^tu^2dx=\frac{3\alpha}{2}\|u\|^2_2+\frac{3\mu}{q}\int_{\R^3}|u|^qdx
  +\frac{3}{2^*_{s}}\int_{\R^3}|u|^{2^*_{s}}dx.
  $$
  \end{proposition}

 \bl\label{Lemma 3.5} Let $u\in H^s(\RN)$ be a weak solution of \eqref{e1.7}, then we can construct the following Pohozaev manifold
  \[\mathcal {P}_{a}=\{u\in S_a: P_\mu(u)=0\},\] where
  $$
  P_\mu(u)=s\|u\|^2+\frac{3-2t}{4}\lambda\int_{\R^3}\phi_u^tu^2dx-s\mu \delta_{q,s}\int_{\R^3}|u|^qdx
  -s\int_{\R^3}|u|^{2^*_{s}}dx.
  $$
  \el
  \bp
  From Proposition \ref{prop 3.4}, we know that $u$ satisfies the Phohzaev identity as follows
   \be\label{e3.4}
  \frac{3-2s}{2}\|u\|^2+\frac{2t+3}{4}\lambda\int_{\R^3}\phi_u^tu^2dx=\frac{3\alpha}{2}\|u\|^2_2+\frac{3\mu}{q}\int_{\R^3}|u|^qdx
  +\frac{3}{2^*_{s}}\int_{\R^3}|u|^{2^*_{s}}dx.
  \ee
  Moreover, since $u$ is the weak solution of  system \eqref{e1.7}, we have
  \be\label{e3.5}
  \|u\|^2+\lambda\int_{\R^3}\phi_u^tu^2dx=\alpha\|u\|^2_2+\mu\int_{\R^3}|u|^qdx
  + \int_{\R^3}|u|^{2^*_{s}}dx.\ee
Combining with \eqref{e3.4} and \eqref{e3.5}, we get
\[s\|u\|^2+\frac{3-2t}{4}\lambda\int_{\R^3}\phi_u^tu^2dx=s\mu \delta_{q,s}\int_{\R^3}|u|^qdx
+s\int_{\R^3}|u|^{2^*_{s}}dx,\] which finishes the proof.  \ep

\par Finally, we  state the following well-known embedding result.
 \bl\label{Lemma 3.6} (\cite{Dinh}). Let $N\geq  2.$ The
embedding $H^s_{rad}(\R^N)\hookrightarrow L^p(\R^N)$ is compact for
any $2 < p < 2^*_s.$ \el

\vskip0.2in

\s{  Proof of Theorem \ref{Theorem 2.1}}

In this section, we aim to show the multiplicity of normalized
solutions to  \eqref{e1.5}-\eqref{e1.6}. To begin with, we recall
the definition of a genus. Let $X$ be a Banach space and let $A$ be
a subset of $X$. The set $A$ is said to be symmetric if $u\in A$
implies that $-u \in A$. We denote the set
\[
\Sigma:=\{A\subset X\setminus\{0\}:A ~\mbox{is closed and symmetric
with respect to the origin}\}.
\]
For $A\in \Sigma$, define
\[
\gamma(A)=\left\{\begin{array}{l} 0,~~\mbox{if}~~A=\emptyset,\vspace{0.2cm}\\
 \inf\{k\in \mathbb{N} :\exists ~\mbox{an odd}~ \varphi \in C(A,\R^k\setminus\{0\})\}, \vspace{0.2cm}\\
 +\infty,\mbox{if no such odd map},
\end{array}\right.\]
and that $\Sigma_k=\{A\in \Sigma: \gamma(A)\geq k\}$.

\par In order to overcome the loss of compactness of the (PS)
sequences, we need to apply for the following
concentration-compactness principle.
\begin{lemma}\label{Lemma 4.1}(\cite{ZZR})  Let $\{u_n\}$ be a bounded
sequence in ${D}^{s,2}(\R^3)$ converging weakly and a.e. to some $u
\in {D}^{s,2}(\R^3)$. We have that
$|(-\Delta)^{\frac{s}{2}}u_n|^2\rightharpoonup \omega$ and
$|u_n|^{2^*_s}\rightharpoonup \zeta$ in the sense of measures.
  Then, there exist some at most a countable set
$J$, a family of points $\{z_j\}_{j\in J}\subset \R^3$, and families
of positive numbers   $\{\zeta_j\}_{j\in J}$ and $\{\omega_j\}_{j\in
J}$ such that

\begin{equation}\label{e4.1}
    \omega \geq|(-\Delta)^{\frac{s}{2}}u|^2+\sum_{j\in J}\omega_j\delta_{z_j},
\end{equation}
\begin{equation}\label{e4.2}
   \zeta=|u|^{2^*_s}+\sum_{j\in J}\zeta_j\delta_{z_j}
\end{equation}
and
\begin{equation}\label{e4.3}
  \omega_j\geq S\zeta^\frac{2}{2^*_s}_j,
    \end{equation}
where $\delta_{z_j}$ is the Dirac-mass of mass 1 concentrated at
$z_j \in \R^3$.
\end{lemma}

\begin{lemma}\label{Lemma 4.2} (\cite{ZZR})  Let
$\{u_n\}\subset{D}^{s,2}(\R^3)$ be a sequence in Lemma \ref{Lemma
4.1} and define that
\[\omega_\infty:=\lim_{R\rightarrow\infty}\limsup_{n\rightarrow\infty}
 \int_{|x|\geq R}|(-\Delta)^{\frac{s}{2}}u_n|^2dx,~~~
 \zeta_\infty:=\lim_{R\rightarrow\infty}\limsup_{n\rightarrow\infty}
 \int_{|x|\geq R}|u_n|^{2^*_s}dx.
 \]
 Then it follows that
\begin{equation} \label{e4.4}\omega_\infty\geq S\zeta^\frac{2}{2^*_s}_\infty,\end{equation}
\begin{equation}\label{e4.5}\limsup_{n\rightarrow\infty}\int_{\R^3}|(-\Delta)
^{\frac{s}{2}}u_n|^2dx=\int_{\R^3}d\omega+\omega_\infty\end{equation}
and
\begin{equation}\label{e4.6}\limsup_{n\rightarrow\infty}\int_{\R^3}|u_n|^{2^*_s}dx
=\int_{\R^3}d\zeta+\zeta_\infty.\end{equation}
\end{lemma}

\vskip0.1in \par For $u\in S_{r,a}$, in view of Lemma \ref{Lemma
3.2},
 and the Sobolev inequality, one has that
\begin{equation}\label{e4.7}\begin{split} I_{\mu}(u)=&\frac{1}{2}\int_{\R^3}|(-\Delta
)^\frac{s}{2} u|^2dx+\frac{\lambda}{4}
\int_{\R^3}\phi^t_uu^2dx-\frac{\mu}{q}\int_{\R^3}|u|^qdx -
  \frac{1}{2^*_s}\int_{\R^3}|u|^{2^*_s}dx\\
  \geq& \frac{1}{2}\|(-\Delta )^\frac{s}{2} u\|^2_2
   -\frac{\mu}{q}a^{q(1-\delta_{q,s})}C_{q,s}\|(-\Delta )^\frac{s}{2} u\|_2^{q\delta_{q,s}}  -\frac{1}{2^*_s}S^{-\frac{2^*_s}{2}}\|(-\Delta )^\frac{s}{2} u\|_2^{2^*_s}\\
 :=&g(\|(-\Delta )^\frac{s}{2} u\|_2),
\end{split}
\end{equation}
where
\[
g(r)=\frac{1}{2}r^2-\frac{\mu}{q}a^{q(1-\delta_{q,s})}C_{q,s}r^{q\delta_{q,s}}
-\frac{1}{2^*_s}S^{-\frac{2^*_s}{2}}r^{2^*_s}.
\]
Recalling that $2<q<2+\frac{4s}{3}$, we get that $q\delta_{q,s}<2$,
and there exists $\beta> 0$ such that, if $\mu
a^{q(1-\delta_{q,s})}\leq\beta$,
 the function $g$ attains its positive local maximum. More precisely, there exist two constants $0<R_1<R_2<+\infty$, such that
  \[g(r)>0,~ \forall r\in (R_1,R_2);~~~g(r)<0, ~\forall r\in (0,R_1)\cup (R_2,+\infty).\] Let $\tau:\R^+\rightarrow [0,1]$  be a
  nonincreasing and $C^\infty$ function satisfying
\[
\tau(r)=\left\{\begin{array}{ll} 1,&\mbox{if}~~r\in [0,  R_1],\vspace{0.2cm}\\
0,&\mbox{if}~~r\in [R_2,+\infty).
\end{array}\right.\]

In the sequel, let us consider the truncated functional
\[
I_{\mu,\tau}(u)=\frac{1}{2}\int_{\R^3}|(-\Delta)^\frac{s}{2}u|^2dx+\frac{\lambda}{4}\int_{\R^3}\phi^t_uu^2dx-\frac{\mu}{q}\int_{\R^3}|u|^qdx-
  \frac{\tau(\|(-\Delta )^\frac{s}{2} u\|_2)}{2^*_s}\int_{\R^3}|u|^{2^*_s}dx.
\]
For $u\in S_{r,a}$, again by Lemma \ref{Lemma 3.2},
 and the Sobolev inequality,  it is easy to see that
\[\begin{split}
I_{\mu,\tau}(u) \geq& \frac{1}{2}\|(-\Delta )^\frac{s}{2} u\|^2_2
-\frac{\mu}{q}a^{q(1-\delta_{q,s})}C_{q,s}\|(-\Delta)^\frac{s}{2}
u\|_2^{q\delta_{q,s}}
-\frac{\tau(\|(-\Delta )^\frac{s}{2} u\|_2)}{2^*_s}S^{-\frac{2^*_s}{2}}\|(-\Delta )^\frac{s}{2} u\|_2^{2^*_s}\\
 :=&\widetilde{g}(\|(-\Delta )^\frac{s}{2} u\|_2),
\end{split}
\]
where
\[
\widetilde{g}(r)=\frac{1}{2}r^2
  -\frac{\mu}{q}a^{q(1-\delta_{q,s})}C_{q,s}r^{q\delta_{q,s}}-\frac{\tau(r)}{2^*_s}S^{-\frac{2^*_s}{2}}r^{2^*_s}.
\]
Then, by the definition of $\tau(\cdot)$, when $a\in
(0,(\frac{\beta}{\mu})^{\frac{1}{q(1-\delta_{q,s})}})$, we have
\[\widetilde{g}(r)<0,~~ \forall
r\in(0,R_1);~~~\widetilde{g}(r)>0,~~\forall r\in (R_1,+\infty).\] In
what follows, we always assume that $a\in
(0,(\frac{\beta}{\mu})^{\frac{1}{q(1-\delta_{q,s})}})$. Without loss
of generality, in the sequel, we may assume that
\begin{equation}\label{e4.8}
\frac{1}{2}r^2-\frac{1}{2^*_s}S^{-\frac{2^*_s}{2}}r^{2^*_s}
\geq0,~\forall~r\in[0,R_1]
\end{equation}
and
\begin{equation}\label{e4.9}
R_1<S^{\frac{3}{4s}}.\end{equation}

\begin{lemma}\label{Lemma 4.3} The functional $I_{\mu,\tau}$ has the following  characteristics:
 \begin{itemize}\item[(i)]$I_{\mu,\tau} \in C^1\left(H^s_{rad}(\R^3),\R\right);$
 \item[(ii)]$I_{\mu,\tau}$ is coercive and bounded from below on $S_{r,a}$. Moreover, if $I_{\mu,\tau}(u)\leq0$, then $\|(-\Delta)^{\frac{s}{2}}u\|_2\leq R_1$ and $I_{\mu,\tau}(u)=I(u)$;
  \item[(iii)] $I_{\mu,\tau}|_{S_{r,a}}$ satisfies the $(PS)_c$ condition for all
  $c<0,$ provided that $\mu>\mu^*_1>0$ large.
 \end{itemize} \end{lemma}

\begin{proof} We can obtain conclusions $(i)$ and $(ii)$ by   a standard argument.
 To prove item  $(iii)$,  let $\{u_n\}$ be a $(PS)_c$ sequence of $I_{\mu,\tau}$ restricted to $S_{r,a}$ with $c < 0$.  By $(ii)$, we see that
$\|(-\Delta)^{\frac{s}{2}}u_n\|_2\leq R_1$ for large $n$, and thus
  $\{u_n\}$ is  a $(PS)_c$ sequence of $I|_{S_{r,a}}$ with $c < 0$;
i.e., $I(u_n)\rightarrow c<0$ and $\|I|'_{S_{r,a}}(u_n)\|\rightarrow
0$ as $n\rightarrow\infty$. Then, $\{u_n\}$ is bounded in
$H^s_{rad}(\R^3)$. Therefore, up to a subsequence, there exists $u
\in H^s_{rad}(\R^3)$ such that $u_n \rightharpoonup u$ in $
H^s_{rad}(\R^3)$ and $u_n \rightarrow u$ in $L^p(\R^3)$ for
$2<p<2^*_{s}$ and $u_n(x) \rightarrow u(x)$ a.e. on $\R^3$. From
$2<q<2+\frac{4s}{3}<2^*_{s}$ and Lemma \ref{Lemma 3.3}, we infer to
\[
\lim_{n\rightarrow\infty}\int_{\R^3}|u_n|^qdx=\int_{\R^3}|u|^qdx,~~~\int_{\R^3}\phi_{u_n}^tu_n^2dx\rightarrow\int_{\R^3}\phi^t_uu^2dx.
\]
Moreover, we have that $u\not\equiv0$. Indeed, assume by
contradiction that, $u \equiv0$, then
$\lim_{n\rightarrow\infty}\int_{\R^3}|u_n|^qdx=0$. From \eqref{e4.8}
and the definition of $I_{\mu,\tau}$, we infer that
\[\begin{split}
0>c=&\lim_{n\rightarrow\infty}I_{\mu,\tau}(u_n)=\lim_{n\rightarrow\infty}I_{\mu}(u_n)\\
=&\lim_{n\rightarrow\infty}\bigg[\frac{1}{2}\int_{\R^3}|(-\Delta
)^\frac{s}{2} u_n|^2dx+\frac{\lambda}{4}
\int_{\R^3}\phi_{u_n}^tu_n^2dx-\frac{\mu}{q}\int_{\R^3}|u_n|^qdx-\frac{1}{2^*_s}\int_{\R^3}|u_n|^{2^*_s}dx\bigg]\\
\geq& \lim_{n\rightarrow\infty}\bigg[\frac{1}{2}\|(-\Delta
)^\frac{s}{2} u_n\|^2_2 -\frac{1}{2^*_s}S^{-\frac{2^*_s}{2}}\|(-\Delta )^\frac{s}{2} u_n\|_2^{2^*_s}-\frac{\mu}{q}\int_{\R^3}|u_n|^qdx\bigg]\\
 \geq&-\frac{\mu}{q} \lim_{n\rightarrow\infty}\int_{\R^3}|u_n|^qdx=0,
\end{split}
\]
which is absurd. On the other hand, setting the function $\Theta(v):
H^s_{rad}(\R^3)\rightarrow\R$  by \[
\Theta(v)=\frac{1}{2}\int_{\R^3}|v|^2dx,
\]
it follows that $S_a=\Theta^{-1}(\{\frac{a^2}{2}\})$.  Then, by
Proposition 5.12 in \cite{Willem},  there exists $\alpha_n\in \R$
such that
\[
\|I'_{\mu}(u_n)-\alpha_n\Theta'(u_n)\|\rightarrow 0,~~\mbox{as}~~
n\rightarrow\infty.
\]
 Hence, we have that
\begin{equation}\label{e4.10}
(-\Delta)^su_n+\phi_{u_n}^tu_n-\mu|u_n|^{q-2}u_n-|u_n|^{2^*_s-2}u_n=\alpha_n
u_n+o_n(1) ~~\mbox{in}~H^{-s}_{rad}(\R^3),
\end{equation}
where $H^{-s}_{rad}(\R^3)$ is the dual space of $H^s_{rad}(\R^3)$.
Thus,   we have   for $\varphi \in H^s_{rad}(\R^3)$,   that
\begin{equation}\label{e4.11}\begin{split}
&\int_{\R^3}(-\Delta)^\frac{s}{2}u_n(-\Delta)^\frac{s}{2}\varphi dx+
\int_{\R^3}\phi_{u_n}^t u_n\varphi dx-\mu\int_{\R^3}|u_n|^{q-2}u_n\varphi dx-\int_{\R^3}|u_n|^{2^*_s-2}u_n\varphi dx\\
&\hspace{1.5cm} =\alpha_n \int_{\R^3}u_n \varphi
 dx+o_n(1),
\end{split}
\end{equation}
and if we choose $\varphi=u_n$, we get \be\label{e4.12} \|(-\Delta
)^\frac{s}{2} u_n\|^2_2+\lambda\int_{\R^3}\phi_{u_n}^t
u_n^2dx-\mu\int_{\R^3}|u_n|^{q}dx-\int_{\R^3}|u_n|^{2^*_s}dx=\alpha_n
 \int_{\R^3}u_n^2dx+o_n(1). \ee
 From  \eqref{e4.12},  and the boundedness
of $\{u_n\}$ in $D^{s,2}(\R^3)$, we can deduce that $\{\alpha_n\}$
is bounded in $\R$.  Then we can assume that, up to a subsequence,
$\alpha_n \rightarrow \alpha $ for some $\alpha\in \R.$ Then, by
\eqref{e4.11}, we can derive that $u$ solves the following equation
\begin{equation}\label{e4.13}
(-\Delta)^su+\phi_{u}u-\mu|u|^{q-2}u-|u|^{2^*_s-2}u=\alpha u.
\end{equation}
Indeed, for any $\varphi \in H^s_{rad}(\R^3)$, it follows by  $u_n
\rightharpoonup u$ in $ H^s_{rad}(\R^3)$ and  $\alpha_n \rightarrow
\alpha$,  that
\[
\int_{\R^3}(-\Delta)^\frac{s}{2}u_n(-\Delta)^\frac{s}{2}\varphi
dx\rightarrow
\int_{\R^3}(-\Delta)^\frac{s}{2}u(-\Delta)^\frac{s}{2}\varphi
dx;~~\mbox{and}~~~\alpha_n \int_{\R^3}u_n \varphi dx\rightarrow
\alpha\int_{\R^3}u \varphi dx.
\]as $ n\rightarrow\infty.$
  Since $\{|u_n|^{2^*_s-2}u_n\}$ is bounded in
 $L^\frac{2^*_s}{2^*_s-1}(\R^3), \{|u_n|^{q-2}u_n\}$ is bounded in
 $L^\frac{2^*_s}{q-1}(\R^3)$,  and $u_n(x)\rightarrow u(x)$
 a.e. on $\R^3$, we obtain that
\[|u_n|^{2^*_s-2}u_n\rightharpoonup |u|^{2^*_s-2}u~~\mbox{
in}~~L^\frac{2^*_s}{2^*_s-1}(\R^3),~~~\mbox{and}~~~|u_n|^{q-2}u_n\rightharpoonup
|u|^{q-2}u~~\mbox{ in}~~L^\frac{2^*_s}{2^*_s-q+1}(\R^3),\] and so,
\[
\int_{\R^3}|u_n|^{2^*_s-2}u_n\varphi dx\rightarrow
\int_{\R^3}|u|^{2^*_s-2}u\varphi
dx~~\mbox{and}~~\int_{\R^3}|u_n|^{q-2}u_n\varphi dx\rightarrow
\int_{\R^3}|u|^{q-2}u\varphi dx,\]as $ n\rightarrow\infty.$ Recall
from Lemma \ref{Lemma 3.3} that
\[\int_{\R^3}\phi_{u_n}^tu_n\varphi dx\rightarrow
\int_{\R^3}\phi_{u}u\varphi dx,~~\forall\varphi\in
H^s_{rad}(\R^3).\]Thus, we have
\begin{equation}\label{e4.14}\begin{split}
&\int_{\R^3}(-\Delta)^\frac{s}{2}u(-\Delta)^\frac{s}{2}\varphi dx+
\int_{\R^3}\phi_{u}^t u\varphi dx-\mu\int_{\R^3}|u|^{q-2}u\varphi dx-\int_{\R^3}|u|^{2^*_s-2}u\varphi dx\\
&\hspace{1.5cm} =\alpha \int_{\R^3}u \varphi
 dx.
\end{split}
\end{equation}
Therefore, $u$ solves equation \eqref{e4.13}.
\par In the sequel, by the concentration-compactness principle, we
can prove that
\begin{equation}\label{e4.15}
\int_{\R^3}|u_n|^{2^*_s}dx\rightarrow
\int_{\R^3}|u|^{2^*_s}dx.\end{equation} In fact, since
$\|(-\Delta)^{\frac{s}{2}}u_n\|_2\leq  R_1$ for $n$ large enough, by
 Lemma \ref{Lemma 4.1}, there exist
 two positive measures, $\zeta, \omega \in \mathcal{M}(\R^3)$, such that
\begin{equation}\label{e4.16}
|(-\Delta)^{\frac{s}{2}}u_n|^2\rightharpoonup
\omega,~~|u_n|^{2^*_s}\rightharpoonup \zeta~~ \mbox{in}~
\mathcal{M}(\R^3)
\end{equation}
as $n\rightarrow\infty$. Then,   by Lemma  \ref{Lemma 4.1}, either
$u_n\rightarrow u$ in $L_{loc}^{2^*_s}(\R^3)$ or there exists a (at
most countable) set of distinct points $\{x_j\}_{j\in J}\subset\R^3$
and positive numbers $\{\zeta_j\}_{j\in J}$ such that \[
\zeta=|u|^{2^*_s}+\sum_{j\in J}\zeta_j\delta_{x_j}.\] Moreover,
there exist some at most a countable set $J\subseteq\mathbb{N}$, a
corresponding set of distinct points  $\{x_j\}_{j\in J}\subset
\R^3$, and two sets of  positive numbers  $\{\zeta_j\}_{j\in J}$ and
$\{\omega_j\}_{j\in J}$ such that items \eqref{e4.1}-\eqref{e4.3}
holds. Now, assume that $J\neq \emptyset.$
 We split the proof into three steps.

\par {\em Step 1.} We prove that $\omega_j=\zeta_j$, where
$\omega_j$,  and $\zeta_j$ come from Lemma  \ref{Lemma 4.1}.

Define $\varphi\in C_0^\infty(\R^3 ) $ as a cut-off function with
$\varphi \in [0,1]$, $\varphi \equiv 1$ in $B_{1/2}(0)$, $\varphi
\equiv 0$ in $\R^3\setminus B_1(0)$. For any $\rho> 0$, define
\[
\varphi_\rho(x):=\varphi\left(\frac{x-x_j}{\rho}\right)=\left\{\begin{array}{ll}1,&|x-x_j|\leq\frac{1}{2}\rho,\vspace{0.2cm}\\
0,&|x-x_j|\geq\rho.
\end{array}\right.\]
By the boundedness of $\{u_n\}$ in $H^s_{rad}(\R^3)$, we have  that
$\{\varphi_\rho u_n\}$ is also bounded in $H^s_{rad}(\R^3)$. Thus,
one has that
\begin{equation}\label{e4.17}
\begin{split}
o_n(1)=&\langle I_{\mu}'(u_n), u_n\varphi_\rho\rangle\\
=&\int_{\R^3}(-\Delta)^\frac{s}{2}u_n(-\Delta)^\frac{s}{2}(
u_n\varphi_\rho)dx+\lambda
\int_{\R^3}\phi_{u_n}^t u_n\varphi_\rho dx-\mu\int_{\R^3} |u_n|^{q}\varphi_\rho dx\\
&-\int_{\R^3}|u_n|^{2^*_s}\varphi_\rho dx.
\end{split}
\end{equation}
 It is easy to check  that
\be\label{e4.18}\begin{split}
&\int_{\R^3}(-\Delta)^\frac{s}{2}u_n(-\Delta)^\frac{s}{2}(u_n\varphi_\rho )dx\\
&=\iint_{\R^{6}}\frac{[u_n(x)-u_n(y)]|u_n(x)-u_n(y)|^2[u_n(x)\varphi_\rho(x) - u_n(y)\varphi_\rho(y)]}{|x-y|^{3+2s}}dxdy\\
&=\iint_{\R^{6}}\frac{|u_n(x)-u_n(y)|^2\varphi_\rho(y)}{|x-y|^{3+2s}}dxdy
+\iint_{\R^{6}}\frac{[u_n(x)-u_n(y)][\varphi_\rho(x) -\varphi_\rho(y)]u_n(x)}{|x-y|^{3+2s}}dxdy\\
&:=T_1+T_2,
\end{split}
\ee where
\[
T_1=\iint_{\R^{6}}\frac{|u_n(x)-u_n(y)|^2\varphi_\rho (y)
}{|x-y|^{3+2s}}dxdy
\]
and
\[
T_2=\iint_{\R^{6}}\frac{[u_n(x)-u_n(y)][\varphi_\rho
(x)-\varphi_\rho (y)]u_n(x)}{|x-y|^{3+2s}}dxdy.
\]
\par For $T_1$, by \eqref{e4.16}, we obtain
\be\label{e4.19}\begin{split}
\lim_{\rho\rightarrow0}\lim_{n\rightarrow\infty}T_1=& \lim_{\rho\rightarrow0}\lim_{n\rightarrow\infty} \iint_{\R^{6}}\frac{|u_n(x)-u_n(y)|^2\varphi_\rho (y) }{|x-y|^{3+2s}}dxdy\\
 =&\lim_{\rho\rightarrow0} \int_{\R^{3}}\varphi_\rho d\omega=\omega(\{x_j\})=\omega_j.
\end{split}
\ee From   H\"{o}lder's inequality, we have
\[\begin{split}
T_2=& \iint_{\R^{6}}\frac{[u_n(x)-u_n(y)][\varphi_\rho (x)-\varphi_\rho (y)]u_n(x)}{|x-y|^{3+2s}}dxdy\\
\leq& \left(\iint_{\R^{6}}\frac{|\varphi_\rho (x)-\varphi_\rho
(y)|^2|u_n(x)|^2}{|x-y|^{3+2s}}dxdy\right)^\frac{1}{2}
\left(\iint_{\R^{6}}\frac{|u_n(x)-u_n(y)|^2}{|x-y|^{3+2s}}dxdy\right)^\frac{1}{2}\\
\leq& C\left(\iint_{\R^{6}}\frac{|\varphi_\rho (x)-\varphi_\rho
(y)|^2|u_n(x)|^2}{|x-y|^{3+2s}}dxdy\right)^\frac{1}{2}.
\end{split}
\]
Analogously to the proof of Lemma 3.4 in \cite{ZZR}, we obtain
\[
\lim_{\rho\rightarrow0}\lim_{n\rightarrow\infty}\iint_{\R^{6}}\frac{|\varphi_\rho
(x)-\varphi_\rho (y)|^2|u_n(x)|^2}{|x-y|^{3+2s}}dxdy=0,
\]
and
\[
\lim_{\rho\rightarrow0}\lim_{n\rightarrow\infty}\int_{\R^3}(-\Delta)^{\frac{s}{2}}u_n
(-\Delta)^{\frac{s}{2}}(u_n\varphi_\rho )
dx=\omega(\{x_j\})=\omega_j.
\]
Again   by \eqref{e4.16}, we have
\begin{equation} \label{e4.20}
\lim_{\rho\rightarrow0}\lim_{n\rightarrow\infty}\int_{\R^3}
|u_n|^{2^*_s}\varphi_\rho
 dx=\lim_{\rho\rightarrow0}\int_{\R^3}\varphi_\rho
d\zeta=\zeta(\{x_j\})=\zeta_j.
\end{equation}
By   the definition of $\varphi_\rho$, and the absolute continuity
of the Lebesgue integral, one has that
\begin{equation} \label{e4.21}
\lim_{\rho\rightarrow0}\lim_{n\rightarrow\infty}\int_{\R^3}
|u_n|^{q}\varphi_\rho
dx=\lim_{\rho\rightarrow0}\int_{\R^3}|u|^{q}\varphi_\rho dx
=\lim_{\rho\rightarrow0}\int_{|x-x_j|\leq\rho} |u|^{q}\varphi_\rho
dx=0.
\end{equation}
  Thus, by     Proposition \ref{prop 3.1} and Lemma \ref{Lemma 3.6},  we have
\begin{equation} \label{e4.22}
\begin{split}
\int_{\R^3}\phi_{u_n}^t u_n^2\varphi_\rho dx&\leq
C\left(\int_{\R^3}|u_n|^{\frac{12}{3+2t}}dx\right)^{\frac{3+2t}{6}}\left(\int_{\R^3}|u_n^2\varphi_\rho|^{\frac{6}{3+2t}}dx\right)^{\frac{3+2t}{6}}\\
&\leq
 C\|u_n\|^2_H\left(\int_{\R^3}|u_n|^{\frac{12}{3+2t}}|\varphi_\rho|^{\frac{6}{3+2t}}dx\right)^{\frac{3+2t}{6}}\\
 &\leq C_1
 \left(\int_{\R^3}|u_n|^{\frac{12}{3+2t}}\varphi_\rho dx\right)^{\frac{3+2t}{6}}.
\end{split}\end{equation}
Therefore, \begin{equation} \label{e4.23}
\begin{split}
\lim_{\rho\rightarrow0}\lim_{n\rightarrow\infty}\int_{\R^3}\phi_{u_n}^t
u_n^2\varphi_\rho dx&\leq
\lim_{\rho\rightarrow0}\lim_{n\rightarrow\infty} C_1
 \left(\int_{\R^3}|u_n|^{\frac{12}{3+2t}}\varphi_\rho
 dx\right)^{\frac{3+2t}{6}}\\
&= \lim_{\rho\rightarrow0} C_1
 \left(\int_{\R^3}|u|^{\frac{12}{3+2t}}\varphi_\rho
 dx\right)^{\frac{3+2t}{6}}\\
 &=\lim_{\rho\rightarrow0} C_1
 \left(\int_{_{|x-x_j|\leq\rho}}|u|^{\frac{12}{3+2t}}\varphi_\rho
 dx\right)^{\frac{3+2t}{6}}=0.\end{split}
\end{equation}
Summing up, from \eqref{e4.17}-\eqref{e4.19} and \eqref{e4.21},
taking the limit as $n\rightarrow\infty$, and then the limit as
$\rho\rightarrow0$, we arrive at
\[
\omega_j=\zeta_j.
\]

\par {\em Step 2.}  We show that
$\omega_\infty=\zeta_\infty$, where $\omega_\infty$ and
 $\zeta_\infty$  are given in Lemma  \ref{Lemma 4.2}. Let $\psi\in C_0^\infty(\R^3 ) $ be a cut-off function with
$\psi \in [0,1]$, $\psi \equiv 0$ in $B_{1/2}(0)$, $\psi \equiv 1$
in $\R^3\setminus B_1(0)$. For any $R> 0$, define
\[
\psi_R(x):=\psi\left(\frac{x}{R}\right)=\left\{\begin{array}{ll}0,&|x|\leq\frac{1}{2}R,\vspace{0.2cm}\\
1,&|x|\geq R.
\end{array}\right.\]
Using  again  the boundedness of $\{u_n\}$ and $\{ u_n\psi_R\}$ in
$H^s_{rad}(\R^3)$,   we have
\begin{equation}\label{e4.24}
\begin{split}
o_n(1)=&\langle I_{\mu}'(u_n), u_n\psi_R\rangle\\
=&\int_{\R^3}(-\Delta)^\frac{s}{2}u_n(-\Delta)^\frac{s}{2}(u_n\psi_R)dx+\lambda\int_{\R^3}\phi_{u_n}^t u_n^2\psi_R dx-\mu\int_{\R^3} |u_n|^{q}\psi_R dx\\
&-\int_{\R^3}|u_n|^{2^*_s}\psi_R dx.
\end{split}
\end{equation}
It is easy to derive  that
\[\begin{split}
&\int_{\R^3}(-\Delta)^{\frac{s}{2}}u_n (-\Delta)^{\frac{s}{2}}(u_n\psi_R ) dx\\
&=\iint_{\R^{6}}\frac{[u_n(x)-u_n(y)][u_n(x)\psi_R(x)-u_n(y)\psi_R (y)]}{|x-y|^{3+2s}}dxdy\\
&= \iint_{\R^{6}}\frac{|u_n(x)-u_n(y)|^2\psi_R (y)
}{|x-y|^{3+2s}}dxdy+
\iint_{\R^{6}}\frac{[u_n(x)-u_n(y)][\psi_R  (x)-\psi_R (y)]u_n(x)}{|x-y|^{3+2s}}dxdy\\
& :=T_3+T_4,
\end{split}
\]
where
\[
T_3=\iint_{\R^{6}}\frac{|u_n(x)-u_n(y)|^2\psi_R  (y)
}{|x-y|^{3+2s}}dxdy
\]
and
\[
T_4=\iint_{\R^{6}}\frac{[u_n(x)-u_n(y)][\psi_R  (x)-\psi_R
(y)]u_n(x)}{|x-y|^{3+2s}}dxdy.
\]
For $T_3$, by \eqref{e4.16} and Lemma \ref{Lemma 4.2}, we  infer to
\[\begin{split}
\lim_{R\rightarrow\infty}\lim_{n\rightarrow\infty}T_3=
\lim_{R\rightarrow\infty}\lim_{n\rightarrow\infty}
\iint_{\R^{6}}\frac{|u_n(x)-u_n(y)|^2\psi_R  (y)
}{|x-y|^{3+2s}}dxdy=\omega_\infty.\end{split}
\]
By virtue of    H\"{o}lder's inequality, we get
\[\begin{split}
T_4=& \iint_{\R^{6}}\frac{[u_n(x)-u_n(y)][\psi_R (x)-\psi_R  (y)]u_n(x)}{|x-y|^{3+2s}}dxdy\\
\leq& \left(\iint_{\R^{6}}\frac{|\psi_R  (x)-\psi_R
(y)|^2|u_n(x)|^2}{|x-y|^{3+2s}}dxdy\right)^\frac{1}{2}
\left(\iint_{\R^{6}}\frac{|u_n(x)-u_n(y)|^2}{|x-y|^{3+2s}}dxdy\right)^\frac{1}{2}\\
\leq& C\left(\iint_{\R^{6}}\frac{|\psi_R  (x)-\psi_R
(y)|^2|u_n(x)|^2}{|x-y|^{3+2s}}dxdy\right)^\frac{1}{2}.
\end{split}
\]
Combining the above proof, we conclude that
\[\begin{split}
&\lim_{R\rightarrow\infty}\lim_{n\rightarrow\infty}\iint_{\R^{6}}\frac{|\psi_R (x)-\psi_R (y)|^2|u_n(x)|^2}{|x-y|^{3+2s}}dxdy\\
&=\lim_{R\rightarrow\infty}\lim_{n\rightarrow\infty}\iint_{\R^{6}}\frac{|[1-\psi_R
(x)]-[1-\psi_R (y)]|^2|u_n(x)|^2}{|x-y|^{3+2s}}dxdy=0.
\end{split}
\]
Hence,
\[
\lim_{R\rightarrow\infty}\lim_{n\rightarrow\infty}\iint_{\R^{6}}(-\Delta)^{\frac{s}{2}}u_n
(-\Delta)^{\frac{s}{2}}(u_n\psi_R) dx=\omega_\infty.
\]
By Lemma \ref{Lemma 4.2}, we  have
\begin{equation}\label{e4.25}
\lim_{R\rightarrow\infty}\lim_{n\rightarrow\infty}\int_{\R^3}
|u_n|^{2^*_s}\psi_R  dx=\zeta_\infty.
\end{equation}
 Analogous  the proof of Lemma
3.3 in \cite{ZZR}, we infer to
\begin{equation}\label{e4.26}
\lim_{R\rightarrow\infty}\lim_{n\rightarrow\infty}\int_{\R^3}
|u_n|^{q}\psi_R dx=\lim_{R\rightarrow\infty}\int_{\R^3}
|u|^{q}\psi_R dx=\lim_{R\rightarrow\infty}\int_{|x|>\frac{1}{2}R}
|u|^{q}\psi_R dx=0.
\end{equation}
Moreover, we can obtain\begin{equation} \label{e4.27}
\begin{split}
\lim_{R\rightarrow\infty}\lim_{n\rightarrow\infty}\int_{\R^3}\phi_{u_n}^t
u_n^2\psi_R dx&\leq
\lim_{R\rightarrow\infty}\lim_{n\rightarrow\infty}C_1
 \left(\int_{\R^3}|u_n|^{\frac{12}{3+2t}}\psi_R
 dx\right)^{\frac{3+2t}{6}}\\
&= \lim_{R\rightarrow\infty} C_1
 \left(\int_{\R^3}|u|^{\frac{12}{3+2t}}\psi_R
 dx\right)^{\frac{3+2t}{6}}\\
 &=\lim_{R\rightarrow\infty} C_1
 \left(\int_{_{|x|\geq R/2}}|u|^{\frac{12}{3+2t}}\psi_R
 dx\right)^{\frac{3+2t}{6}}=0.\end{split}
\end{equation}
Summing up, from \eqref{e4.24}-\eqref{e4.27}, taking the limit as
$n\rightarrow\infty$, and then the limit as $R\rightarrow\infty$, we
have
\[
\omega_\infty=\zeta_\infty.\]

\par {\em Step 3.}  We claim that $\zeta_j=0$ for any $j\in J$
and $\zeta_\infty=0.$
\par  Suppose  by contradiction that, there exists $j_0\in J$ such that
$\zeta_{j_0}>0$ or $\zeta_\infty>0$. Step 1, Step 2, and Lemmas
 \ref{Lemma 4.1}, \ref{Lemma 4.2} imply that
\begin{equation}\label{e4.28}
 \zeta_{j_0}\leq (S^{-1}\omega_{j_0})^\frac{ 2^*_s}{2}=(S^{-1} \zeta_{j_0})^\frac{ 2^*_s}{2},
\end{equation}
and
\begin{equation}\label{e4.29}
\zeta_\infty=(S^{-1}\omega_\infty)^\frac{2^*_s}{2}=(S^{-1}\zeta_\infty)^\frac{2^*_s}{2}.
\end{equation}
Consequently, we get $ \zeta_{j_0}\geq S^{\frac{3}{2s}}$ or
$\zeta_{\infty}\geq S^{\frac{3}{2s}}$. If the former case occurs,
 we have
\be\label{e4.30}\begin{split}
 R_1^2\geq \lim_{n\rightarrow\infty}\|(-\Delta )^\frac{s}{2}
 u_n\|^2_2
 &\geq S \lim_{n\rightarrow\infty}\left(\int_{\R^3}|u_n|^{2^*_s}dx\right)^{\frac{2}{2^*_s}}\\
  &\geq S \lim_{n\rightarrow\infty}\left(\int_{\R^3}|u_n|^{2^*_s}\varphi_\rho dx\right)^{\frac{2}{2^*_s}} =S\left(\int_{\R^3}\varphi_\rho
d\zeta\right)^{\frac{2}{2^*_s}}.\end{split}\ee
 Taking the limit $ \rho\rightarrow0$ in the last inequality, we get
    \[R_1^2\geq S (\zeta_{j_0})^{\frac{2}{2^*_s}}\geq S (S^{\frac{3}{2s}})^{\frac{2}{2^*_s}}=S^{\frac{3}{2s}},\]
 which contradicts \eqref{e4.9}. If the last case happens,   we
 have
\be\label{e4.31}\begin{split}
 R_1^2\geq \lim_{n\rightarrow\infty}\|(-\Delta )^\frac{s}{2}
 u_n\|^2_2
 &\geq S \lim_{n\rightarrow\infty}\left(\int_{\R^3}|u_n|^{2^*_s}dx\right)^{\frac{2}{2^*_s}}\\
  &\geq S \lim_{n\rightarrow\infty}\left(\int_{\R^3}|u_n|^{2^*_s}\psi_R
  dx\right)^{\frac{2}{2^*_s}}\\
 & \geq S \lim_{n\rightarrow\infty}\left(\int_{|x|\geq R}|u_n|^{2^*_s} dx\right)^{\frac{2}{2^*_s}}.\end{split}\ee
 Taking the limits $n\rightarrow\infty$ and $ R\rightarrow\infty$ in \eqref{e4.31}, we infer
 to
    \[R_1^2\geq S (\zeta_{\infty})^{\frac{2}{2^*_s}}\geq S (S^{\frac{3}{2s}})^{\frac{2}{2^*_s}}=S^{\frac{3}{2s}},\]
 which also contradicts \eqref{e4.9}. Therefore,
$\zeta_j=0$ for any $j\in J$ and $\zeta_\infty=0$. As a result,
  by    Lemma \ref{Lemma 4.1},  we obtain that $u_n\rightarrow u$ in $L_{loc}^{2^*_s}(\R^3)$;
   while by Lemma    \ref{Lemma 4.2}, we know that $u_n\rightarrow u$ in
   $L^{2^*_s}(\R^3)$.

\par Now, we prove  there exists $\mu^*_1>0$ independently on
$n\in\mathbb{N}$ such that if $\mu>\mu^*_1$, the Lagrange multiplier
$\alpha<0$ in \eqref{e4.13}. Indeed, note that $\{u_n\}\subset
S_{r,s}$ and $\|(-\Delta)^{\frac{s}{2}}u_n\|_2\leq R_1$, as can be
seen from the previous proof of this lemma, and
\eqref{e3.2}-\eqref{e3.3} that, there exists $Q_1>0$ independently
on $n$, such that
 \be \label{e4.32}\begin{split}
Q_1\leq \int_{\R^3}|u_n|^qdx&\leq
{C}(q,s) \|(-\Delta)^{\frac{s}{2}}u_n\|_2^{q\delta_{q,s}}\|u_n\|_2^{q(1-\delta_{q,s})}\\
&\leq
 {C}(q,s)R_1^{q\delta_{q,s}}a^{q(1-\delta_{q,s})},
\end{split}
\ee and
 \be \label{e4.33}\begin{split}
\int_{\R^3}\phi_{u_n}^tu_n^2dx\leq
\Gamma_t\|u_n\|_{\frac{12}{3+2t}}^4 &\leq
\Gamma_tC\left({12}/{3+2t},s\right)^{\frac{3+2t}{3}}\|(-\Delta)^{\frac{s}{2}}u_n\|_2^{\frac{3-2t}{s}}\|u_n\|_2^{\frac{2t+4s-3}{s}}\\
&\leq\Gamma_tC\left({12}/{3+2t},s\right)^{\frac{3+2t}{3}}R_1^{\frac{3-2t}{s}}a^{\frac{2t+4s-3}{s}}\\
&:=Q_2,
\end{split}
\ee where $Q_2=Q_2(s,t,R_1,a)>0.$ We define the constant
\be\label{e4.34}
\mu^*_1:=\frac{q\lambda(2t+4s-3)Q_2}{2[6-q(3-2s)]Q_1}. \ee By
\eqref{e4.32}-\eqref{e4.34}  we have
\be\label{e4.35}\mu^*_1>\lim_{n\rightarrow+\infty}\left\{\frac{q\lambda(2t+4s-3)\int_{\R^3}\phi_{u_n}^tu_n^2dx}{2[6-q(3-2s)]\int_{\R^3}|u_n|^qdx}\right\}=\frac{q\lambda(2t+4s-3)
\int_{\R^3}\phi_{u}^tu^2dx}{2[6-q(3-2s)]\int_{\R^3}|u|^qdx}>0.\ee
   Recall  by \eqref{e4.13} and its Pohozaev identity $P_{\mu}(u)=0$, we infer to
\be\label{e4.36} s\alpha
\|u\|_2^2=\lambda\frac{2t+4s-3}{4}\int_{\R^3}\phi_{u}^tu^2dx+\frac{q(3-2s)-6}{2q}\mu\int_{\R^3}|u|^qdx.
\ee
 Now, if $\mu>\mu^*_1,$ we conclude from \eqref{e4.35}, that
 \[\mu>\frac{q\lambda(2t+4s-3)\int_{\R^3}\phi_{u}^tu^2dx}{2[6-q(3-2s)]\int_{\R^3}|u|^qdx}.\]
Thus, from \eqref{e4.36}, we infer to
$\lim_{n\rightarrow+\infty}\alpha_n=\alpha<0.$  Hence, taking into
account  \eqref{e4.12}, we
   derive
\begin{equation}\label{e4.37}\begin{split}
&\lim_{n\rightarrow\infty}\left[\|(-\Delta )^\frac{s}{2} u_n\|^2_2+
 \lambda\int_{\R^3}\phi^t_{u_n}u_n^2 dx-\alpha\|u_n\|_2^2\right]\\
 &=\lim_{n\rightarrow\infty}\left[\mu\|u_n\|_q^q+\int_{\R^3}|u_n|^{2^*_s} dx+o_n(1)\right]\\
& = \mu\|u\|_q^q+ \int_{\R^3}|u|^{2^*_s}  dx=
\|(-\Delta)^\frac{s}{2} u\|^2_2+\lambda\int_{\R^3}\phi _{u}u^2
dx-\alpha\|u\|_2^2.
\end{split}
\end{equation}
Since $\alpha<0$ for $\mu>\mu^*_1$ large, we obtain by Fatou's
Lemma,
\begin{equation}\label{e4.38}\begin{split}
&\lim_{n\rightarrow\infty}\left[\|(-\Delta )^\frac{s}{2} u_n\|^2_2+
 \lambda\int_{\R^3}\phi^t_{u_n}u_n^2 dx-\alpha\|u_n\|_2^2\right]\\
 &\geq  \|(-\Delta )^\frac{s}{2} u\|^2_2+
 \lambda\int_{\R^3}\phi^t_{u}u^2 dx  +\liminf_{n\rightarrow\infty}(-\alpha\|u_n\|_2^2),\end{split}\end{equation}
and from \eqref{e4.37}-\eqref{e4.38}, one has
\begin{equation}\label{e4.39}-\alpha\|u\|_2^2\geq
\liminf_{n\rightarrow\infty}(-\alpha\|u_n\|_2^2).\end{equation} But
by Fatou's Lemma, we see that
\begin{equation}\label{e4.40}  \liminf_{n\rightarrow\infty}(-\alpha\|u_n\|_2^2)\geq -\alpha\|u\|_2^2.\end{equation}
Combining \eqref{e4.39} with \eqref{e4.40} we get
\[
\lim_{n\rightarrow\infty}\left(-\alpha\|u_n\|_2^2\right)=-\alpha\|u\|_2^2;
\]
that is,
\[
\lim_{n\rightarrow\infty}\|u_n\|_2^2=\|u\|_2^2.
\]
Thus, by \eqref{e4.37} we have
\[
\lim_{n\rightarrow\infty}\|(-\Delta )^\frac{s}{2}
u_n\|^2_2=\|(-\Delta )^\frac{s}{2} u\|^2_2.
\]
Theerfore, $u_n\rightarrow u$ in $H^s_{rad}(\R^3)$ and $\|u\|_2=a$.
The proof is complete.
\end{proof}

For $\varepsilon>0$, we introduce the set
\[
I^{-\varepsilon}_{\mu,\tau}=\left\{u\in H^s_{rad}(\R^3)\cap S_a:
 I_{\mu,\tau}(u)\leq-\varepsilon\right\}\subset H^s_{rad}(\R^3).
\]
By the fact that $I_{\mu,\tau}(u)$ is continuous and even on
$H^s_{rad}(\R^3)$, $I^{-\varepsilon}_{\mu,\tau}$ is closed and
symmetric.
\begin{lemma}\label{Lemma 4.4}
For any fixed $k \in \mathbb{N}$ , there exists
$\varepsilon_k:=\varepsilon(k)>0$ and
 $\mu_k:=\mu(k)>0$ such that, for $0<\varepsilon\leq\varepsilon_k$ and $\mu\geq \mu_k$, one has  that $\gamma(I^{-\varepsilon}_{\mu,\tau}) \geq k$.
 \end{lemma} The proof of Lemma \ref{Lemma 4.4}  is
similar to Lemma 3.2 in \cite{AJM}, so we omit it here.

\par In the sequel,  we define the set
\[
\Sigma_k:=\left\{\Omega\subset H^s_{rad}(\R^3)\cap S_a: \Omega\mbox{
is closed and symmetric}, \gamma(\Omega)\geq k\right\},
\]
and by  Lemma \ref{Lemma 4.3}-(ii), we know that
\[
c_k:=\inf_{\Omega\in \Sigma_k}\sup_{u\in
\Omega}I_{\mu,\tau}(u)>-\infty
\]
for all $k\in \mathbb{N}$. To prove Theorem \ref{Theorem 2.1},  we
introduce the critical value, we define
\[
K_c:=\{u\in H^s_{rad}(\R^3)\cap S_a: I'_{\mu,\tau}(u)=0,
I_{\mu,\tau}(u)=c\}.
\]
Then, we can  derive  the following conclusion:
\begin{lemma}\label{Lemma 4.5}
If $c = c_k = c_{k+1} = \cdot\cdot\cdot = c_{k+\ell}$, then  one has
$\gamma(K_c)\geq \ell +1$. Especially, $I_{\mu,\tau}(u)$ admits at
least $\ell+1$ nontrivial critical points.\end{lemma}
\begin{proof}
 For $\varepsilon>0$, it is easy to check that
$I^{-\varepsilon}_{\mu,\tau}\in \Sigma$. For any fixed $k\in
\mathbb{N}$, by Lemma  \ref{Lemma 4.4}, there exists
$\varepsilon_k:=\varepsilon(k)>0$ and $\mu_k:=\mu(k)>0$ such that,
if $0<\varepsilon\leq\varepsilon_k$ and $\mu\geq \mu_k$, we
 have $\gamma(I^{-\varepsilon_k}_{\mu,\tau}) \geq k.$ Thus, $I^{-\varepsilon_k}_{\mu,\tau} \in \Sigma_k$,
 and moreover,
\[
c_k\leq\sup_{u\in
 I^{-\varepsilon_k}_{\mu,\tau}}I_{\mu,\tau}(u)=-\varepsilon_k<0.
\]
Assume that $0> c = c_k = c_{k+1} = \cdot\cdot\cdot = c_{k+\ell}$.
Then, by Lemma \ref{Lemma 4.3}-(iii), $I_{\mu,\tau}(u)$ satisfies
the $(PS)_c$-condition at the level $c < 0$. So, $K_c$ is a compact
set. By Theorem 2.1 in \cite{AJM}, or Theorem 2.1 in \cite{JLu}, we
know that the restricted functional  $I_{\mu,\tau}|_{S_a}$ possesses
at least $\ell+1$ nontrivial critical points.
\end{proof}
 \noindent   {\em Proof of Theorem \ref{Theorem 2.1}.} Let $\mu\geq\mu^*_k=\max\{\mu^*_1,\mu_k\}$. From Lemma \ref{Lemma 4.3}-(ii), we see that the critical points of $I_{\mu,\tau}(u)$ found in
Lemma  \ref{Lemma 4.5} are the critical points of $I_{\mu}$, which
completes the proof. \qed

\vskip0.2in

  \s{ Proof of Theorem \ref{Theorem 2.2}}

\par From   Lemma \ref{Lemma 3.5},  we see that any critical point   of $I_{\mu}|_{S_a}$ belongs to $\mathcal {P}_{a}$.
Consequently, the properties of the manifold $\mathcal {P}_{a}$ have
relation to the mini-max structure of $I_{\mu}|_{S_a}$.  For $u\in
S_a$ and $t\in\R$, we introduce the transformation (e.g.
\cite{Soave1}): \be\label{e5.1}
  (\theta\star u)(x):=e^{\frac{3\theta}{2}}u(e^{\theta}x),~~~~
x\in\R^3,~~\theta\in\R.\ee It is easy to check that the dilations
preserve the $L^2$-norm such that $\theta\star u\in S_a$, by direct
calculation, one has \be\label{e5.2}\begin{split}
 I(u,\theta)=I_{\mu} ((\theta\star u))=&\frac{e^{2s\theta}}{2}\|u\|^2+\frac{\lambda e^{(3-2t)\theta}}{4}\int_{\R^3}\phi^t_uu^2dx-\frac{\mu}{q}e^{(\frac{3q}{2}-3)\theta}\int_{\R^3}|u|^qdx\\
&-\frac{1}{2^*_{s}}e^{3(\frac{2^*_{s}}{2}-1)\theta}\int_{\R^3}|u|^{2^*_{s}}dx,
\end{split}\ee

\bl \label{Lemma 5.1}  Let $u\in S_a,$ then
\begin{itemize}
\item[(i)]$\|(-\Delta)^{\frac{s}{2}}(\theta\star u)\|_2\rightarrow0$  and $I_{\mu} ((\theta\star u))\rightarrow0$ as $\theta\rightarrow-\infty$;
\item[(ii)] $\|(-\Delta)^{\frac{s}{2}}(\theta\star u)\|_2\rightarrow+\infty$ and
$I_{\mu} ((\theta\star u))\rightarrow-\infty$ as $\theta\rightarrow
+\infty$.
\end{itemize}
\el \bp  A direct computation shows that \be\label{e5.3}
\int_{\R^3}|(-\Delta)^{\frac{s}{2}}(\theta\star
u)|^2dx=e^{2s\theta}\int_{\R^3}|(-\Delta)^{\frac{s}{2}} u|^2dx,\ee
and
\[\|(-\Delta)^{\frac{s}{2}}(\theta\star u)\|_2\rightarrow
0~~~\mbox{as}~~\theta\rightarrow-\infty.\] Notice  that
\be\label{e5.4}\begin{split}
 I_{\mu}  ((\theta\star u))=&\frac{e^{2s\theta}}{2}\|u\|^2+\frac{\lambda e^{(3-2t)\theta}}{4}\int_{\R^3}\phi^t_uu^2dx-\frac{\mu}{q}e^{(\frac{3q}{2}-3)\theta}\int_{\R^3}|u|^qdx\\
&-\frac{1}{2^*_{s}}e^{\frac{3(2^*_{s}-2)}{2}\theta}\int_{\R^3}|u|^{2^*_{s}}dx,
\end{split}\ee
by    $q>2$, we infer to  \[I_{\mu}  ((\theta\star
u))\rightarrow-\infty,~~~\mbox{as}~~\theta\rightarrow+\infty.\]Hence,
item (i) follows. Using $2s + 2t > 3$,  it is easy to obtain that
$\frac{3(2^*_{s}-2)}{2}>3-2t$, and conclusion (ii) holds. \ep

\bl\label{Lemma 5.2} There exist  $K=K_a> 0 $ and $\widetilde{a} >
0$ such that for all $0 < a < \widetilde{a },$ \be\label{e5.5} 0 <
\sup_{u\in \mathcal {A}_a} I_{\mu} (u) < \inf_{u\in \mathcal
{B}_a}I_{\mu}(u), \ee where $\mathcal {A}_a:= \{u\in S_{r,a} :
\int_{\R^3} |(-\Delta)^{\frac{s}{2}}u|^2dx\leq K_a\}, ~~\mathcal
{B}_a:= \{u\in S_{r,a}  : \int_{\R^3} |(-\Delta)^{\frac{s}{2}}u|^2dx
=2 K_a\}.$ \el

\bp By Lemma \ref{Lemma 3.2}, we have for any $q\in(2,2^*_s)$, that
\be\label{e5.6}\|u\|_q^q\leq
C(q,s)\|(-\Delta)^{\frac{s}{2}}u\|_2^{q\delta_{q,s}}\|u\|_2^{q(1-\delta_{q,s})}.\ee
By the Sobolev inequality \eqref{e3.1}, and  \eqref{e5.6}, for $u
\in S_{r,a} $, we have \be\label{e5.7}\begin{split}
 &I_{\mu}((\theta\star u))-I_{\mu}(u)\\
 &=\frac{1}{2}\|(\theta\star u)\|^2-\frac{1}{2}\| u\|^2+\frac{\lambda}{4}\int_{\R^3}\phi^t_{(\theta\star u)}|(\theta\star u)|^2dx
 -\frac{\lambda}{4}\int_{\R^3}\phi^t_uu^2dx\\
 &\hspace{0.45cm}-\frac{\mu}{q}\int_{\R^3}|(\theta\star u)|^qdx+\frac{\mu}{q}\int_{\R^3}|u|^qdx-\frac{1}{2^*_{s}}\int_{\R^3}|(\theta\star u)|^{2^*_{s}}dx+
 \frac{1}{2^*_{s}}\int_{\R^3}|u|^{2^*_{s}}dx\\
 &\geq \frac{1}{2}\|(\theta\star u)\|^2-\frac{1}{2}\|u\|^2
 -\lambda \Gamma_t K_a^{\frac{3-2t}{2s}}\|u\|_2^{\frac{4s+2t-3}{s}}-\frac{\mu}{q}\int_{\R^3}|(\theta\star u)|^qdx-\frac{1}{2^*_{s}}\int_{\R^3}|(\theta\star u)|^{2^*_{s}}dx\\
 &\geq \frac{1}{2}\|(\theta\star u)\|^2-\frac{1}{2}\|u\|^2
 -\lambda\Gamma_t K_a^{\frac{3-2t}{2s}}a^{\frac{4s+2t-3}{s}}-\frac{\mu}{q}C(q,s)a^{\frac{6-q(3-2s)}{2s}}\left(\|(\theta\star u)\|^2\right)^{\frac{q\delta_{q,s}}{2}}\\
 &\hspace{0.45cm} -\frac{S^{-\frac{2^*_s}{2}}}{2^*_{s}}\left(\|(\theta\star u)\|^2\right)^{\frac{2^*_s}{2}}.
 \end{split}
 \ee
Let $\|u\|^2\leq K_a$ and choose $\theta>0$ such that
$\|(\theta\star u)\|^2 =2 K_a$, here $K_a$ will be determined later,
set\[ \widetilde{a}=\left(\frac{K_a^{\frac{2t+2s-3}{2s}}}{16\lambda
 \Gamma_t}\right)^{\frac{s}{4s+2t-3}},\] then we get
\be\label{e5.8}\begin{split}
 &I_{\mu}((\theta\star u))-I_{\mu}(u)\\
 &\geq \frac{1}{2}K_a
 -\lambda \Gamma_t K_a^{\frac{3-2t}{2s}}\widetilde{a}^{\frac{4s+2t-3}{s}}-\frac{\mu}{q}2^{\frac{q\delta_{q,s}}{2}}C(q,s)\widetilde{a}^{\frac{6-q(3-2s)}{2s}}K_a^{\frac{3(q-2)}{4s}}
 -\frac{S^{-\frac{2^*_s}{2}}}{2^*_{s}}2^{\frac{2^*_s}{2}}K_a^{\frac{2^*_s}{2}}\\
 &\geq
 \frac{1}{2}K_a-\frac{1}{16}K_a-\frac{\mu}{q}2^{\frac{3(q-2)}{4s}}C(q,s)\left(\frac{1}{16\lambda{\Gamma_t}}\right)^{\frac{6-q(3-2s)}{2(4s+2t-3)}}
 K_a^{\frac{[6-q(3-2s)][2t+2s-3]}{4s(4s+2t-3)}}
K_a^{\frac{3(q-2)}{4s}}\\
&\hspace{0.45cm}-\frac{S^{-\frac{2^*_s}{2}}}{2^*_{s}}2^{\frac{2^*_s}{2}}K_a^{\frac{2^*_s}{2}}\\
&= \frac{7}{16}K_a-\frac{\mu
 2^{\frac{3(q-2)}{4s}}C(q,s)}{q(16\lambda{\Gamma_t})^{\frac{6-q(3-2s)}{2(4s+2t-3)}}} K_a^{\gamma_1}K_a-\frac{2^{\frac{2^*_s}{2}}}{2^*_sS^{\frac{2^*_s}{2}}}K_a^{\frac{2^*_s-2}{2}}K_a\\
 &\geq\frac{5}{16}K_a>0,
 \end{split}
 \ee where
 $\gamma_1:=\frac{[2t+2s-3][6-q(3-2s)]+[3(q-2)-4s][4s+2t-3]}{4s(4s+2t-3)}$.
If we take
\[ K_a=\min\left\{\left(\frac{q[16\lambda {\Gamma_t}]^{\frac{6-q(3-2s)}{2(4s+2t-3)}}}{16\mu
 2^{\frac{3(q-2)}{4s}}C(q,s)}\right)^{\gamma_2}, ~~\left(\frac
 {2^*_s S^{\frac{2^*_s}{2}}}{2^{\frac{2^*_s}{2}}16}\right)^{\frac{2}{2^*_s-2}}\right\}\]
with
$\gamma_2:=\frac{4s(4s+2t-3)}{[2t+2s-3][6-q(3-2s)]+[3(q-2)-4s][4s+2t-3]}$,
then, we deduce by \eqref{e5.8} that   \eqref{e5.5} holds.
 \ep
 By Lemma \ref{Lemma 5.2}, we can deduce  the following
\par\noindent {\bf Corollary 5.1.} {\em Let $K_a, \widetilde{a}$ be
given in Lemma \ref{Lemma 5.2}, and $u\in S_{r,a} $ with
$\|u\|^2\leq K_a$, then $I_{\mu}(u)> 0.$ Furthermore, we have
\[L_0:=\inf\left\{I_{\mu}(u):~u\in S_{r,a} , \|u\|^2=\frac{1}{2}
K_a\right\}>0.\]} \bp As in the proof of Lemma \ref{Lemma 5.2}, we
have that
\[
I_{\mu}(u)  \geq \frac{1}{2}\|u\|^2
  -\frac{\mu}{q}C(q,s){a}^{\frac{6-q(3-2s)}{2s}}\left(\| u\|^2\right)^{\frac{3(q-2)}{4s}}
 -\frac{S^{-\frac{2^*_s}{2}}}{2^*_{s}}\left(\|u\|^2\right)^{\frac{2^*_s}{2}}>0,
 \]if $\|u\|^2\leq K_a$, and the conclusion follows. \ep

\par Next, we study the characterizations of the mountain pass levels for
$I(u,\theta)$ and $I_{\mu}(u).$ Denote  the closed set
$I_{\mu}^d:=\{u\in
 S_{r,a}:  I_{\mu}(u) \leq d \},$ and  $ S_{r,a}:= H^s_r (\R^3)\cap S_a.$

\begin{proposition}\label{Prop 5.3} Under assumptions $2+\frac{4s}{3}<q < 2^*_s,$
define
 \[\widetilde{c}_{\mu}(a):=\inf_{\widetilde{\gamma}\in\widetilde{\Gamma}}\max_{t\in[0,1]}I(\widetilde{\gamma}(t)),\]
where
\[\widetilde{\Gamma}_a=\{\widetilde{\gamma}\in C([0,1],
S_{r,a}\times\R):~\widetilde{\gamma}(0)\in (\mathcal {A}_a,0),
\widetilde{\gamma}(1)\in (I_{\mu}^0,0)\},\] and
 \[{c}_{\mu}(a):=\inf_{{\gamma}\in{\Gamma}}\max_{t\in[0,1]}I_{\mu}({\gamma}(t)),\]
where
\[{\Gamma}_a=\{{\gamma}\in C([0,1],
S_{r,a}):~{\gamma}(0)\in \mathcal {A}_a, \gamma(1)\in I_{\mu}^0\},\]
then we have
\[\widetilde{c}_{\mu}(a)={c}_{\mu}(a)>0.\]
\end{proposition}
\bp
  Note that $\Gamma_a\times\{0\}\subset\widetilde{\Gamma}_a,$ we
see that $\widetilde{c}_{\mu}(a)\leq {c}_{\mu}(a).$ On the other
hand, for
$\widetilde{\gamma}(t)=(\widetilde{\gamma}_1(t),\widetilde{\gamma}_2(t))\in\widetilde{\Gamma}_a$,
we denote by $\gamma(t)=
\widetilde{\gamma}_1(t)\star\widetilde{\gamma}_2(t)$. Thus,
$\gamma(t)\in\Gamma_a,$ and so
\[\max_{t\in[0,1]}I(\widetilde{\gamma}(t))=\max_{t\in[0,1]}I_{\mu}(\widetilde{\gamma}_1(t)\star\widetilde{\gamma}_2(t))=\max_{t\in[0,1]}I_{\mu}(\gamma(t)),\]
which implies that $\widetilde{c}_{\mu}(a)\geq {c}_{\mu}(a)>0,$
using Corollary 5.1.
 \ep
 Next, we  show the existence of the
 $(PS)_{c_{\mu}(a)}$-sequence for $I(u, \theta)$ on $S_{r,a}\times\R\subset \mathbb{H}$. It is obtained
by a standard argument using Ekeland's variational principle and
constructing pseudo-gradient flow, see Proposition 2.2 \cite{J}.

 \begin{proposition} \label{Prop 5.4} Let
$\{h_n\}\subset\widetilde{\Gamma}_a$ satisfying that
\[\max_{t\in[0,1]}I(h_n(t))\leq \widetilde{c}_{\mu}(a)+\frac{1}{n},\]
then   there exists a sequence $\{(v_n, \theta_n)\}\subset S_{r,a}
\times\R$ such that

\begin{itemize}
\item[(i)] $I(v_n,\theta_n)\in
[\widetilde{c}_{\mu}(a)-\frac{1}{n},\widetilde{c}_{\mu}(a)+\frac{1}{n}],$
\item[(ii)] $\min_{t\in [0,1]}\|(v_n, \theta_n)-h_n(t)\|_{\mathbb{H}}\leq \frac{1}{\sqrt{n}}$; and
\item[(iii)]$\|(I|_{S_{r,a}\times\R})'(v_n,\theta_n)\|\leq\frac{2}{\sqrt{n}},$
that is,
\[|\langle I'(v_n,\theta_n),z\rangle_{\mathbb{H}^{-1}\times\mathbb{H}}|\leq\frac{2}{\sqrt{n}}\|z\|_{\mathbb{H}},\]for
all
    \[z\in \widetilde{T}_{(v_n,\theta_n)} \triangleq\{(z_1, z_2)\in \mathbb{H}: \langle v_n, z_1\rangle_{L^2} = 0\}.\]
\end{itemize}
\end{proposition}

\par It follows from the above proposition, we can obtain a special
$(PS)_{c_{\mu}(a)}$-sequence for $I_{\mu}(u)$ on $S_{r,a}\subset
H^s(\R^3).$

 \begin{proposition} \label{Prop 5.5}  Under the assumption $2+\frac{4s}{3}< q < 2^*_s$,
there exists a sequence $\{u_n\}\subset S_{r,a}$ such that

\begin{itemize}
\item[(1)] $I_{\mu}(u_n)\rightarrow c_{\mu}(a)$ as $n\rightarrow\infty;$

\item[(2)] $P_{\mu}(u_n)\rightarrow 0$ as $n\rightarrow\infty;$

\item[(3)] $(I_{\mu}|_{S_{r,a}})'(u_n)\rightarrow 0$ as
$n\rightarrow\infty,$ i.e., $\langle
I_{\mu}'(u_n),z\rangle_{H^{-1}\times H}\rightarrow0$, uniformly for
all $z$ satisfying
\[\|z\|_H\leq 1, ~~\mbox{where}~~z\in T_{u_n}:=\{z\in H:~\langle
u_n,z\rangle_{L^2}=0\}\]
\end{itemize}
\end{proposition}
\bp By Proposition \ref{Prop 5.3},
$\widetilde{c}_{\mu}(a)={c}_{\mu}(a)$. Hence, we can take $\{h_n =
((h_n)_1, 0)\}\in\widetilde{\Gamma}_a$ so as to
\[\max_{t\in[0,1]}I(h_n(t))\leq \widetilde{c}_{\mu}(a)+\frac{1}{n}.\]
It follows from Proposition \ref{Prop 5.4} that, there exists a
sequence $\{(v_n,\theta_n)\}\subset S_{r,a}\times\R$ such that as
$n\rightarrow\infty$, one has \be\label{e5.9}
I(v_n,\theta_n)\rightarrow {c}_{\mu}(a),~~~\theta_n\rightarrow0;\ee
 \be\label{e5.10}
(I|_{S_{r,a}\times\R})'(v_n,\theta_n)\rightarrow0.\ee Set $ u_n =
\theta_n\star v_n.$  Then, $I_{\mu}(u_n) = I(v_n, \theta_n),$ and by
\eqref{e5.9}, item (1) holds. To prove  conclusion (2), we utilize
\[\begin{split}
\partial_{\theta}I(v_n, \theta_n)&=se^{2s\theta_n}\|v_n\|^2+\frac{(3-2t)\lambda}{4}e^{(3-2t)\theta_n}\int_{\R^3}\phi_{v_n}v_n^2dx-
\frac{3\mu(q-2)}{2q}e^{(\frac{3q}{2}-3)\theta_n}\int_{\R^3}|v_n|^qdx\\
&\hspace{0.45cm}-\frac{3(2^*_s-2)}{22^*_{s}}e^{\frac{3(2^*_{s}-2)}{2}\theta_n}\int_{\R^3}|v_n|^{2^*_{s}}dx\\
&=s\|(-\Delta)^{\frac{s}{2}}u_n\|^2+\frac{(3-2t)\lambda}{4}\int_{\R^3}\phi_{u_n}^tu_n^2dx-\frac{3\mu(q-2)}{2q}\int_{\R^3}|u_n|^qdx\\
&\hspace{0.45cm}-\frac{3(2^*_s-2)}{22^*_{s}}\int_{\R^3}|u_n|^{2^*_{s}}dx\\
&=P_{\mu}(u_n)
\end{split}
\] which implies  item (2) by \eqref{e5.10}. To
show item (3), we set $z_n\in T_{u_n}.$  Then,

\[\begin{split}
I_{\mu}'(u_n)z_n&=\iint_{\R^6}\frac{(u_n(x)-u_n(y))(z_n(x)-z_n(y))}{|x-y|^{3+2s}}dxdy+\lambda\int_{\R^3}\phi_{u_n}^tu_nz_ndx\\
&\hspace{0.45cm}-\mu\int_{\R^3}|u_n|^{q-2}u_nz_ndx-\int_{\R^3}|u_n|^{2^*_s-2}u_nz_ndx\\
&=e^{\frac{(4s-3)\theta_n}{2}}\iint_{\R^6}\frac{(v_n(x)-v_n(y))(z_n(e^{-\theta_n}x)-z_n(e^{-\theta_n}y))}{|x-y|^{3+2s}}dxdy
\\
&\hspace{0.45cm}+e^{\frac{3-4t}{2}\theta_n}\int_{\R^3}\phi_{v_n}v_n(x)z_n(e^{-\theta_n}
x)dx-\mu
e^{\frac{3(q-3)}{2}\theta_n}\int_{\R^3}|v_n|^{q-2}v_n(x)z_n(e^{-\theta_n}x)dx\\
&\hspace{0.45cm}-e^{\frac{3(2^*_s-3)}{2}\theta_n}\int_{\R^3}|v_n|^{2^*_s-2}v_n(x)z_n(e^{-\theta_n}x)dx.
\end{split}
\]
Denote by $\widetilde{z}_n(x) =
e^{-\frac{3s}{2}}z_n(e^{-\theta_n}x),$ then we get
\[\langle I_{\mu}'(u_n),z_n\rangle_{H^{-1}\times H}=\langle I'(v_n,\theta_n),(\widetilde{z}_n,0)\rangle_{\mathbb{H}^{-1}\times\mathbb{H}}.\]
It is easy to check that
\[\begin{split}
\langle v_n,\widetilde{z}_n\rangle_{L^2}&=\int_{\R^3}v_n(x)
e^{-\frac{3s}{2}}z_n(e^{-\theta_n}x)dx\\
&=\int_{\R^3}v_n(e^{\theta_n}x)
e^{\frac{3s}{2}}z_n(x)dx\\
&=\int_{\R^3}u_n(x) z_n(x)dx=0
\end{split}
\]Therefore, we see that $(\widetilde{z}_n, 0)\in \widetilde{T}_{(v_n,\theta_n)}$. On the other hand,
\[ \|(\widetilde{z}_n, 0)\|_{\mathbb{H}}^2= \|\widetilde{z}_n\|_{H}^2=\|{z}_n\|_{2}^2+ e^{-2s\theta_n}\|{z}_n\|^2\leq C\|z_n\|^2,
              \]
where the last inequality follows by $\theta_n\rightarrow 0$.
Consequently, we conclude item (3). \ep

\vskip0.1in
 \par\noindent {\bf Remark 5.1} {\em From Propositions \ref{Prop 5.4},\ref{Prop 5.5}, we know that  $u_n := \theta_n\star v_n \subset
S_{r,a}$ is a (PS) sequence for  $I_{\mu}$  with the level
$c_{\mu}(a)$, that is \be\label{e5.11} I_{\mu}(u_n)\rightarrow
 c_{\mu}(a)~~~\mbox{as}~~~ n\rightarrow+\infty,\ee and
\be\label{e5.12}(I_{\mu}|_{S_{r,a}})' (u_n)\rightarrow 0
~~~\mbox{as}~~~ n\rightarrow+\infty.\ee }

\bl\label{Lemma 5.6} The (PS) sequence  $\{u_n\}$ mentioned in
Remark 5.1 is bounded in $H^s_{rad}(\R^3).$   Moreover, suppose that
$c_{\mu}(a)<\frac{s}{3}S^{\frac{3}{2s}}$, and $\lambda<\lambda^*_1$
for some $\lambda^*_1>0,$ then
$\lim_{n\rightarrow+\infty}\alpha_n=\alpha<0$.
 \el
\bp From Remark 5.1 we see that $I_{\mu}(u_n)$ is bounded. In fact,
by   $P_{\mu}(u_n)\rightarrow0$ as $n\rightarrow\infty,$ we have
\[\left|(1+2t)I_{\mu}(u_n)+P_{\mu}(u_n)\right|\leq 3c_{\mu}(a),\]which implies that,
\be\label{e5.13}
\begin{split}
&\frac{1+2s+2t}{2}\|(-\Delta)^{\frac{s}{2}}u_n\|_2^2+\lambda\int_{\R^3}\phi_{u_n}^tu^2_ndx-\mu\left(\frac{1+2t}{2}+s\delta_{q,s}\right)\int_{\R^3}|u_n|^qdx\\
&\hspace{0.45cm}-\left(\frac{1+2t}{2^*_s}+s\right)\int_{\R^3}|u_n|^{2^*_s}dx\geq
-3c_{\mu}(a). \end{split} \ee In view of the boundedness of
$I_{\mu}(u_n),$ we have
\be\label{e5.14}\|(-\Delta)^{\frac{s}{2}}u_n\|^2+\frac{\lambda}{2}\int_{\R^3}\phi_{u_n}^tu^2_ndx\leq
6c_{\mu}(a)+\frac{2\mu}{q} \int_{\R^3}|u_n|^qdx+\frac{2}{2^*_s}
\int_{\R^3}|u_n|^{2^*_s}dx.\ee By \eqref{e5.13}-\eqref{e5.14}, we
 obtain
\[\frac{2s+2t-3}{4}\int_{\R^3}\phi_{u_n}^tu^2_ndx+\mu \frac{(\delta_{q,s}-2)s}{q}\int_{\R^3}|u_n|^qdx+
 \frac{(2^*_s-2)s}{2^*_s}\int_{\R^3}|u_n|^{2^*_s}dx\leq 3c_{\mu}(a)(2+2s+2t).\]
Note that $2s+2t>3, q>2+\frac{4s}{3}$, we have that $
q\delta_{q,s}-2>0,$  and so
\[\int_{\R^3}\phi_{u_n}^tu^2_ndx,~~~\int_{\R^3}|u_n|^qdx~~~\mbox{and}~~~\int_{\R^3}|u_n|^{2^*_s}dx\]are
all bounded. Thus, $\|(-\Delta)^{\frac{s}{2}}u_n\|_2\leq R_2$ for
some $R_2>0$ independently on $n\in\mathbb{N}$. Since
 $\{u_n\}\subset S_{r,a},$  we see that $\{u_n\}$ is bounded in $H^s_{rad}(\R^3).$
 Thus,   passing to a subsequence,
 and we may assume that
$u_n\rightharpoonup u$ for some $u\in H^s_{rad}(\R^3)$, and so
$u_n\rightarrow u $ in $L^p(\R^3), \forall p\in (2,2^*_s)$.

\par Now, we set the functional
 $\Phi: ~H^s_{rad}(\R^3) \rightarrow \R$ as
\[\Phi(u)=\frac{1}{2}\int_{\R^3}|u|^2dx,\]then $S_{r,a} = \Phi^{-1}\left(\frac{a^2}{2}\right)$. As a result,  it can be derived  from   Proposition
5.12 \cite{Willem} that there is a sequence $\{\alpha_n\}\subset\R$
such that
\[ I_{\mu}'(u_n)-\alpha_n\Phi'(u_n) \rightarrow
0~~\mbox{in}~H^{-s}_{rad}(\R^3)~~\mbox{as}~~n\rightarrow\infty.\]
 That is, we have
\begin{equation}\label{e5.15}
(-\Delta)^su_n+\phi_{u_n}^tu_n-\mu|u_n|^{q-2}u_n-|u_n|^{2^*_s-2}u_n=\alpha_n
u_n+o_n(1) ~~\mbox{in}~H^{-s}_{rad}(\R^3),
\end{equation}Similar to the proof of Lemma \ref{Lemma 4.3}, we  know that $u$ solves the equation
\begin{equation}\label{e5.16}
(-\Delta)^su+\phi^t_{u}u-\mu|u|^{q-2}u-|u|^{2^*_s-2}u=\alpha u.
\end{equation}
 Moreover,
$u\not\equiv0$. In fact, argue by contradiction that $u\equiv0$.
Then $u_n\rightarrow0$ in $L^p(\R^3),~ \forall ~p\in (2,2^*_s)$, and
by $P_{\mu}(u_n)=o_n(1)$, \eqref{e3.3}, we have
\[\begin{split}
o_n(1)&=s\|u_n\|^2+\lambda\frac{3-2t}{4}\int_{\R^3}\phi^t_{u_n}u_n^2dx-\mu
s\delta_{q,s}\int_{\R^3}|u_n|^qdx-s\int_{\R^3}|u_n|^{2^*_s}dx\\
&=s\|u_n\|^2-s\int_{\R^3}|u_n|^{2^*_s}dx+o_n(1).
\end{split}
\]We may assume that
$\lim_{n\rightarrow+\infty}\|u_n\|^2=\lim_{n\rightarrow+\infty}\int_{\R^3}|u_n|^{2^*_s}dx=\vartheta\geq0.$
 Thus, we have
 \be\label{e5.17}\begin{split}
c_{\mu}(a)+o_n(1)&=I_{\mu}(u_n)\\
  &=\frac{1}{2}\int_{\R^3}|(-\Delta )^\frac{s}{2}
u_n|^2dx+\frac{\lambda}{4}
\int_{\R^3}\phi^t_{u_n}u_n^2dx-\frac{\mu}{q}\int_{\R^3}|u_n|^qdx
-\frac{1}{2^*_s}\int_{\R^3}|u_n|^{2^*_s}dx\\
&=\frac{1}{2}\vartheta-\frac{1}{2^*_s}\vartheta+o_n(1)=\frac{s}{3}\vartheta+o_n(1).
\end{split}
\ee On the other hand, by the Sobolev inequality \eqref{e3.1}, we
have $\vartheta\geq S \vartheta^{\frac{2}{2^*_s}}.$ Then we have two
possible cases: (i) $\vartheta=0$; (ii) $\vartheta\geq
S^{\frac{3}{2s}}$.
\par If $\vartheta=0$, then by \eqref{e5.17} we get $I_{\mu}(u_n)\rightarrow
0$, which contradicts  to $I_{\mu}(u_n)\rightarrow c_{\mu}(a)>0.$
Now if the second case $\vartheta\geq S^{\frac{3}{2s}}$ occurs, then
by \eqref{e5.17} we get $I_{\mu}(u_n)\rightarrow
\frac{s}{3}\vartheta\geq \frac{s}{3}S^{\frac{3}{2s}}$, which
contradicts  to $I_{\mu}(u_n)\rightarrow
 c_{\mu}(a)<\frac{s}{3}S^{\frac{3}{2s}}$. Hence, $u\not\equiv0.$
 Moreover, by \eqref{e5.15} and $P_{\mu}(u_n)=o_n(1)$, we have
 \be\label{A1}
s\alpha_n
\|u_n\|_2^2=\lambda\frac{2t+4s-3}{4}\int_{\R^3}\phi^t_{u_n}u_n^2dx+\frac{q(3-2s)-6}{2q}\mu\int_{\R^3}|u_n|^qdx+o_n(1).
\ee Since  $\{u_n\}\subset S_{r,a}$ is bounded in $H^s_{rad}(\R^3),$
then by Lemma \ref{Lemma 3.6} and \eqref{A1},  we derive that
$\{\alpha_n\}$ is bounded and
 $\lim_{n\rightarrow+\infty}\alpha_n=\alpha\in\R$.
  By a  similar argument  as in  \eqref{e4.32} and
  \eqref{e4.33}, for all $n\in\mathbb{N}$, we have
\be \label{A2}\begin{split} T_1\leq \int_{\R^3}|u_n|^qdx&\leq
{C}(q,s) \|(-\Delta)^{\frac{s}{2}}u_n\|_2^{q\delta_{q,s}}\|u_n\|_2^{q(1-\delta_{q,s})}\\
&\leq
 {C}(q,s)R_2^{q\delta_{q,s}}a^{q(1-\delta_{q,s})},
\end{split}
\ee and
 \be \label{A3}\begin{split}
\int_{\R^3}\phi_{u_n}^tu_n^2dx\leq
\Gamma_t\|u_n\|_{\frac{12}{3+2t}}^4&\leq
\Gamma_tC\left({12}/{3+2t},s\right)^{\frac{3+2t}{3}}\|(-\Delta)^{\frac{s}{2}}u_n\|_2^{\frac{3-2t}{s}}\|u_n\|_2^{\frac{2t+4s-3}{s}}\\
&\leq\Gamma_tC\left({12}/{3+2t},s\right)^{\frac{3+2t}{3}}R_2^{\frac{3-2t}{s}}a^{\frac{2t+4s-3}{s}}\\
&:=T_2,
\end{split}
\ee where $T_2=Q_2(s,t,R_2,a)>0.$ We define the positive constant
\be\label{A4} \lambda_1^*:=\frac{2[6-q(3-2s)]\mu
T_1}{q(2t+4s-3)T_2}. \ee Therefore, if $\lambda<\lambda_1^*$, we get
\[\lambda q(2t+4s-3)T_2<2[6-q(3-2s)]\mu T_1.\]
Hence, by \eqref{A2},\eqref{A3} we see that \be\label{A5}\lambda
\frac{2t+4s-3}{4}\int_{\R^3}\phi_{u_n}^tu_n^2dx<\frac{[6-q(3-2s)]\mu}{2q}
\int_{\R^3}|u_n|^qdx.\ee Taking the limit in \eqref{A4} as
$n\rightarrow+\infty,$ and applying Lemmas \ref{Lemma
3.3},\ref{Lemma 3.6}, we obtain \be\label{A6}\lambda
\frac{2t+4s-3}{4}\int_{\R^3}\phi_{u}^tu^2dx<\frac{[6-q(3-2s)]\mu}{2q}
\int_{\R^3}|u|^qdx.\ee
 Consequently, passing the limit in \eqref{A1} as
$n\rightarrow+\infty,$ and using \eqref{A6} we deduce that
\[
s\alpha
 a^2=\lambda\frac{2t+4s-3}{4}\int_{\R^3}\phi^t_{u}u^2dx+\frac{q(3-2s)-6}{2q}\mu\int_{\R^3}|u|^qdx<0.
\] Thus, we have that $\alpha< 0$, if $\lambda<\lambda^*_1$ small.\ep

 \bl\label{Lemma 5.7} If $2+\frac{4s}{3} < q < 2^*_s,$ and inequality \eqref{e2.5} holds,
 then there $\lambda^*_2>0$, such that $ c_{\mu}(a) <\frac{s}{3}S^{\frac{3}{2s}}$ for $\lambda<\lambda^*_2$ small. \el
 \bp From \cite{CT}, we know that $S$ defined in \eqref{e3.1} is attained in $\R^3$
 by functions
 \[
 U_{\varepsilon}(x)=\frac{C(s)\varepsilon^{3-2s}}{(\varepsilon^2+|x|^2)^{\frac{3-2s}{2}}}\]
  for any $\varepsilon>0$ and   $C(s)$ being normalized constant such that
\[\|(-\Delta)^{\frac{s}{2}}U_{\varepsilon}\|_2^2=\int_{\R^3}|U_{\varepsilon}|^{2^*_s}dx=S^{\frac{3}{2s}}.\]
 \par  We define $u_{\varepsilon}  =\varphi
 U_{\varepsilon}$,  and
\[v_{\varepsilon}=a\frac{u_{\varepsilon}}{\|u_{\varepsilon}\|_2}\in S_a\cap H_{rad}^s(\R^3),\]where  $\varphi(x)\in C_0^{\infty}(B_2(0))$ is a radial cutoff function
such that $0\leq\varphi(x)\leq 1$ and $\varphi(x)\equiv 1$ on
 $B_1(0).$
From Proposition 21 and Proposition 22 in \cite{SV}, we have
\be\label{e5.18} \int_{\R^3}|(-\Delta)^{\frac{s}{2}}
u_{\varepsilon}|^2dx= S^{\frac{3}{2s}}+O(\varepsilon^{3-2s}).\ee

 \be\label{e5.19}
\int_{\R^3}|u_{\varepsilon}|^{2^*_s}dx=
S^{\frac{3}{2s}}+O(\varepsilon^{3}).\ee For any $p>1$, by a direct
computation \cite{Teng}, we obtain the following estimations:
 \be\label{e5.20}
\int_{\R^N}|u_{\varepsilon}|^pdx  =\begin{cases}
O(\varepsilon^{\frac{3(2-p)+2sp}{2}}),&\mbox{if}~~p>\frac{3}{3-2s};\vspace{0.1cm}\\
O(\varepsilon^{\frac{3}{2}}|\log\varepsilon|),&\mbox{if}~~p=\frac{3}{3-2s};\vspace{0.1cm}\\
O(\varepsilon^{\frac{(3-2s)p}{2}}),&\mbox{if}~~p<\frac{3}{3-2s},
\end{cases}
 \ee
and especially,
 \be\label{e5.21}\begin{split}
\int_{\R^3}|u_{\varepsilon}|^2dx & =\left\{\begin{array}{ll}
 C\varepsilon^{2s},&\mbox{if}~~0<s<\frac{3}{4};\vspace{0.1cm}\\
 C\varepsilon^{2s}|\log\varepsilon|,&\mbox{if}~~s=\frac{3}{4};\vspace{0.1cm}\\
 C\varepsilon^{3-2s},&\mbox{if}~~\frac{3}{4}<s<1.
\end{array}
\right.
\end{split}
 \ee

Define the function
 \be\label{e5.22}\begin{split}
 \Psi^{\mu}_{v_{\varepsilon}}(\theta)&:=I_{\mu} ((\theta\star v_{\varepsilon}))=\frac{e^{2s\theta}}{2}\|v_{\varepsilon}\|^2+
 \frac{e^{(3-2t)\theta}}{4}\lambda\int_{\R^3}\phi^t_{v_{\varepsilon}}v_{\varepsilon}^2dx-\frac{\mu}{q}e^{\frac{3(q-2)}{2}\theta}\int_{\R^3}|v_{\varepsilon}|^qdx\\
&\hspace{0.45cm}-\frac{1}{2^*_{s}}e^{2^*_s
s\theta}\int_{\R^3}|v_{\varepsilon}|^{2^*_{s}}dx,
\end{split}\ee
then it is easy to see that $
\Psi^{\mu}_{v_{\varepsilon}}(\theta)\rightarrow 0^+$ as
$\theta\rightarrow-\infty$, and
$\Psi^{\mu}_{v_{\varepsilon}}(\theta)\rightarrow -\infty$ as
$\theta\rightarrow+\infty$. Therefore,
$\Psi^{\mu}_{v_{\varepsilon}}$ can obtain its global positive
maximum at some $\theta_{\varepsilon,\mu} >0$.    A direct
computation yields that
 \be\label{e5.23}\begin{split}&(\Psi^{\mu}_{v_{\varepsilon}})'(\theta)\\
 &=se^{2s\theta}\|v_{\varepsilon}\|^2+\frac{3-2t}{4}e^{(3-2t)\theta}\lambda\int_{\R^3}\phi^t_{v_{\varepsilon}}v_{\varepsilon}^2dx\\
&\hspace{0.45cm}-
\frac{3\mu(q-2)}{2q}e^{\frac{3(q-2)}{2}\theta}\int_{\R^3}|v_{\varepsilon}|^qdx-se^{2^*_{s}
s\theta}\int_{\R^3}|v_{\varepsilon}|^{2^*_s}dx\\
&=s\|\theta\star
v_{\varepsilon}\|^2+\frac{3-2t}{4}\lambda\int_{\R^3}\phi^t_{\theta\star
v_{\varepsilon}}|\theta\star v_{\varepsilon}|^2dx-
\frac{3\mu(q-2)}{2q}\int_{\R^3}|\theta\star
v_{\varepsilon}|^qdx-s\int_{\R^3}|\theta\star
v_{\varepsilon}|^{2^*_s}dx\\
&=P_{\mu}(\theta\star v_{\varepsilon});\end{split}\ee
and\[\begin{split}
(\Psi^{\mu}_{v_{\varepsilon}})''(\theta)&=2s^2e^{2s\theta}\|v_{\varepsilon}\|^2+\frac{(3-2t)^2}{4}e^{(3-2t)\theta}\lambda\int_{\R^3}\phi^t_{v_{\varepsilon}}v_{\varepsilon}^2dx\\
&\hspace{0.45cm}-\mu
 qs^2\delta_{q,s}^2e^{\frac{3(q-2)}{2}\theta}\int_{\R^3}|v_{\varepsilon}|^qdx-2^*_{s}s^2e^{2^*_{s}
s\theta}\int_{\R^3}|v_{\varepsilon}|^{2^*_s}dx.\end{split}
\]
Let $\theta_{\varepsilon,\mu}$ be the maximum point of
$\Psi^{\mu}_{v_{\varepsilon}}(\theta)$, then
 $\theta_{\varepsilon,\mu}$ is unique. In fact, combining with $(\Psi^{\mu}_{v_{\varepsilon}})'(\theta_{\varepsilon,\mu})=0,$ and $3-2t-2s<0, 2-q\delta_{q,s}<0, 2-2^*_{s}<0,$ we have
\[\begin{split}
&(\Psi^{\mu}_{v_{\varepsilon}})''(\theta_{\varepsilon,\mu})\\
&=2s^2e^{2s\theta_{\varepsilon,\mu}}\|v_{\varepsilon}\|^2+
\frac{(3-2t)^2}{4}e^{(3-2t)\theta_{\varepsilon,\mu}}\lambda\int_{\R^3}\phi^t_{v_{\varepsilon}}v_{\varepsilon}^2dx\\
&\hspace{0.45cm}- \mu q
s^2\delta_{q,s}^2e^{\frac{3(q-2)}{2}\theta_{\varepsilon,\mu}}\int_{\R^3}|v_{\varepsilon}|^qdx-2^*_{s}s^2e^{2^*_{s}
s\theta_{\varepsilon,\mu}}\int_{\R^3}|v_{\varepsilon}|^{2^*_s}dx\\
&=2s^2\|\widetilde{u}_{\varepsilon}\|^2+\frac{(3-2t)^2}{4}\lambda\int_{\R^3}\phi^t_{\widetilde{u}_{\varepsilon}}\widetilde{u}_{\varepsilon}^2dx-
\mu
s^2q\delta_{q,s}^2\int_{\R^3}|\widetilde{u}_{\varepsilon}|^qdx-2^*_{s}s^2\int_{\R^3}|\widetilde{u}_{\varepsilon}|^{2^*_s}dx\\
&=\frac{(3-2t)(3-2t-2s)}{4}\lambda\int_{\R^3}\phi^t_{\widetilde{u}_{\varepsilon}}\widetilde{u}_{\varepsilon}^2dx+\mu
s^2\delta_{q,s}[2-q\delta_{q,s}]\int_{\R^3}|\widetilde{u}_{\varepsilon}|^qdx+s^2[2-2^*_{s}]\int_{\R^3}|\widetilde{u}_{\varepsilon}|^{2^*_s}dx<0,
\end{split}
\]where $\widetilde{u}_{\varepsilon}=\theta_{\varepsilon,\mu}\star
v_{\varepsilon},$ and the uniqueness of $\theta_{\varepsilon,\mu}$
follows. Using
$(\Psi^{\mu}_{v_{\varepsilon}})'(\theta_{\varepsilon,\mu})=P_{\mu}(\theta_{\varepsilon,\mu}\star
v_{\varepsilon})=0$ again,  we have \be\label{e5.24}\begin{split}
se^{2^*_{s} s\theta_{\varepsilon,\mu}}
 \int_{\R^3}|v_{\varepsilon}|^{2^*_s}dx&=se^{2s\theta_{\varepsilon,\mu}}\|v_{\varepsilon}\|^2+\lambda\frac{3-2t}{4}e^{(3-2t)\theta_{\varepsilon,\mu}}
\int_{\R^3}\phi^t_{v_{\varepsilon}}v_{\varepsilon}^2dx\\
&\hspace{0.45cm}-\frac{3\mu(q-2)}{2q}e^{\frac{3(q-2)}{2}\theta_{\varepsilon,\mu}}\int_{\R^3}|v_{\varepsilon}|^qdx\\
&\leq
se^{2s\theta_{\varepsilon,\mu}}\|v_{\varepsilon}\|^2+\lambda\frac{3-2t}{4}e^{(3-2t)\theta_{\varepsilon,\mu}}\int_{\R^3}\phi^t_{v_{\varepsilon}}v_{\varepsilon}^2dx\\
 &=e^{2s\theta_{\varepsilon,\mu}}\left(s\|v_{\varepsilon}\|^2+\lambda\frac{3-2t}{4}e^{(3-2t-2s)\theta_{\varepsilon,\mu}}\int_{\R^3}\phi^t_{v_{\varepsilon}}v_{\varepsilon}^2dx\right)\\
&\leq e^{2s\theta_{\varepsilon,\mu}}
2\max\left\{s\|v_{\varepsilon}\|^2,
 \lambda\frac{3-2t}{4}e^{(3-2t-2s)\theta_{\varepsilon,\mu}}\int_{\R^3}\phi^t_{v_{\varepsilon}}v_{\varepsilon}^2dx\right\}.
\end{split}
\ee \par In the sequel,  we distinguish  the following  two possible
cases.

 \vskip0.1in \par\noindent
{\bf Case 1.} $s\|v_{\varepsilon}\|^2>\lambda
\frac{3-2t}{4}e^{(3-2t-2s)\theta_{\varepsilon,\mu}}\int_{\R^3}\phi^t_{v_{\varepsilon}}v_{\varepsilon}^2dx$.
\par In this case, we have from \eqref{e5.24} that \be\label{e5.25} se^{2^*_{s}
s\theta_{\varepsilon,\mu}}
\int_{\R^3}|v_{\varepsilon}|^{2^*_s}dx<e^{2s\theta_{\varepsilon,\mu}}
2s\|v_{\varepsilon}\|^2 \Longrightarrow
e^{(2^*_{s}-2)s\theta_{\varepsilon,\mu}}\leq
\frac{2\|v_{\varepsilon}\|^2}{\|v_{\varepsilon}\|_{2^*_s}^{2^*_s}},\ee
and from
 $(\Psi^{\mu}_{v_{\varepsilon}})'(\theta_{\varepsilon,\mu})=0$, we
  have
\be\label{e5.26}\begin{split}&e^{(2^*_{s}-2)s\theta_{\varepsilon,\mu}}\\
&=\frac{\|v_{\varepsilon}\|^2}{\|v_{\varepsilon}\|_{2^*_s}^{2^*_s}}
+\lambda\frac{3-2t}{4s}\frac{e^{(3-2t-2s)\theta_{\varepsilon,\mu}}
\int_{\R^3}\phi^t_{v_{\varepsilon}}v_{\varepsilon}^2dx}{\|v_{\varepsilon}\|_{2^*_s}^{2^*_s}}
-\mu\delta_{q,s}e^{(q\delta_{q,s}-2)s\theta_{\varepsilon,\mu}}\frac{\|v_{\varepsilon}\|_{q}^{q}}{\|v_{\varepsilon}\|_{2^*_s}^{2^*_s}}\\
&\geq
\frac{\|v_{\varepsilon}\|^2}{\|v_{\varepsilon}\|_{2^*_s}^{2^*_s}}-\mu\delta_{q,s}\left(\frac{2\|v_{\varepsilon}\|^2}{\|v_{\varepsilon}\|_{2^*_s}^{2^*_s}}\right)^{\frac{q\delta_{q,s}-2}{2^*_s-2}}\frac{\|v_{\varepsilon}\|_{q}^{q}}{\|v_{\varepsilon}\|_{2^*_s}^{2^*_s}} \\
&=\frac{\|u_{\varepsilon}\|_2^{2^*_s-2}\|u_{\varepsilon}\|^2}{a^{2^*_s-2}\|u_{\varepsilon}\|_{2^*_s}^{2^*_s}}-\mu\delta_{q,s}\left(\frac{2\|u_{\varepsilon}\|^{2^*_s-2}_2}{a^{2^*_s-2}}\frac{\|u_{\varepsilon}\|^2}{\|u_{\varepsilon}\|^{2^*_s}_{2^*_s}}\right)^{\frac{q\delta_{q,s}-2}{2^*_s-2}}
 \frac{\|u_{\varepsilon}\|_{q}^{q}}{\|u_{\varepsilon}\|_{2^*_s}^{2^*_s}}
\frac{\|u_{\varepsilon}\|_{2}^{2^*_s-q}}{a^{2^*_s-q}}
\end{split}\ee
\[\begin{split}
&=\frac{\|u_{\varepsilon}\|_2^{2^*_s-2}(\|u_{\varepsilon}\|^2)^{\frac{q\delta_{q,s}-2}{2^*_s-2}}}{a^{2^*_s-2}\|u_{\varepsilon}\|_{2^*_s}^{2^*_s}}\left[(\|u_{\varepsilon}\|^2)^{\frac{2^*_s-q\delta_{q,s}}{2^*_s-2}}-
\frac{\mu\delta_{q,s}2^{\frac{q\delta_{q,s}-2}{2^*_s-2}}a^{q(1-\delta_{q,s})}\|u_{\varepsilon}\|^q_q}{(\|u_{\varepsilon}\|_2)^{q(1-\delta_{q,s})}(\|u_{\varepsilon}\|_{2^*_s}^{2^*_s})^{\frac{q\delta_{q,s}-2}{2^*_s-2}}}\right].
\end{split}
\] Notice that, by \eqref{e5.18}-\eqref{e5.21}, there
exist positive constants $C_1, C_2$ and $C_3$ depending on $s$ and
$q$ such that
\be\label{e5.27}(\|u_{\varepsilon}\|^2)^{\frac{2^*_s-q\delta_{q,s}}{2^*_s-2}}\geq
C_1,~~C_2\leq(\|u_{\varepsilon}\|_{2^*_s}^{2^*_s})^{\frac{q\delta_{q,s}-2}{2^*_s-2}}\leq
\frac{1}{C_2}.\ee and \be\label{e5.28}
\frac{\|u_{\varepsilon}\|^{q}_q}{\|u_{\varepsilon}\|_2^{q(1-\gamma_{q,s})}}
 =\left\{\begin{array}{ll}
C_3\varepsilon^{3-\frac{3-2s}{2}q-sq(1-\gamma_{q,s})}=C_3,&\mbox{if}~~0<s<\frac{3}{4};\vspace{0.2cm}\\
C_3|\ln\varepsilon|^{\frac{q(\gamma_{q,s}-1)}{2}},&\mbox{if}~~s=\frac{3}{4};\vspace{0.2cm}\\
C_3\varepsilon^{3-\frac{3-2s}{2}q-\frac{(3-2s)q(1-\gamma_{q,s})}{2}},&\mbox{if}~~\frac{3}{4}<s<1;
\end{array} \right.
\ee Next, we show that
\be\label{e5.29}e^{(2^*_{s}-2)s\theta_{\varepsilon,\mu}}\geq
C\frac{\|u_{\varepsilon}\|^{2^*_s-2}_{2}}{a^{2^*_s-2}},\ee under
suitable conditions. To this aim, we distinguish the following three
 subcases.
\par {\em  Subcase (i).} $0<s<\frac{3}{4}.$  In this case, it holds that
\be\label{e5.30}3-\frac{3-2s}{2}q-sq(1-\delta_{q,s})=0,\ee and from
\eqref{e5.26}-\eqref{e5.28} we have
\[e^{(2^*_s-2)s\theta_{\varepsilon,\mu}}\geq
\frac{C\|u_{\varepsilon}\|^{2^*_s-2}_{2}}{a^{2^*_s-2}} \left[C_1
-\mu\delta_{q,s}a^{q(1-\gamma_{q,s})}2^{\frac{q\delta_{q,s}-2}{2^*_s-2}}\frac{C_3}{C_2}\right],\]
and we see that inequality \eqref{e5.29} holds only when
$\mu\gamma_{q,s}a^{q(1-\delta_{q,s})}<
C_1C_2(C_3)^{-1}2^{-\frac{q\delta_{q,s}-2}{2^*_s-2}}$. Thus, we have
to give a more precise estimate, let us   come back to \eqref{e5.26}
and observe that by well-known interpolation inequality, we have
\be\label{e5.31}
 \frac{\|u_{\varepsilon}\|^q_q}{(\|u_{\varepsilon}\|_2)^{q(1-\delta_{q,s})}(\|u_{\varepsilon}\|_{2^*_s}^{2^*_s})^{\frac{q\delta_{q,s}-2}{2^*_s-2}}}
 \leq \frac{(\|u_{\varepsilon}\|^{2^*_s}_{2^*_s})^{\frac{q-2}{2^*_s-2}}(\|u_{\varepsilon}\|^2_2)^{\frac{2^*_s-q}{2^*_s-2}}}{(\|u_{\varepsilon}\|_2)^{q(1-\delta_{q,s})}
 (\|u_{\varepsilon}\|_{2^*_s}^{2^*_s})^{\frac{q\delta_{q,s}-2}{2^*_s-2}}}
 =(\|u_{\varepsilon}\|_{2^*_s}^{2^*_s})^{\frac{q(1-\delta_{q,s})}{2^*_s-2}}.
\ee Therefore, by \eqref{e5.31} and \eqref{e5.26} we have
\be\label{e5.32}
e^{(2^*_{s}-2)s\theta_{\varepsilon,\mu}}\geq\frac{\|u_{\varepsilon}\|_2^{2^*_s-2}(\|u_{\varepsilon}\|^2)^{\frac{q\delta_{q,s}-2}{2^*_s-2}}}{a^{2^*_s-2}\|u_{\varepsilon}\|_{2^*_s}^{2^*_s}}
\left[(\|u_{\varepsilon}\|^2)^{\frac{2^*_s-q\delta_{q,s}}{2^*_s-2}}-\mu\delta_{q,s}2^{\frac{q\delta_{q,s}-2}{2^*_s-2}}a^{q(1-\delta_{q,s})}
 (\|u_{\varepsilon}\|_{2^*_s}^{2^*_s})^{\frac{q(1-\delta_{q,s})}{2^*_s-2}}\right].
\ee From the  estimations \eqref{e5.18},\eqref{e5.19}, we see that
the right hand side of \eqref{e5.32} is positive provided that
\[\begin{split}
\mu\delta_{q,s}a^{q(1-\gamma_{q,s})}2^{\frac{q\delta_{q,s}-2}{2^*_s-2}}&<\frac{\left(\|u_{\varepsilon}\|^2\right)^{\frac{2^*_s-q\gamma_{q,s}}{2^*_s-2}}
}{(\|u_{\varepsilon}\|_{2^*_s}^{2^*_s})^{\frac{q(1-\delta_{q,s})}{2^*_s-2}}}\\
&=\frac{\left(S^{\frac{3}{2s}}+O(\varepsilon^{3-2s})\right)^{\frac{2^*_s-q\gamma_{q,s}}{2^*_s-2}}
}{\left(S^{\frac{3}{2s}}+O(\varepsilon^3)\right)^{\frac{q(1-\delta_{q,s})}{2^*_s-2}}}=S^{\frac{3(2^*_s-q)}{2s(2^*_s-2)}}+O(\varepsilon^{3-2s}).
\end{split}
\]   Therefore, if $0<s<\frac{3}{4}$ and
 \be\label{e5.33}\mu{\delta_{q,s}}
a^{q(1-\delta_{q,s})}2^{\frac{q\delta_{q,s}-2}{2^*_s-2}}<{S^{\frac{3(2^*_s-q)}{2s(2^*_s-2)}}},\ee
we have
\[e^{(2^*_{s}-2)s\theta_{\widetilde{v}_{\varepsilon}}}\geq\frac{C\|u_{\varepsilon}\|^{2^*_{s}-2}_{2}}{a^{2^*_{s}-2}}.\]

  \par {\em  Subcase (ii).} $s=\frac{3}{4}$.  In this case, then we have $3<q<4$, and
\[
|\ln\varepsilon|^{\frac{q(\gamma_{q,s}-1)}{2}}=|\ln\varepsilon|^{\frac{q-2^*_s}{4s(3-2s)}}\rightarrow0~~
\mbox{as}~~\varepsilon\rightarrow0.\] Consequently,
\[\frac{\|u_{\varepsilon}\|^{q}_q}{\|u_{\varepsilon}\|_2^{q(1-\gamma_{q,s})}}
\leq
C_3\varepsilon^{3-\frac{3-2s}{2}q-sq(1-\gamma_{q,s})}|\ln\varepsilon|^{\frac{q(\gamma_{q,s}-1)}{2}}=o_{\varepsilon}(1).\]
Therefore, we get
\[e^{(2^*_{s}-2)s\theta_{{v}_{\varepsilon}}}\geq
 C\frac{\|u_{\varepsilon}\|^{2^*_{s}-2}_{2}}{a^{2^*_{s}-2}}\left[C_1
-\mu\gamma_{q,s}a^{q(1-\gamma_{q,s})}2^{\frac{q\delta_{q,s}-2}{2^*_s-2}}\frac{C_3}{C_2}o_{\varepsilon}(1)\right]
\geq\frac{C\|u_{\varepsilon}\|^{2^*_{s}-2}_{2}}{a^{2^*_{s}-2}}.\]

\par \par {\em  Subcase (iii).}  $\frac{3}{4}<s<1$. By the definition of $\delta_{q,s}$ and a direct computation
we infer to
\[\begin{split}&3-\frac{3-2s}{2}q-\frac{(3-2s)q(1-\gamma_{q,s})}{2}\\
&=(3-2s)\left[\frac{3}{3-2s}-q-\frac{3(q-2)}{4s}\right]=\frac{3-4s}{4s}\left[q-\frac{6}{3-2s}\right](3-2s)>0.
\end{split}\]
Thus,
 $\varepsilon^{3-\frac{3-2s}{2}q-\frac{(3-2s)q(1-\gamma_{q,s})}{2}}\rightarrow0~~\mbox{as}~~\varepsilon\rightarrow0,$ and so
\[\frac{\|u_{\varepsilon}\|^{q}_q}{\|u_{\varepsilon}\|_2^{q(1-\gamma_{q,s})}}
\leq
C\varepsilon^{3-\frac{3-2s}{2}q-\frac{(3-2s)q(1-\gamma_{q,s})}{2}}=o_{\varepsilon}(1).\]
Therefore, we conclude that,
\[e^{(2^*_{s}-2)s\theta_{{v}_{\varepsilon}}}\geq
 C\frac{\|u_{\varepsilon}\|^{2^*_{s}-2}_{2}}{a^{2^*_{s}-2}}\left[C_1
-\mu\gamma_{q,s}a^{q(1-\gamma_{q,s})}\frac{C_3}{C_2}o_{\varepsilon}(1)\right]
\geq\frac{C\|u_{\varepsilon}\|^{22^*_{\alpha,s}-2}_{2}}{a^{2^*_{s}-2}}.\]

\vskip0.2in

\par {\bf Case 2.} $s\|v_{\varepsilon}\|^2\leq
\lambda\frac{3-2t}{4}e^{(3-2t-2s)\theta_{\varepsilon,\mu}}\int_{\R^3}\phi^t_{v_{\varepsilon}}v_{\varepsilon}^2dx$.
\vskip0.1in
\par In this case, we have from \eqref{e5.24} that $
se^{2^*_{s} s\theta_{\varepsilon,\mu}}
\int_{\R^3}|v_{\varepsilon}|^{2^*_s}dx<e^{2s\theta_{\varepsilon,\mu}}
\frac{3-2t}{2}e^{(3-2t-2s)\theta_{\varepsilon,\mu}}\lambda\int_{\R^3}\phi^t_{v_{\varepsilon}}v_{\varepsilon}^2dx,$
which implies that \be\label{e5.34}
e^{(s2^*_{s}+2t-3)\theta_{\varepsilon,\mu}}\leq
\frac{3-2t}{2s}\frac{\lambda\int_{\R^3}\phi^t_{v_{\varepsilon}}v_{\varepsilon}^2dx}{\|v_{\varepsilon}\|_{2^*_s}^{2^*_s}},\ee
and from
 $(\Psi^{\mu}_{v_{\varepsilon}})'(\theta_{\varepsilon,\mu})=0$ and \eqref{e5.34}, together with \eqref{e3.2}-\eqref{e3.3} and H\"{o}lder inequality, we induce that
\be\label{e5.35}\begin{split} &e^{(2^*_{s}-2)s\theta_{\varepsilon,\mu}}\\
&=\frac{\|v_{\varepsilon}\|^2}{\|v_{\varepsilon}\|_{2^*_s}^{2^*_s}}
+\frac{3-2t}{4s}\frac{e^{(3-2t-2s)\theta_{\varepsilon,\mu}}\lambda
\int_{\R^3}\phi^t_{v_{\varepsilon}}v_{\varepsilon}^2dx}{\|v_{\varepsilon}\|_{2^*_s}^{2^*_s}}
-\mu\delta_{q,s}e^{(q\delta_{q,s}-2)s\theta_{\varepsilon,\mu}}\frac{\|v_{\varepsilon}\|_{q}^{q}}{\|v_{\varepsilon}\|_{2^*_s}^{2^*_s}}\\
&\geq
\frac{\|v_{\varepsilon}\|^2}{\|v_{\varepsilon}\|_{2^*_s}^{2^*_s}}-\mu\delta_{q,s}\left(\frac{3-2t}{2s}\frac{\lambda\int_{\R^3}\phi_{v_{\varepsilon}}v_{\varepsilon}^2dx}{\|v_{\varepsilon}\|_{2^*_s}^{2^*_s}}
\right)^{\frac{(q\delta_{q,s}-2)s}{s2^*_{s}+2t-3}}\frac{\|v_{\varepsilon}\|_{q}^{q}}{\|v_{\varepsilon}\|_{2^*_s}^{2^*_s}} \\
&\geq
\frac{\|v_{\varepsilon}\|^2}{\|v_{\varepsilon}\|_{2^*_s}^{2^*_s}}-\mu\delta_{q,s}\left(\frac{3-2t}{2s}\frac{
\lambda
 \Gamma_t\|v_{\varepsilon}\|_{\frac{12}{3+2t}}^4}{\|v_{\varepsilon}\|_{2^*_s}^{2^*_s}}
\right)^{\frac{(q\delta_{q,s}-2)s}{s2^*_{s}+2t-3}}\frac{\|v_{\varepsilon}\|_{q}^{q}}{\|v_{\varepsilon}\|_{2^*_s}^{2^*_s}} \\
&\geq
\frac{\|v_{\varepsilon}\|^2}{\|v_{\varepsilon}\|_{2^*_s}^{2^*_s}}-\mu\delta_{q,s}\left(\frac{(3-2t)\lambda
 \Gamma_t}{2s}\right)^{\frac{(q\delta_{q,s}-2)s}{s2^*_{s}+2t-3}}\left(\frac{
\|v_{\varepsilon}\|^{4\tau}_2\|v_{\varepsilon}\|_{2^*_s}^{4(1-\tau)}}{\|v_{\varepsilon}\|_{2^*_s}^{2^*_s}}
\right)^{\frac{(q\delta_{q,s}-2)s}{s2^*_{s}+2t-3}}\frac{\|v_{\varepsilon}\|_{q}^{q}}{\|v_{\varepsilon}\|_{2^*_s}^{2^*_s}}\\
&=
\frac{\|v_{\varepsilon}\|^2}{\|v_{\varepsilon}\|_{2^*_s}^{2^*_s}}-\mu\delta_{q,s}D(s,t)a^{{\frac{4\tau(q\delta_{q,s}-2)s}{s2^*_{s}+2t-3}}}\left(\frac{
 1}{\|v_{\varepsilon}\|_{2^*_s}^{2^*_s-4(1-\tau)}}
\right)^{\frac{(q\delta_{q,s}-2)s}{s2^*_{s}+2t-3}}\frac{\|v_{\varepsilon}\|_{q}^{q}}{\|v_{\varepsilon}\|_{2^*_s}^{2^*_s}}
\\
&=\frac{\|u_{\varepsilon}\|_2^{2^*_s-2}\|u_{\varepsilon}\|^2}{a^{2^*_s-2}\|u_{\varepsilon}\|_{2^*_s}^{2^*_s}}-\mu\delta_{q,s}D(s,t)a^{{\frac{4\tau(q\delta_{q,s}-2)s}{s2^*_{s}+2t-3}}}
\\
&\hspace{4.45cm}\times\left(\frac{\|u_{\varepsilon}\|_2^{2^*_s-4(1-\tau)}}{a^{2^*_s-4(1-\tau)}\|u_{\varepsilon}\|_{2^*_s}^{2^*_s-4(1-\tau)}}
\right)^{{\frac{(q\delta_{q,s}-2)s}{s2^*_{s}+2t-3}}}\frac{\|u_{\varepsilon}\|_{q}^{q}}{\|u_{\varepsilon}\|_{2^*_s}^{2^*_s}}
\frac{\|u_{\varepsilon}\|_{2}^{2^*_s-q}}{a^{2^*_s-q}}\\
&=\frac{\|u_{\varepsilon}\|_2^{2^*_s-2}\|u_{\varepsilon}\|^2}{a^{2^*_s-2}\|u_{\varepsilon}\|_{2^*_s}^{2^*_s}}-\frac{\mu\delta_{q,s}D(s,t)a^{{\frac{4\tau(q\delta_{q,s}-2)s
-(2^*_s-4(1-\tau))(q\delta_{q,s}-2)s}{s2^*_{s}+2t-3}}-2^*_s+q}}
{\|u_{\varepsilon}\|_{2^*_s}^{2^*_s+[2^*_s-4(1-\tau)]\frac{(q\delta_{q,s}-2)s}{s2^*_{s}+2t-3}}}
\\
&\hspace{6.45cm}\times
\|u_{\varepsilon}\|_2^{2^*_s-q+[2^*_s-4(1-\tau)]\frac{(q\delta_{q,s}-2)s}{s2^*_{s}+2t-3}}\|u_{\varepsilon}\|_{q}^{q}\\
&=\frac{\|u_{\varepsilon}\|_2^{2^*_s-2}(\|u_{\varepsilon}\|^2)^{\frac{q\delta_{q,s}-2}{2^*_s-2}}}{a^{2^*_s-2}\|u_{\varepsilon}\|_{2^*_s}^{2^*_s}}\\
\end{split}
\ee
\[\begin{split}
&\hspace{0.45cm}\times
\bigg[(\|u_{\varepsilon}\|^2)^{\frac{2^*_s-q\delta_{q,s}}{2^*_s-2}}-\frac{\mu\delta_{q,s}D(s,t)a^{{\frac{4\tau(q\delta_{q,s}-2)s
-(2^*_s-4(1-\tau))(q\delta_{q,s}-2)s}{s2^*_{s}+2t-3}}+q-2}}
{(\|u_{\varepsilon}\|^2)^{\frac{q\delta_{q,s}-2}{2^*_s-2}}\|u_{\varepsilon}\|_{2^*_s}^{[2^*_s-4(1-\tau)]\frac{(q\delta_{q,s}-2)s}{s2^*_{s}+2t-3}}}
\\
&\hspace{6.45cm}\times
\|u_{\varepsilon}\|_2^{2-q+[2^*_s-4(1-\tau)]\frac{(q\delta_{q,s}-2)s}{s2^*_{s}+2t-3}}\|u_{\varepsilon}\|_{q}^{q}\bigg],
\end{split}
\] where   $0<\tau=\frac{2t+4s-3}{4s}<1,$ and
\[D(s,t)=\left(\frac{(3-2t)\lambda
 \Gamma_t}{2s}\right)^{\frac{(q\delta_{q,s}-2)s}{s2^*_{s}+2t-3}}.\]By a
direct computation, we have the following  clearer expressions
\be\label{e5.36}\begin{split}
[2^*_s-4(1-\tau)]\frac{(q\delta_{q,s}-2)s}{s2^*_{s}+2t-3}&=\left[2^*_s-4\left(1-\frac{2t+4s-3}{4s}\right)\right]\frac{(q\delta_{q,s}-2)s}{s2^*_{s}+2t-3}\\
&=\left[2^*_s-\frac{3-2t}{s}\right]\frac{(q\delta_{q,s}-2)s}{s2^*_{s}+2t-3}=q\delta_{q,s}-2;
\end{split}
\ee
 \be\label{e5.37}
 2-q+[2^*_s-4(1-\tau)]\frac{(q\delta_{q,s}-2)s}{s2^*_{s}+2t-3}=2-q+q\delta_{q,s}-2=(\delta_{q,s}-1)q;
\ee and \be\label{e5.38}\begin{split} &\frac{4\tau(q\delta_{q,s}-2)s
-(2^*_s-4(1-\tau))(q\delta_{q,s}-2)s}{s2^*_{s}+2t-3}+q-2\\
&=\frac{s(q\delta_{q,s}-2)(4-2^*_s)}{s2^*_{s}+2t-3}+q-2\\
&=\frac{1}{s2^*_{s}+2t-3}\left[(q-2)(s2^*_{s}+2t-3)-(2^*_s-4)s(q\delta_{q,s}-2)\right]\\
&=\frac{1}{s2^*_{s}+2t-3}\left[(q-2)(s2^*_{s}+2t-3)-(2^*_s-4)\left(\frac{3(q-2)}{2}-2s\right)\right]\\
&=\frac{(q-2)2t+2s(2^*_s-4)}{s2^*_{s}+2t-3}>0,
\end{split}\ee
where the last inequality holds true since
$q\in(2+\frac{4s}{3},2^*_s), 2s+2t>3$. Consequently, we have
\[\begin{split}(q-2)2t+2s(2^*_s-4)&>\frac{4s}{3}2t+2s(2^*_s-4)\\
&=2s\left(\frac{4t}{3}+2^*_s-4\right)=2s\frac{24s+12t-18-8st}{3(3-2s)}>0.
\end{split}
\]
Substituting formulas \eqref{e5.36}-\eqref{e5.38}     into
\eqref{e5.35},  we  infer to \be\label{e5.39}\begin{split}
 e^{(2^*_{s}-2)s\theta_{\varepsilon,\mu}}&\geq\frac{\|u_{\varepsilon}\|_2^{2^*_s-2}(\|u_{\varepsilon}\|^2)^{\frac{q\delta_{q,s}-2}{2^*_s-2}}}{a^{2^*_s-2}\|u_{\varepsilon}\|_{2^*_s}^{2^*_s}}\\
&\hspace{0.45cm}\times
\bigg[(\|u_{\varepsilon}\|^2)^{\frac{2^*_s-q\delta_{q,s}}{2^*_s-2}}-\frac{\mu\delta_{q,s}D(s,t)a^{\frac{(q-2)2t+2s(2^*_s-4)}{s2^*_{s}+2t-3}}}
{(\|u_{\varepsilon}\|^2)^{\frac{q\delta_{q,s}-2}{2^*_s-2}}\|u_{\varepsilon}\|_{2^*_s}^{q\delta_{q,s}-2}}\times
\frac{\|u_{\varepsilon}\|_{q}^{q}}{\|u_{\varepsilon}\|_2^{q(1-\delta_{q,s})}}\bigg].
\end{split}
\ee
 Notice that, by \eqref{e5.18}-\eqref{e5.21}, there
exist positive constants $C_4, C_5$ and $C_6$ depending on $s$ and
$q$ such that
\be\label{e5.40}(\|u_{\varepsilon}\|^2)^{\frac{q\delta_{q,s}-2}{2^*_s-2}}\geq
C_4,~~\frac{1}{C_5}\leq
\|u_{\varepsilon}\|_{2^*_s}^{q\delta_{q,s}-2}\leq C_5.\ee and
\be\label{e5.41}
\frac{\|u_{\varepsilon}\|^{q}_q}{\|u_{\varepsilon}\|_2^{q(1-\gamma_{q,s})}}
 =\left\{\begin{array}{ll}
C_6\varepsilon^{3-\frac{3-2s}{2}q-sq(1-\gamma_{q,s})}=C_6,&\mbox{if}~~0<s<\frac{3}{4};\vspace{0.2cm}\\
C_6|\ln\varepsilon|^{\frac{q(\gamma_{q,s}-1)}{2}},&\mbox{if}~~s=\frac{3}{4};\vspace{0.2cm}\\
C_6\varepsilon^{3-\frac{3-2s}{2}q-\frac{(3-2s)q(1-\gamma_{q,s})}{2}},&\mbox{if}~~\frac{3}{4}<s<1;
\end{array} \right.
\ee Next, we show that
\be\label{e5.42}e^{(2^*_{s}-2)s\theta_{\varepsilon,\mu}}\geq
C\frac{\|u_{\varepsilon}\|^{2^*_s-2}_{2}}{a^{2^*_s-2}},\ee  for some
positive constant $C>0$. To obtain the estimation \eqref{e5.42}, as
in Case 1, we have to consider the three cases:
 (i) $0<s<\frac{3}{4};$ (ii) $s=\frac{3}{4};$ and (iii)
$\frac{3}{4}<s<1.$

\par When $0<s<\frac{3}{4}, $    it holds that
\be\label{e5.43}3-\frac{3-2s}{2}q-sq(1-\delta_{q,s})=0,\ee and from
\eqref{e5.39}-\eqref{e5.41} we have
\[e^{(2^*_s-2)s\theta_{\varepsilon,\mu}}\geq
\frac{C\|u_{\varepsilon}\|^{2^*_s-2}_{2}}{a^{2^*_s-2}} \left[C_1
-\mu\delta_{q,s}D(s,t)a^{\frac{(q-2)2t+2s(2^*_s-4)}{s2^*_{s}+2t-3}}
\frac{C_6}{C_4C_5}\right],\] and we see that inequality
\eqref{e5.42} holds only when
$\mu\delta_{q,s}D(s,t)a^{\frac{(q-2)2t+2s(2^*_s-4)}{s2^*_{s}+2t-3}}<
C_1C_4C_5 C_6^{-1}$. Thus, we have to give a more precise estimate,
let us   come back to \eqref{e5.39} and observe that by well-known
interpolation inequality, we have \be\label{e5.44}
 \begin{split}
 \frac{\|u_{\varepsilon}\|_{q}^{q}}
{(\|u_{\varepsilon}\|^2)^{\frac{q\delta_{q,s}-2}{2^*_s-2}}\|u_{\varepsilon}\|_{2^*_s}^{q\delta_{q,s}-2}\|u_{\varepsilon}\|_2^{q(1-\delta_{q,s})}}
 &\leq \frac{(\|u_{\varepsilon}\|^{2^*_s}_{2^*_s})^{\frac{q-2}{2^*_s-2}}(\|u_{\varepsilon}\|^2_2)^{\frac{2^*_s-q}{2^*_s-2}}}
 {(\|u_{\varepsilon}\|^2)^{\frac{q\delta_{q,s}-2}{2^*_s-2}}(\|u_{\varepsilon}\|_2)^{q(1-\delta_{q,s})}
 (\|u_{\varepsilon}\|_{2^*_s}^{2^*_s})^{\frac{q\delta_{q,s}-2}{2^*_s-2}}}\\
 &
 =\frac{(\|u_{\varepsilon}\|^{2^*_s}_{2^*_s})^{\frac{q(1-\delta_{q,s})}{2^*_s-2}}}
 {(\|u_{\varepsilon}\|^2)^{\frac{q\delta_{q,s}-2}{2^*_s-2}}}.
 \end{split}
\ee Therefore, by \eqref{e5.39} and \eqref{e5.44} we derive as

\be\label{e5.45}\begin{split}
e^{(2^*_{s}-2)s\theta_{\varepsilon,\mu}}&\geq\frac{\|u_{\varepsilon}\|_2^{2^*_s-2}(\|u_{\varepsilon}\|^2)^{\frac{q\delta_{q,s}-2}{2^*_s-2}}}{a^{2^*_s-2}\|u_{\varepsilon}\|_{2^*_s}^{2^*_s}}
\\
&\hspace{0.45cm}\times\left[(\|u_{\varepsilon}\|^2)^{\frac{2^*_s-q\delta_{q,s}}{2^*_s-2}}-\mu\delta_{q,s}D(s,t)a^{\frac{(q-2)2t+2s(2^*_s-4)}{s2^*_{s}+2t-3}}
 \frac{(\|u_{\varepsilon}\|^{2^*_s}_{2^*_s})^{\frac{q(1-\delta_{q,s})}{2^*_s-2}}}
 {(\|u_{\varepsilon}\|^2)^{\frac{q\delta_{q,s}-2}{2^*_s-2}}}\right].
 \end{split}
\ee  We observe that  the right hand side of \eqref{e5.45} is
positive provided that
\[\begin{split}
\mu\delta_{q,s}D(s,t)a^{\frac{(q-2)2t+2s(2^*_s-4)}{s2^*_{s}+2t-3}}&<\frac{\|u_{\varepsilon}\|^2}{(\|u_{\varepsilon}\|_{2^*_s}^{2^*_s})^{\frac{q(1-\delta_{q,s})}{2^*_s-2}}}\\
&=\frac{S^{\frac{3}{2s}}+O(\varepsilon^{3-2s})}{\left(S^{\frac{3}{2s}}+O(\varepsilon^3)\right)^{\frac{q(1-\delta_{q,s})}{2^*_s-2}}}
=S^{\frac{3[(2^*_s-2)-q(1-\delta_{q,s})]}{2s(2^*_s-2)}}+O(\varepsilon^{3-2s}).
\end{split}
\]   Therefore, if $0<s<\frac{3}{4}$ and
\be\label{e5.46}\mu\delta_{q,s}D(s,t)a^{\frac{(q-2)2t+2s(2^*_s-4)}{s2^*_{s}+2t-3}}<S^{\frac{3[(2^*_s-2)-q(1-\delta_{q,s})]}{2s(2^*_s-2)}},\ee
we see that \eqref{e5.42} holds for some constant $C>0$.

\par For the cases: $s=\frac{3}{4}$, and   $\frac{3}{4}<s<1$, we still have the following estimations  as  in {\bf Case 1},
\[\frac{\|u_{\varepsilon}\|^{q}_q}{\|u_{\varepsilon}\|_2^{q(1-\gamma_{q,s})}}
\leq
C_3\varepsilon^{3-\frac{3-2s}{2}q-sq(1-\gamma_{q,s})}|\ln\varepsilon|^{\frac{q(\gamma_{q,s}-1)}{2}}=o_{\varepsilon}(1);\]
and
\[\frac{\|u_{\varepsilon}\|^{q}_q}{\|u_{\varepsilon}\|_2^{q(1-\gamma_{q,s})}}
\leq
C\varepsilon^{3-\frac{3-2s}{2}q-\frac{(3-2s)q(1-\gamma_{q,s})}{2}}=o_{\varepsilon}(1),\]
respectively. Moreover, we derive  that
\be\label{e5.47}e^{(2^*_{s}-2)s\theta_{\widetilde{v}_{\varepsilon}}}\geq
 C\frac{\|u_{\varepsilon}\|^{2^*_{s}-2}_{2}}{a^{2^*_{s}-2}}\left[C_1
-\mu\delta_{q,s}D(s,t)a^{\frac{(q-2)2t+2s(2^*_s-4)}{s2^*_{s}+2t-3}}\frac{C_5}{C_4}o_{\varepsilon}(1)\right]
\geq\frac{C\|u_{\varepsilon}\|^{2^*_{s}-2}_{2}}{a^{2^*_{s}-2}}.\ee
\par To sum up, condition \eqref{e2.5} can ensure that \eqref{e5.33}, \eqref{e5.46} occur, so as to guarantee \eqref{e5.47} hold.
\par In what follows we focus on an upper estimate of
$\max_{\theta\in\R}\Psi^{\mu}_{v_{\varepsilon}}(\theta)$. We split
the argument into two  steps.
\par {\bf Step 1.} We estimate for
$\max_{\theta\in\R}\Psi^0_{v_{\varepsilon}}(\theta)$, where,
\[ \Psi^0_{v_{\varepsilon}}(\theta):= \frac{e^{2s\theta}}{2}\|v_{\varepsilon}\|^2-\frac{e^{2^*_{s}s\theta}}{2^*_{s}}
\int_{\R^{N}}|v_{\varepsilon}|^{2^*_{s}}dx.\] It is easy to see that
for every $v_{\varepsilon}\in S_{r,a}$ the function
$\Psi^0_{v_{\varepsilon}}(\theta)$ has a unique critical point
$\theta_{\varepsilon,0}$, which is a strict maximum point and is
given by \be\label{e5.48}
e^{s\theta_{\varepsilon,0}}=\left(\frac{\|{v}_{\varepsilon}\|^2}{\int_{\R^{N}}
|{v}_{\varepsilon}|^{2^*_{s}}dx}\right)^{\frac{1}{2^*_{s}-2}}.\ee
Using the fact that
\[\sup_{\theta\geq
0}\left(\frac{\theta^{2}}{2}a-\frac{\theta^{2^{*}_{s}}}{2^{*}_{s}}
b\right)
=\frac{s}{3}\left(\frac{a}{b^{2/2^*_s}}\right)^{\frac{2^*_s}{2^*_s-2}},\]
for any fixed ~$a,b>0.$ We can deduce by \eqref{e5.18},
\eqref{e5.19}, that \be\label{e5.49}\begin{split}
\Psi^0_{{v}_{\varepsilon}}(\theta_{\varepsilon,0})&=\frac{s}{3}
\left(\frac{\|{v}_{\varepsilon}\|^2}{(\int_{\R^{N}}|{v}_{\varepsilon}|^{2^*_{s}}dx)^{\frac{2}{2^*_s}}}\right)^{\frac{2^*_s}{2^*_s-2}}=\frac{s}{3}
\left(\frac{\|{u}_{\varepsilon}\|^2}{(\int_{\R^{N}}|{u}_{\varepsilon}|^{2^*_{s}}dx)^{\frac{2}{2^*_s}}}\right)^{\frac{2^*_s}{2^*_s-2}}\\
&=\frac{s}{3}\left(\frac{S^{\frac{3}{2s}}+O(\varepsilon^{3-2s})}{(S^{\frac{3}{2s}}+O(\varepsilon^3))^{\frac{2}{2^*_s}}}\right)^{\frac{2^*_s}{2^*_s-2}}=\frac{s}{3}S^{\frac{3}{2s}}+O(\varepsilon^{3-2s}).
\end{split}
\ee
\par {\bf Step 2.} We next estimate for
$\max_{\theta\in\R}\Psi^{\mu}_{v_{\varepsilon}}(t)$.  Recall
 \eqref{e3.3}, \eqref{e5.24} and H\"{o}lder inequality,  we have
  \be\label{e5.50}\begin{split}
&e^{(2^*_{s}-2)s\theta_{\varepsilon,\mu}}\\
&\leq \frac{2\max\left\{\|v_{\varepsilon}\|^2,\lambda
 \frac{3-2t}{4s}e^{(3-2t-2s)\theta_{\varepsilon,\mu}}\int_{\R^3}\phi^t_{v_{\varepsilon}}v_{\varepsilon}^2dx\right\}}{\|v_{\varepsilon}\|^{2^*_s}_{2^*_s}}\\
 &\leq \frac{2\max\left\{\|v_{\varepsilon}\|^2,
 \lambda\frac{3-2t}{4s}e^{(3-2t-2s)\theta_{\varepsilon,\mu}}\Gamma_t\|v_{\varepsilon}\|^{4\tau}_2\|v_{\varepsilon}\|_{2^*_s}^{4(1-\tau)}\right\}}{\|v_{\varepsilon}\|^{2^*_s}_{2^*_s}}\\
  &= \frac{2\max\left\{a^2\|u_{\varepsilon}\|^2\|u_{\varepsilon}\|^{2^*_s-2}_2,
 \lambda\frac{3-2t}{4s}e^{(3-2t-2s)\theta_{\varepsilon,\mu}}\Gamma_ta^4\|u_{\varepsilon}\|^{4(1-\tau)}_{2^*_s}\|u_{\varepsilon}\|_2^{2^*_s-4(1-\tau)}\right\}}{a^{2^*_s}\|u_{\varepsilon}\|^{2^*_s}_{2^*_s}}.
\end{split}
\ee From the estimations \eqref{e5.18}-\eqref{e5.19} and
\eqref{e5.50}, we see that the number $\theta_{\varepsilon,\mu}$ can
not go to $+\infty$, and there exists some $\theta^*\in\R$ such that
\be\label{e5.51}\theta_{\varepsilon,\mu}\leq\theta^*,~~~\mbox{for
all}~~\varepsilon,\mu>0.\ee Hence, by virtue of \eqref{e5.50},
\eqref{e5.51} and \eqref{e3.3} we derive to
 \be\label{e5.52}\begin{split}
 &\max_{\theta\in\R}\Psi^{\mu}_{v_{\varepsilon}}(\theta)\\
 &=\Psi^{\mu}_{v_{\varepsilon}}(\theta_{\varepsilon,\mu})=\Psi^{0}_{v_{\varepsilon}}(\theta_{\varepsilon,\mu})
+\frac{e^{(3-2t)\theta_{\varepsilon,\mu}}}{4}\lambda\int_{\R^3}\phi^t_{v_{\varepsilon}}v_{\varepsilon}^2dx-\mu
\frac{e^{q\gamma_{q,s}s\theta_{\varepsilon,\mu}}}{q}\int_{\R^N}|v_{\varepsilon}|^qdx\\
&\leq\sup_{\theta\in\R}\Psi^{0}_{v_{\varepsilon}}(\theta)+\frac{e^{(3-2t)\theta_{\varepsilon,\mu}}}{4}\lambda\int_{\R^3}\phi^t_{v_{\varepsilon}}v_{\varepsilon}^2dx-\mu
\frac{e^{q\gamma_{q,s}s\theta_{\varepsilon,\mu}}}{q}\int_{\R^N}|v_{\varepsilon}|^qdx\\
&\leq\Psi^{0}_{v_{\varepsilon}}(\theta_{v_{\varepsilon,0}})+C\lambda\left(\int_{\R^3}|v_{\varepsilon}|^{\frac{12}{3+2t}}dx\right)^{\frac{3+2t}{3}}-\frac{C\mu
a^{q(1-\gamma_{q,s})}}{q}\frac{\int_{\R^N}|u_{\varepsilon}|^qdx}{\|u_{\varepsilon}\|_2^{q(1-\gamma_{q,s})}}\\
&\leq\frac{s}{3}S^{\frac{3}{2s}}+O(\varepsilon^{3-2s})+C\frac{\lambda
a^4}{\|u_{\varepsilon}\|_2^4}\left(\int_{\R^3}|u_{\varepsilon}|^{\frac{12}{3+2t}}dx\right)^{\frac{3+2t}{3}}-\frac{C\mu
a^{q(1-\gamma_{q,s})}}{q}\frac{\int_{\R^N}|u_{\varepsilon}|^qdx}{\|u_{\varepsilon}\|_2^{q(1-\gamma_{q,s})}}\\
&\leq\frac{s}{3}S^{\frac{3}{2s}}+C_1\varepsilon^{3-2s}+C_2\lambda\frac{\left(\int_{\R^3}|u_{\varepsilon}|^{\frac{12}{3+2t}}dx\right)^{\frac{3+2t}{3}}}{\|u_{\varepsilon}\|_2^4}
-C_3\frac{\int_{\R^N}|u_{\varepsilon}|^qdx}{\|u_{\varepsilon}\|_2^{q(1-\gamma_{q,s})}}.\end{split}
\ee Next, we separate three cases:
\par  {\em Case 1:} $0<s<\frac{3}{4}.$  In this case, owing to
$2t+8s<9$, we get $p=\frac{12}{3+2t}>\frac{3}{3-2s}$, it following
from
 \eqref{e5.20}-\eqref{e5.21} and \eqref{e5.28}  that,
 \be\label{e5.53}\begin{split}
&\frac{s}{3}S^{\frac{3}{2s}}+C_1\varepsilon^{3-2s}+C_2\lambda\frac{\left(\int_{\R^3}|u_{\varepsilon}|^{\frac{12}{3+2t}}dx\right)^{\frac{3+2t}{3}}}{\|u_{\varepsilon}\|_2^4}
-C_3\frac{\int_{\R^3}|u_{\varepsilon}|^qdx}{\|u_{\varepsilon}\|_2^{q(1-\gamma_{q,s})}}\\
&=\frac{s}{3}S^{\frac{3}{2s}}+C_1\varepsilon^{3-2s}+C_2\lambda\frac{\varepsilon^{2t+4s-3}}{\varepsilon^{4s}}
-C_3 \\
&<\frac{s}{3}S^{\frac{3}{2s}},
\end{split}
\ee if we choose $\lambda=\varepsilon^{s}$.

\par  {\em Case 2:} $s=\frac{3}{4}.$  In this case,  we still have
$2t+8s=2t+6<9$, and also, $p=\frac{12}{3+2t}>\frac{3}{3-2s}$.
Moreover, $2+\frac{q(\gamma_{q,s}-1)}{2}=\frac{q(3-2s)}{4s}>0$,
hence
 \[\varepsilon^{2t+2s-3}\rightarrow0,~~\varepsilon^{3-2s}(\log\varepsilon)^2\rightarrow0,~~\mbox{and}~~|\ln\varepsilon|^{2+\frac{q(\gamma_{q,s}-1)}{2}}\rightarrow+\infty,\]
 when $\varepsilon\rightarrow 0^+.$ Consequently, if we choose  $\lambda=\varepsilon^{2s}$,  then  we
 have \be\label{e5.54}\begin{split}
&\frac{s}{3}S^{\frac{3}{2s}}+C_1\varepsilon^{3-2s}+C_2\lambda\frac{\left(\int_{\R^3}|u_{\varepsilon}|^{\frac{12}{3+2t}}dx\right)^{\frac{3+2t}{3}}}{\|u_{\varepsilon}\|_2^4}
-C_3\frac{\int_{\R^3}|u_{\varepsilon}|^qdx}{\|u_{\varepsilon}\|_2^{q(1-\gamma_{q,s})}}\\
&=\frac{s}{3}S^{\frac{3}{2s}}+C_1\varepsilon^{3-2s}+C_2\lambda\frac{\varepsilon^{2t+4s-3}}{
 \varepsilon^{4s}|\log\varepsilon|^2}
-C_3|\ln\varepsilon|^{\frac{q(\gamma_{q,s}-1)}{2}} \\
&=\frac{s}{3}S^{\frac{3}{2s}}+\frac{1}{(\log\varepsilon)^2}\left[C_1\varepsilon^{3-2s}(\log\varepsilon)^2+C_2\varepsilon^{2t+2s-3}
-C_3|\ln\varepsilon|^{2+\frac{q(\gamma_{q,s}-1)}{2}}\right]\\
&<\frac{s}{3}S^{\frac{3}{2s}},
\end{split}
\ee when $\varepsilon>0$ small enough.

\par  {\em Case 3:} $\frac{3}{4}<s<1.$ In this case,   using
the fact that $2t+2s>3, q>2+\frac{4s}{3}$, we can obtain  the
inequality  by a direct computation,
\[3-\frac{3-2s}{2}q-\frac{(3-2s)q(1-\gamma_{q,s})}{2}<3-2s.\]
Thus,   from \eqref{e5.20}-\eqref{e5.21} and \eqref{e5.28}, letting
$\lambda=\varepsilon^{6-4s}$  we derive that
 \be\label{e5.55}\begin{split}
&\frac{s}{3}S^{\frac{3}{2s}}+C_1\varepsilon^{3-2s}+C_2\lambda\frac{\left(\int_{\R^3}|u_{\varepsilon}|^{\frac{12}{3+2t}}dx\right)^{\frac{3+2t}{3}}}{\|u_{\varepsilon}\|_2^4}
-C_3\frac{\int_{\R^3}|u_{\varepsilon}|^qdx}{\|u_{\varepsilon}\|_2^{q(1-\gamma_{q,s})}}\\
&=\frac{s}{3}S^{\frac{3}{2s}}+C_1\varepsilon^{3-2s}+C_2\left\{\begin{split}
\lambda\frac{\varepsilon^{2t+4s-3}}{\varepsilon^{6-4s}},~~&\mbox{if}~~\frac{12}{3+2t}>\frac{3}{3-2s},\\
\lambda\frac{\varepsilon^{2t+4s-3}|\ln\varepsilon|^{\frac{3+2t}{3}}}{\varepsilon^{6-4s}},~~&\mbox{if}~~\frac{12}{3+2t}=\frac{3}{3-2s},\\
\lambda\frac{\varepsilon^{2(3-2s)}}{\varepsilon^{6-4s}},~~&\mbox{if}~~\frac{12}{3+2t}<\frac{3}{3-2s}\\
\end{split}
\right.\\
&\hspace{0.45cm}-C_3\varepsilon^{3-\frac{3-2s}{2}q-\frac{(3-2s)q(1-\gamma_{q,s})}{2}-(3-2s)}\\
&<\frac{s}{3}S^{\frac{3}{2s}}.
\end{split}
\ee

\par Since $ {v}_{\varepsilon} \in S_{r,a}, $ from Lemma \ref{Lemma 5.1} we can take  $\theta_1<
0$ and $\theta_2> 0$ such that $\theta_1\star {v}_{\varepsilon} \in
\mathcal {A}_a$ and  $I_{\mu}(\theta_2\star {v}_{\varepsilon}) < 0,$
respectively. Then we can define a path \[\gamma_{{v}_{\varepsilon}
} : t\in [0, 1] \mapsto ((1 - t)\theta_1 + t\theta_2)\star
{v}_{\varepsilon} \in \Gamma_a.\] To sum up, by the estimations
\eqref{e5.52}-\eqref{e5.55}, we can derive that
 \be\label{e5.56}
 c_{r,\mu}(a) \leq \max_{t\in [0,1]}I_{\mu}(\gamma_{{v}_{\varepsilon}
 }(t))\leq \max_{\theta\in\R}\Psi^{\mu}_{{v}_{\varepsilon}}(\theta)<\frac{s}{3}S^{\frac{3}{2s}},\ee
for $\varepsilon>0$ small   enough, which is the desired result. \ep

\bl\label{Lemma 5.8}   Let $\{u_n\}$ be the $(PS)$ sequence in
$S_{r,a}$ at level $c_{\mu}(a),$ with $c_{\mu}(a)
<\frac{s}{3}S^{\frac{3}{2s}}$, assume that $u_n\rightharpoonup u,$
then, $u \not\equiv 0.$ \el
 \bp
Arguing by contradiction, we suppose that $u \equiv 0.$ Noticing
that $\{u_n\}$  is bounded in $H^s_{rad}(\R^3),$ going to a
subsequence, we may assume that
$\|(-\Delta)^{\frac{s}{2}}u_n\|_2^2\rightarrow \ell\geq0.$ By Lemma
\ref{Lemma 3.6}, $u_n\rightarrow 0$ in $L^p(\R^3), \forall p\in
(2,2^*_s)$. From Proposition \ref{Prop 5.5} and Lemmas \ref{Lemma
3.3},\ref{Lemma 3.6}, we have $P_{\mu}(u_n)\rightarrow0$  such that,
 \[\begin{split}
  \int_{\R^3}|u_n|^{2^*_{s}}dx&= \|(-\Delta)^{\frac{s}{2}}u_n\|^2_2+\frac{3-2t}{4s}\lambda\int_{\R^3}\phi_{u_n}^tu_n^2dx-\mu
\delta_{q,s}\int_{\R^3}|u_n|^qdx\\
&=\|(-\Delta)^{\frac{s}{2}}u_n\|^2_2+o_n(1)\\
&=\ell+o_n(1),
\end{split}
\] as $n\rightarrow\infty$. Then, using Sobolev's
inequality, one has $\ell \geq S\ell^{\frac{2}{2^*_s}},$ and so,
either $\ell\geq S^{\frac{3}{2s}}$ or $\ell = 0.$ In the case
$\ell\geq S^{\frac{3}{2s}}$, from $I_{\mu} (u_n)\rightarrow
c_{\mu}(a),
 P_{\mu}(u_n)\rightarrow 0$, we know
\[\begin{split}
&c_{\mu}(a) + o_n(1)\\
 &= I_{\mu} (u_n)=I_{\mu}(u_n)-\frac{1}{s2^*_s}P_{\mu}(u_n)\\
 &=\frac{s}{3}\|(-\Delta)^{\frac{s}{2}}u_n\|_2^2+\lambda\frac{s2^*_s+2t-3}{4s2^*_s}\int_{\R^3}\phi_{u_n}^t|u_n|^2dx-\mu\frac{2^*_s-q}{q2^*_s}
  \int_{\R^3}|u_n|^qdx+o_n(1)\\
  &=\frac{s}{3}\ell+o_n(1)
\end{split}\]
which means $c_{\mu}(a)= \frac{s}{3}\ell$, that is $c_{\mu}(a)\geq
\frac{s}{3} S^{\frac{3}{2s}}$, which contradicts the assumption
$c_{\mu}(a)< \frac{s}{3} S^{\frac{3}{2s}}$. In the case $\ell = 0$,
one has \[
 \|(-\Delta)^{\frac{s}{2}}u_n\|^2_2\rightarrow 0,~~\int_{\R^3}|u_n|^{2^*_{s}}dx\rightarrow 0,\] and combining with \[\int_{\R^3}\phi_{u_n}^tu_n^2dx\rightarrow 0,~~
 \int_{\R^3}|u_n|^qdx\rightarrow0,\]  we have, $I_{\mu} (u_n)\rightarrow 0$, which is absurd since $c_{\mu}(a)>0.$ Therefore, $u\not\equiv0$.
\ep

\bl\label{Lemma 5.9}  Let $\{u_n\}$ be the $(PS)$ sequence in
$S_{r,a}$ at level $c_{\mu}(a)$, with  $c_{\mu}(a)< \frac{s}{3}
S^{\frac{3}{2s}}$, assume that $P_{\mu}(u_n)\rightarrow 0$ when
$n\rightarrow\infty$, and $\lambda<\lambda^*_1$ small. Then one of
the following alternatives holds:
\par (i) either going to a subsequence $u_n\rightharpoonup u$ weakly
in
 $H^s_{rad}(\R^3)$, but not strongly, where $u \not\equiv0$ is a solution to
 \be\label{e5.57}(-\Delta)^su +\lambda\phi^t_u u= \alpha
u+\mu|u|^{q-2}u+|u|^{2^*_s-2}u,~~~\mbox{in}~~\R^3,\ee where
$\alpha_n\rightarrow \alpha< 0,$ and
\[I_{\mu}(u)<c_{\mu}(a)-\frac{s}{3}S^{\frac{3}{2s}};\]
\par (ii) or passing to a subsequence $u_n\rightarrow u$ strongly in $H^s_{rad}(\R^3), I_{\mu}(u) =c_{\mu}(a)$ and $u$ is
a solution of \eqref{e1.5}-\eqref{e1.6} for some $\alpha < 0.$
 \el
\bp  By  Lemma \ref{Lemma 5.6},  we have that $\{u_n\}\subset
S_{r,a}$ is a bounded $(PS)$ sequence for $I_{\mu}$ in
$H^s_{rad}(\R^3)$, and so $u_n\rightharpoonup u$ in
$H^s_{rad}(\R^3)$ for some $u$. By the Lagrange multiplier
principle, there exists $\{\alpha_n\}\subset\R$ satisfying
\be\label{e5.58}
 \begin{split}
 & \int_{\R^3}(-\Delta)^{\frac{s}{2}}u_n(-\Delta)^{\frac{s}{2}}\varphi dx-\alpha_n\int_{\R^3}u_n\varphi dx+\lambda\int_{\R^3}\phi_{u_n}^tu\varphi dx-\mu\int_{\R^3}|u_n|^{q-2}u_n\varphi dx
  \\
&\hspace{0.45cm}- \int_{\R^3}|u_n|^{2^*_{s}-2}u_n\varphi
dx=o_n(1)\|\varphi\|,\end{split} \ee for any $\varphi\in
H^s_{rad}(\R^3).$ Moreover,  one has
$\lim_{n\rightarrow\infty}\alpha_n=\alpha < 0.$ Letting
$n\rightarrow\infty$ in \eqref{e5.58}, we have \[
\int_{\R^3}(-\Delta)^{\frac{s}{2}}u(-\Delta)^{\frac{s}{2}}\varphi
dx+\lambda\int_{\R^3}\phi_{u}^tu\varphi
dx-\mu\int_{\R^3}|u|^{q-2}u\varphi
dx-\int_{\R^3}|u|^{2^*_{s}-2}u\varphi dx-\alpha\int_{\R^3}u\varphi
dx=0,\] which implies that $u$ solves the equation
\be\label{e5.59}(-\Delta)^su +\lambda\phi^t_u u= \alpha
u+\mu|u|^{q-2}u+|u|^{2^*_s-2}u,~~~\mbox{in}~~\R^3,\ee and we have
the Pohoz\u{a}ev identity $P_{\mu}(u) = 0.$
\par Let $v_n = u_n-u,$ then
$v_n\rightharpoonup0$ in $H^s_{rad}(\R^3).$ According to Brezis-Lieb
lemma \cite{Willem} and Lemma \ref{Lemma 3.3}, one has
\be\label{e5.60}\|(-\Delta)^{\frac{s}{2}}u_n\|^2_2=\|(-\Delta)^{\frac{s}{2}}u\|^2_2+\|(-\Delta)^{\frac{s}{2}}v_n\|^2_2+o_n(1),~~~\|u_n\|^{2^*_s}_{2^*_s}=\|u\|^{2^*_s}_{2^*_s}+\|v_n\|^{2^*_s}_{2^*_s}+o_n(1),\ee
and
\be\label{e5.61}\int_{\R^3}\phi_{u_n}^tu_n^2dx=\int_{\R^3}\phi_{u}u^2dx+o_n(1),~~\|u_n\|^q_q=\|u\|^q_q+\|v_n\|^q_q+o_n(1).\ee
Then, from $P_{\mu}(u_n)\rightarrow   0, u_n \rightarrow u$ in
$L^p(\R^3)$, one can derive that
\[\begin{split}
&\|(-\Delta)^{\frac{s}{2}}u\|_2^2+\|(-\Delta)^{\frac{s}{2}}v_n\|_2^2+\frac{3-2t}{4s}\lambda\int_{\R^3}\phi_u^tu^2dx\\
&=\mu \delta_{q,s}\int_{\R^3}|u|^qdx
 +\int_{\R^3}|u|^{2^*_{s}}dx+\int_{\R^3}|v_n|^{2^*_{s}}dx+o_n(1).
 \end{split}\]
By $P_{\mu}(u) = 0,$ we have
\be\label{e5.62}\|(-\Delta)^{\frac{s}{2}}v_n\|^2_2=\int_{\R^3}|v_n|^{2^*_{s}}dx+o_n(1).\ee
Passing to a subsequence, we may assume that
\be\label{e5.63}\lim_{n\rightarrow\infty}\|(-\Delta)^{\frac{s}{2}}v_n\|^2_2=\lim_{n\rightarrow\infty}\int_{\R^3}|v_n|^{2^*_{s}}dx=\ell\geq0.\ee
Then, it follows from Sobolev's inequality that $\ell \geq
S\ell^{\frac{2}{2^*_s}},$ and so, either $\ell\geq S^{\frac{3}{2s}}$
or $\ell = 0.$ In the case $\ell\geq S^{\frac{3}{2s}}$, from
$I_{\mu} (u_n)\rightarrow c_{\mu}(a),
 P_{\mu}(u_n)\rightarrow 0$, we know
\be\label{e5.64}\begin{split}
c_{\mu}(a)&=\lim_{n\rightarrow\infty}I_{\mu}(u_n)=\lim_{n\rightarrow\infty}\left\{I_{\mu}(u)+\frac{1}{2}\|v_n\|^2-\frac{1}{2^*_s}\int_{\R^3}|v_n|^{2^*_s}dx+o_n(1)\right\}\\
 &=I_{\mu}(u)+\frac{s}{3}\ell\geq I_{\mu}(u)+\frac{s}{3} S^{\frac{3}{2s}}
 \end{split}\ee
 which means that item (i) holds.
\par If $\ell = 0,$ then $\|u_n -u\| = \|v_n \|\rightarrow 0,$ one has $u_n\rightarrow u$ in
 $D^{s,2}(\R^3),$ and so $u_n\rightarrow u$ in $L^{2^*_s}(\R^3).$ To prove that $u_n\rightarrow u$ in
 $H^s_{rad}(\R^3),$ it remains only to prove that $u_n\rightarrow u$ in $ L^2(\R^3).$ Fix $\psi = u_n - u$ as a
test function in \eqref{e5.58}, and $u_n- u$ as a test function of
 \eqref{e5.59}, we deduce that
\be\label{e5.65}
\begin{split}
 & \int_{\R^3}|(-\Delta)^{\frac{s}{2}}(u_n-u)|^2 dx-\int_{\R^3}(\alpha_nu_n-\alpha u)(u_n-u) dx+\lambda\int_{\R^3}(\phi_{u_n}^tu_n-\phi_{u}^tu)(u_n-u) dx\\
 &\hspace{0.45cm}=\mu\int_{\R^3}(|u_n|^{q-2}u_n-|u|^{q-2}u)(u_n-u) dx
 +\int_{\R^3}(|u_n|^{2^*_{s}-2}u_n-|u|^{2^*_{s}-2}u)(u_n-u)
 dx+o_n(1).
  \end{split}
  \ee
Passing the limit in \eqref{e5.65} as $ n\rightarrow\infty$, we have
\[0=\lim_{n\rightarrow\infty}\int_{\R^3}(\alpha_nu_n-\alpha u)(u_n-u) dx=\lim_{n\rightarrow\infty}\alpha\int_{\R^3}(u_n- u)^2
dx,\]and then $u_n\rightarrow u$ in $L^2(\R^3)$. Therefore,   item
(ii) holds. \ep

  \par Now, we are ready to complete the proof of Theorem \ref{Theorem 2.2}.
\vskip0.1in
\par\noindent {\em Proof of Theorem  \ref{Theorem 2.2}.} Let $\lambda<\Lambda^*:=\min\{\lambda^*_1,\lambda^*_2\}.$
By virtue of Lemmas \ref{Lemma 5.1}-\ref{Lemma 5.2},\ref{Lemma
5.6}-\ref{Lemma 5.7},  Propositions \ref{Prop 5.3}-\ref{Prop 5.5},
there exists a bounded
 $(PS)_{c_{\mu}(a)}$-sequence $\{u_n\}\subset S_{r,a}$, with
 $c_{\mu}(a)<\frac{s}{3}S^{\frac{3}{2s}}$, and $u\in
 H^s_{rad}(\R^3)$ such that one of the
 alternatives  of Lemma \ref{Lemma 5.9} holds.  We assert that (i) of Lemma \ref{Lemma 5.9} can not occur.
Indeed, suppose by contradiction that,  item (i) holds, then $u$ is
a nontrivial solution of \eqref{e5.57}, and by Lemma \ref{Lemma 5.9}
and Lemma \ref{Lemma 5.7}, we have
\[I_{\mu}(u)<c_{\mu}(a)-\frac{s}{3}S^{\frac{3}{2s}}<0.\]
On the other hand, we have
\[\begin{split}
    I_{\mu}(u)&=I_{\mu}(u)-\frac{1}{2s} P_{\mu}(u)\\
    &=\frac{2s+2t-3}{8}\lambda\int_{\R^3}\phi^t_uu^2dx+\frac{q\delta_{q,s}-2}{2q}\mu\int_{\R^3}|u|^qdx+\frac{s}{3}\int_{\R^3}|u|^{2^*_s}dx\\
    &\geq0,\end{split}
    \]
which leads  to a contradiction. Therefore, $u_n\rightarrow u$
strongly in $H^s_{rad}(\R^3)$ with $I_{\mu}(u) =c_{\mu}(a)$, and $u$
is a solution of \eqref{e1.5}-\eqref{e1.6} for some $\alpha < 0.$
Moreover, $u(x)>0$ in $\R^3$. In fact, we note that all the
calculations above can be repeated word by word, replacing $I_{\mu}$
with the functional \be\label{e5.66}
I_{\mu}^{+}(u)=\frac{1}{2}\int_{\R^{3}}|(-\Delta)^{\frac{s}{2}}u|^{2}dx+\frac{\lambda
}{4}\int_{\R^{3}}\phi_u^tu^{2}dx-\frac{\mu}{q}\int_{\R^{3}}
    |u^{+}|^{q}dx-\frac{1}{2^{*}_{s}}\int_{\R^{3}}|u^{+}|^{2^{*}_{s}}dx.
\ee Then $u$ is the critical point of $I_{\mu}^+$ restricted on the
set $S_{r,a}$, it solves the equation
  \be\label{e5.67}(-\Delta)^su +\lambda\phi^t_u u= \alpha
u+\mu|u^+|^{q-2}u+|u^+|^{2^*_s-2}u.~~~\mbox{in}~~\R^3,\ee
  Using $u^{-}= \min\{u,0\}$ as a test function in
\eqref{e5.67}, in view of  $(a-b)(a^{-}-b^{-})\geq|a^{-}-b^{-}|^2,
\forall  a,b\in\R,$ we conclude that
\[\begin{split}
   & \|(-\Delta ^{\frac{s}{2}})u^{-}\|_2^2=\iint_{\R^{6}}\frac{|u^-(x)-u(y)|^2}{|x-y|^{3+2s}}dxdy \\
   &\leq  \|(-\Delta
    ^{\frac{s}{2}})u^{-}\|_2^2+\lambda\int_{\R^3}\phi^t_u |u^-|^2dx-\alpha\int_{\R^3}|u^-|^2dx\\
    &\leq  \iint_{\R^{6}}\frac{(u(x)-u(y))((u^{-}(x)-u^{-}(y))}{|x-y|^{3+2s}}dxdy +\lambda\int_{\R^3}\phi^t_u |u^-|^2dx-\alpha\int_{\R^3}|u^-|^2dx\\
    &=0.
\end{split}\]
Thus, $u^{-} =0$ and $u \geq 0, \forall x\in\R^3$, is a solution of
\eqref{e5.67}. By the regularity result  \cite{YYZ} we know that $u
\in L^{\infty}(\R^3)\cap C^{0,\alpha}(\R^3)$ for some
$\alpha\in(0,1)$. Suppose  $u(x_0) = 0$ for some $x_0\in\R^3$, then
$(-\Delta)^s u (x_0) = 0$ and by the definition of $(-\Delta)^s$, we
have \cite{NPV}:
\[(-\Delta)^s  u(x_0) =-\frac{C_{s}}{2}\int_{\R^3}\frac{u(x_0+y)+u(x_0-y)-2u(x_0)}{|y|^{3+2s}}dy.\]Hence,
$\int_{\R^3}\frac{u(x_0+y)+u(x_0-y)}{|y|^{3+2s}}dy=0,$ which implies
  $u\equiv0$, a contradiction. Thus, $u(x)>0, \forall x\in\R^3.$ \qed

\vskip0.2in

\section{Proof of Theorem \ref{Theorem 2.3}}

In this section, we deal with the $L^2$-supercritical case
$2+\frac{4s}{3}<q<2^*_s$, when   parameter $\mu>0$ large. In view of
$\frac{3(q-2)}{2s}>2$,  the truncated functional $I_{\mu,\tau}$
defined in Section 4 is still unbounded from below on $S_{r,a}$, and
 the truncation technique can not be applied to study problem
\eqref{e1.5}-\eqref{e1.6}.

To overcome this difficulty,  as in Section 5 we introduce the
transformation (e.g. \cite{Soave1}): \be\label{e6.1}
  (\theta\star u)(x):=e^{\frac{3\theta}{2}}u(e^{\theta}x),~~~~
x\in\RN,~~\theta\in\R,\ee  and the auxiliary functional
\be\label{e6.2}\begin{split}
 I(u,\theta)=I_{\mu} ((\theta\star u))=&\frac{e^{2s\theta}}{2}\|u\|^2+\frac{\lambda e^{(3-2t)\theta}}{4}\int_{\R^3}\phi^t_uu^2dx
 -\frac{\mu}{q}e^{q\delta_{q,s}s\theta}\int_{\R^3}|u|^qdx\\
&-\frac{1}{2^*_{s}}e^{\frac{3(2^*_{s}-2)}{2}\theta}\int_{\R^3}|u|^{2^*_{s}}dx.
\end{split}\ee
From Lemmas \ref{Lemma 5.1}, \ref{Lemma 5.2}, we have the
 the mountain pass level value
$c_{\mu}(a)$ by
\[
c_\mu(a):=\inf_{\gamma\in
\Gamma}\max_{t\in[0,1]}I_{\mu}(\gamma(t))>0,
\]
where
\[{\Gamma}_a=\{\gamma\in C([0,1],
S_{r,a}):~\gamma(0)\in A_a, \gamma(1)\in I_{\mu}^0\}.\]

\par In what follows, we set $g(t)=\mu|t|^{q-2}t+|u|^{2^*_s-2}u$, for any
$t\in \R$. From Propositions \ref{Prop 5.4},\ref{Prop 5.5}, we know
that there exist a  $(PS)_{c_{\mu}(a)}$-sequence $\{u_n\}\subset
S_{r,a}$ satisfying
\[
I_{\mu}(u_n)\rightarrow c_\mu(a), ~~
\|I_{\mu}'|_{S_{r,a}}(u_n)\|\rightarrow0 ~ \mbox{ and}~
P_{\mu}(u_n)\rightarrow0,~~\mbox{as}~~ n\rightarrow\infty,
\]
where
\[\begin{split}
P_{\mu}(u_n)=&s\int_{\R^3}|(-\Delta)^{\frac{s}{2}}u_n|^2dx+\frac{3-2t}{4}\lambda\int_{\R^3}\phi_u^tu^2dx+3\int_{\R^3}G(u_n)dx-\frac{3}{2}\int_{\R^3}g(u_n)u_ndx.
\end{split}\]
 Similar to the Section 5, setting the functional $\Psi(v): H_{rad}^s(\R^3)\rightarrow\R$ given by
\[
\Psi(v)=\frac{1}{2}\int_{\R^3}|v|^2dx,
\]
it follows that $S_{r,a}=\Psi^{-1}(\{\frac{a^2}{2}\})$, and  by
Proposition 5.12 in \cite{Willem}, there exists $\alpha_n\in \R$
such that
\[
\|I_{\mu}'(u_n)-\alpha_n\Psi'(u_n)\|\rightarrow 0,~~\mbox{as}~~
n\rightarrow\infty.
\]
 That is,  we have
\begin{equation}\label{e6.3}
(-\Delta)^su_n+\lambda\phi_{u_n}^tu_n-g(u_n)=\alpha_n u_n+o_n(1)
~~\mbox{in}~H^{-s}_{rad}(\R^3).
\end{equation}
 Therefore, for any $\varphi \in H^s_{rad}(\R^3)$ , one has
\begin{equation}\label{e6.4}
 \int_{\R^3}(-\Delta)^\frac{s}{2}u_n(-\Delta)^\frac{s}{2}\varphi
dx+\lambda \int_{\R^3}\phi_{u_n}^t u_n\varphi dx- \int_{\R^3}g(u_n)
\varphi dx =\alpha_n \int_{\R^3}u_n \varphi dx+o_n(1).
 \end{equation}
 \par In the sequel, we study the asymptotical behavior of the mountain pass level
value $c_{\mu}(a)$ as $\mu\rightarrow+\infty,$ and the properties of
the $(PS)_{c_{\mu}(a)}$-sequence $\{u_n\}\subset S_{r,a}$ as
$n\rightarrow+\infty.$

\bl\label{Lemma 6.1} The limit
$\lim_{\mu\rightarrow+\infty}c_{\mu}(a) = 0$ holds. \el \bp Recall
Lemmas \ref{Lemma 5.1}, \ref{Lemma 5.2}, we see that for fixed
$u_0\in S_{r,a}$, there exists two constants $\theta_1,\theta_2$
satisfying $\theta_1<0<\theta_2$ such that $u_1:=\theta_1\star
u_0\in A$ and $I_{\mu}(u_2)<0$.  Then we can define a path
\[\eta_{0}: \tau\in[0,1]\rightarrow
((1-\tau)\theta_1+\tau\theta_2)\star u_0\in \Gamma_a.\] Thus, we
have
\[\begin{split}
c_{\mu}(a)&\leq \max_{t\in[0,1]}I_{\mu}(\eta_0(t))\\
&\leq\max_{r\geq0}\left\{\frac{r^{2s}}{2}\|u_0\|^2+\frac{r^{3-2t}}{4}\lambda\int_{\R^3}\phi^t_{u_0}u_0^2dx-\frac{\mu}{q}r^{\frac{3q-6}{2}}\int_{\R^3}|u_0|^qdx\right\}\\
&:=\max_{r\geq0}h(r).
\end{split}
\]Note that $\frac{3q-6}{2}>2s>3-2t,$ we have that $\lim_{r\rightarrow0^+}h(r)=0^+, \lim_{r\rightarrow +\infty}h(r)=-\infty,$ and so, there exists a
unique maximum point  $r_0>0$ such that
$\max_{r\geq0}h(r)=h(r_0)>0$. Hence,  we  distinguish two  cases:
$r_0 \geq 1$ and $0\leq r_0 < 1.$
\par  If $r_0 \geq 1,$ we have by $2s+2t>3$, that
\[\begin{split}
\max_{t\in[0,1]}I_{\mu}(\eta_0(t))&\leq h(r_0)\\
&\leq \frac{r_0^{2s}}{2}\|u_0\|^2+\frac{r_0^{2s}}{4}\lambda\int_{\R^3}\phi^t_{u_0}u_0^2dx-\frac{\mu}{q}r_0^{\frac{3q-6}{2}}\int_{\R^3}|u_0|^qdx\\
&\leq\max_{r\geq0}\left\{2\max\left\{\frac{1}{2}\|u_0\|^2,\frac{\lambda}{4}\int_{\R^3}\phi^t_{u_0}u_0^2dx\right\}r^{2s}-\frac{\mu}{q}r^{\frac{3q-6}{2}}\int_{\R^3}|u_0|^qdx\right\}\\
&=2a(r_{max})^{2s}-\frac{\mu b}{q}(r_{max})^{\frac{3q-6}{2}}\\
&=\frac{2a(3q-6-4s)}{3q-6}\left[\frac{8qsa}{\mu
b(3q-6)}\right]^{\frac{4s}{3q-6-4s}},\end{split}
\]where
\[r_{max}=\left[\frac{8qsa}{\mu
b(3q-6)}\right]^{\frac{4s}{3q-6-4s}},~a=\max\left\{\frac{1}{2}\|u_0\|^2,\frac{\lambda}{4}\int_{\R^3}\phi^t_{u_0}u_0^2dx\right\},~b=\int_{\R^3}|u_0|^qdx.\]
Therefore, for $2+\frac{4s}{3} <q < 2^*_s$, we have a positive
constant $\widetilde{C}$ independent of $\mu$ such that
\[\gamma_{\mu}(a)\leq
\widetilde{C}
\mu^{-\frac{4s}{3q-6-4s}}\rightarrow0,~~~~\mbox{as}~~\mu\rightarrow+\infty.\]
  \par If $0 \leq r_0 < 1,$ we infer to
\[\begin{split}
\max_{t\in[0,1]}I_{\mu}(\eta_0(t))&\leq
\frac{r_0^{2s}}{2}\|u_0\|^2+\frac{r_0^{3-2t}}{4}\int_{\R^3}\phi^t_{u_0}u_0^2dx-\frac{\mu}{q}r_0^{\frac{3q-6}{2}}\int_{\R^3}|u_0|^qdx\\
&\leq\max_{r\geq0}\left\{2\max\left\{\frac{1}{2}\|u_0\|^2,\frac{1}{4}\int_{\R^3}\phi^t_{u_0}u_0^2dx\right\}r^{3-2t}-\frac{\mu}{q}r^{\frac{3q-6}{2}}\int_{\R^3}|u_0|^qdx\right\}\\
&=2a(\widetilde{r}_{max})^{3-2t}-\frac{\mu b}{q}(\widetilde{r}_{max})^{\frac{3q-6}{2}}\\
&=\frac{2a(3q+4t-12)}{3q-6}\left[\frac{4qa(3-2t)}{\mu
b(3q-6)}\right]^{\frac{2(3-2t)}{3q+4t-12}},\end{split}
\]where
\[\widetilde{r}_{max}=\left[\frac{4qa(3-2t)}{\mu
b(3q-6)}\right]^{\frac{2}{3q+4t-12}}.\]Since $2+\frac{4s}{3} <q <
2^*_s$, and $2s+2t>3$, we can deduce that $3q+4t-12>0$, then there
exists  a positive constant $C_1$ independent of $\mu$ such that
\[c_{\mu}(a)\leq
 C_1
\mu^{-\frac{2(3-2t)}{3q+4t-12}}\rightarrow0,~~~~\mbox{as}~~\mu\rightarrow+\infty.\]
This completes the proof.
 \ep

\begin{lemma}\label{Lemma 6.2} There exists a constant $C=C(q,s)>0$ such that
\[
\limsup_{n\rightarrow\infty}\int_{\R^3}G(u_n)dx\leq Cc_\mu(a),
\]
\[
\limsup_{n\rightarrow\infty}\int_{\R^3}g(u_n)u_ndx\leq Cc_\mu(a),
\]
and
\[\limsup_{n\rightarrow\infty} \int_{\R^3}\phi^t_{u_n}u_n^2dx\leq Cc_\mu(a),~
\limsup_{n\rightarrow\infty}\int_{\R^3}|(-\Delta)^{\frac{s}{2}}u_n|^2dx\leq
Cc_\mu(a).
\] \end{lemma}
\begin{proof}
Since $I_{\mu}(u_n)\rightarrow c_\mu(a)$ and
$P_{\mu}(u_n)\rightarrow0$ as $n\rightarrow\infty$, we have
\be\label{e6.5}\begin{split} &3c_\mu(a)+o_n(1)=3I_{\mu}(u_n)+P_{\mu}(u_n)\\
&=\frac{3+2s}{2}\int_{\R^3}|(-\Delta)^{\frac{s}{2}}u_n|^2dx+\lambda\frac{3-t}{2}\int_{\R^3}\phi^t_{u_n}u_n^2dx-\frac{3}{2}\int_{\R^3}g(u_n)u_ndx\\
&=\frac{3+2s}{2}\left(2c_{\mu}(a)-\frac{\lambda}{2}\int_{\R^3}\phi^t_{u_n}u_n^2dx+2\int_{\R^3}G(u_n)dx+o_n(1)\right)\\
&\hspace{0.45cm}+\lambda\frac{3-t}{2}\int_{\R^3}\phi^t_{u_n}u_n^2dx-\frac{3}{2}\int_{\R^3}g(u_n)u_ndx\\
&=(3+2s)\left[c_\mu(a)+\int_{\R^3}G(u_n)dx+o_n(1)\right]-\frac{3}{2}\int_{\R^3}g(u_n)u_ndx-\lambda\frac{2t+2s-3}{4}\int_{\R^3}\phi^t_{u_n}u_n^2dx.
\end{split}
\ee Hence,
\[\begin{split}
2sc_\mu(a)+o_n(1)&=\lambda\frac{2t+2s-3}{4}\int_{\R^3}\phi
_{u_n}u_n^2dx+\frac{3}{2}\int_{\R^3}g(u_n)u_ndx-(3+2s)\int_{\R^3}G(u_n)dx\\
&\geq\frac{3q}{2}\int_{\R^3}G(u_n)dx-(3+2s)\int_{\R^3}G(u_n)dx\\
&=\frac{3q-2(3+2s)}{2}\int_{\R^3}G(u_n)dx,
\end{split}
\]
which implies that
\begin{equation}\label{e6.6}
\limsup_{n\rightarrow\infty}\int_{\R^3}G(u_n)dx\leq
\frac{4s}{3q-2(3+2s)}c_\mu(a) \leq Cc_\mu(a)
\end{equation}
and then
\begin{equation}\label{e6.7}
\limsup_{n\rightarrow\infty}\int_{\R^3}g(u_n)u_ndx\leq Cc_\mu(a).
\end{equation}
Then, from \eqref{e6.5}-\eqref{e6.7}, we have

\begin{equation}\label{e6.8}
\begin{split}
&\limsup_{n\rightarrow\infty}\left\{
\frac{3+2s}{2}\int_{\R^3}|(-\Delta)^{\frac{s}{2}}u_n|^2dx+\lambda\frac{3-t}{2}\int_{\R^3}\phi
_{u_n}u_n^2dx\right\}\\
&=\limsup_{n\rightarrow\infty}\left\{
3c_\mu(a)+\frac{3}{2}\int_{\R^3}g(u_n)u_ndx+o_n(1)\right\}\leq C
c_{\mu}(a).
\end{split}
\end{equation}
Consequently,  the proof is completed.
\end{proof}

\begin{lemma}\label{Lemma 6.3} There exists $\mu^\ast_1:=\mu^\ast_1(a)>0$ such that $u\not\equiv0$ for all $\mu\geq\mu^*_1$.
 \end{lemma}

\begin{proof} From Lemma \ref{Lemma 5.6}, we know that  $\{u_n\}$ is bounded in $H^s_{rad}(\R^3)$,  and by  Lemma
\ref{Lemma 3.6}, up to a subsequence, there exists $u \in
H^s_{rad}(\R^3)$ such that $ u_n\rightharpoonup u $ weakly in
$H^s_{rad}(\R^3),$ $u_n\rightarrow u$ strongly in  $L^t(\R^3)$, for
$\ t\in (2,2^*_s),$ $ u_n\rightarrow u \mbox{ a.e.}$  on $\R^3.$ In
view of $2+\frac{4s}{3}<q<2^*_s$, and Lemmas \ref{Lemma 3.3},
\ref{Lemma 3.6}, then
\begin{equation}\label{e6.9}
\lim_{n\rightarrow\infty}\int_{\R^3}|u_n|^q dx=\int_{\R^3}|u|^q
dx,~~\lim_{n\rightarrow\infty}\int_{\R^3}\phi_{u_n}^tu_n^2dx=\int_{\R^3}\phi^t_{u}u^2dx.
\end{equation}
  Suppose by contradiction that, $u\equiv0$. Then, by \eqref{e6.9} and $P_{\mu}(u_n)=o_n(1)$, we deduce
  as
\[\begin{split}
 o_n(1) &= \|(-\Delta)^{\frac{s}{2}}u_n\|^2_2+\frac{3-2t}{4s}\lambda\int_{\R^3}\phi^t_{u_n}u_n^2dx-\mu
\delta_{q,s}\int_{\R^3}|u_n|^qdx-\int_{\R^3}|u_n|^{2^*_{s}}dx\\
&=\|(-\Delta)^{\frac{s}{2}}u_n\|^2_2-\int_{\R^3}|u_n|^{2^*_{s}}dx+
o_n(1).
\end{split}
\]
Without loss of generality, we may assume that
\[\int_{\R^3}|(-\Delta)^{\frac{s}{2}}u_n|^2dx\rightarrow \ell,~~  \mbox{ and } ~~\int_{\R^3}|u_n|^{2^*_s}dx\rightarrow \ell,\]
as $n\rightarrow\infty$.   By Sobolev's inequality we get  $\ell
\geq S\ell^{\frac{2}{2^*_s}},$ and so, either $\ell\geq
S^{\frac{3}{2s}}$ or $\ell = 0.$
\par  If $\ell\geq S^{\frac{3}{2s}}$, then from $I_{\mu}
(u_n)\rightarrow c_{\mu}(a),
 P_{\mu}(u_n)\rightarrow 0$, we have
\[\begin{split}
&c_{\mu}(a) + o_n(1)\\
 &= I_{\mu} (u_n)=I_{\mu}(u_n)-\frac{1}{s2^*_s}P_{\mu}(u_n)\\
 &=\frac{s}{3}\|(-\Delta)^{\frac{s}{2}}u_n\|_2^2+\lambda\frac{s2^*_s+2t-3}{4s2^*_s}\int_{\R^3}\phi^t_{u_n}u_n^2dx-\mu\frac{2^*_s-q\delta_{q,s}}{q2^*_s}
  \int_{\R^3}|u_n|^qdx+o_n(1)\\
  &=\frac{s}{3}\ell+o_n(1),
\end{split}\]
which implies that  $c_{\mu}(a)= \frac{s}{3}\ell$, and so,
$c_{\mu}(a)\geq \frac{s}{3} S^{\frac{3}{2s}}$, but this is
 impossible since by Lemma \ref{Lemma 6.1}, there exists some $\mu^*_1:=\mu^\ast_1(a)>0$ such that $c_{\mu}(a)<\frac{s}{3} S^{\frac{3}{2s}}$ as
 $\mu>\mu^*_1$.
 \par If  $\ell = 0$, then we have  $ \|(-\Delta)^{\frac{s}{2}}u_n\|^2_2\rightarrow 0,$ thus
 $I_{\mu} (u_n)\rightarrow 0$, which is absurd since $c_{\mu}(a)>0.$
Therefore, $u\not\equiv0$.\end{proof}

\begin{lemma}\label{Lemma 6.4} $\{\alpha_n\}$ is bounded in $\R$, and
$\limsup_{n\rightarrow\infty}|\alpha_n|\leq \frac{C}{a^2}c_\mu(a)$
has the following estimation:
\[
\alpha_n
   =\frac{1}{a^2}\bigg[\lambda\frac{2t+4s-3}{4s}\int_{\R^3}\phi_{u_n}^tu_n^2dx+\frac{q(3-2s)-6}{2qs}\mu\int_{\R^3}|u_n|^qdx\bigg]+o_n(1).
\] Moreover, there exists some $\mu^*_2:=\mu^*_2(a)>0$ such that
$\lim_{n\rightarrow+\infty}\alpha_n=\alpha<0$, if $\mu>\mu^*_2$
large.
\end{lemma}
\begin{proof}
By \eqref{e6.3} and the fact that $u_n\in S_{r,a}$, we have
\[\begin{split}
\int_{\R^3}|(-\Delta)^{\frac{s}{2}}u_n|^2dx+\lambda
\int_{\R^3}\phi_{u_n}^tu_n^2dx-\int_{\R^3}g(u_n)u_ndx & =\alpha_n
\int_{\R^3}|u_n|^2dx+o_n(1)\\&=\alpha_n a^2+o_n(1).
\end{split}
\]
It indicates that
\[
\alpha_n
=\frac{1}{a^2}\left[\int_{\R^3}|(-\Delta)^{\frac{s}{2}}u_n|^2dx+\lambda
\int_{\R^3}\phi_{u_n}^t |u_n|^2dx-\int_{\R^3}g(u_n)u_ndx
\right]+o_n(1).
\]
By Lemma \ref{Lemma 5.6} we have  the boundedness of $\{u_n\}$ in
$H^s_{rad}(\R^3)$, and so, $\{\alpha_n\}$ is bounded in $\R$. By
Lemma \ref{Lemma 6.2} we know  that
$\limsup_{n\rightarrow\infty}|\alpha_n|\leq \frac{C}{a^2}c_\mu(a)$.
Moreover, together  with $P_{\mu}(u_n)\rightarrow0$ as
$n\rightarrow\infty$, we  derive as
\[\begin{split}
\alpha_n
 =&\frac{1}{a^2}\left[\int_{\R^3}|(-\Delta)^{\frac{s}{2}}u_n|^2dx+\lambda
\int_{\R^3}\phi_{u_n}^t
|u_n|^2dx-\int_{\R^3}g(u_n)u_ndx-\frac{1}{s}P_{\mu}(u_n)
\right] +o_n(1)\\
   =&\frac{1}{a^2}\bigg[\lambda\frac{2t+4s-3}{4s}\int_{\R^3}\phi_{u_n}^tu_n^2dx+\frac{q(3-2s)-6}{2qs}\mu\int_{\R^3}|u_n|^qdx\bigg]+o_n(1).
\end{split}
\]By  \eqref{e6.9}  and similar arguments to that of
\eqref{e4.32}-\eqref{e4.35}, we see that there exists
$\mu^*_2:=\mu^*_2(a)>0$, such that
 \be \label{e6.10}\begin{split}
\alpha&=\lim_{n\rightarrow\infty}\alpha_n\\
   &=\lim_{n\rightarrow\infty}\frac{1}{a^2}
    \left\{\lambda\frac{2t+4s-3}{4s}\int_{\R^3}\phi_{u_n}^tu_n^2dx+\frac{q(3-2s)-6}{2qs}\mu\int_{\R^3}|u_n|^qdx+o_n(1)\right\}\\
&= \frac{1}{a^2}\bigg[\lambda\frac{2t+4s-3}{4s}\int_{\R^3}\phi_{u}^tu^2dx+\frac{q(3-2s)-6}{2qs}\mu\int_{\R^3}|u|^qdx\bigg]\\
   &<0,
   \end{split}
\ee  for $\mu>\mu^*_2$ large.
\end{proof}

\par Subsequently, using  the concentration-compactness principle, we
 derive  the following lemma, whose  proof is similar to that of Lemma \ref{Lemma 4.3} in
Section 5, we omit its details here.

\begin{lemma}\label{Lemma 6.5} For $\mu>\mu^\ast:=\max\{\mu^*_1,\mu^*_2\}$, there holds $\int_{\R^3}|u_n|^{2^*_s}dx\rightarrow
\int_{\R^3}|u|^{2^*_s}dx.$
\end{lemma}
\par With the help of the above technical lemmas, we can prove
Theorem\ref{Theorem 2.3} as follows. \vskip0.1in \noindent    {\em
Proof of Theorem \ref{Theorem 2.3}.} Let
$\mu>\mu^\ast:=\max\{\mu^*_1,\mu^*_2\}.$ From Lemmas \ref{Lemma
5.1}, \ref{Lemma 5.2},  the functional $I_{\mu}$ satisfies the
 Mountain pass geometry, from Propositions
\ref{Prop 5.4},\ref{Prop 5.5}, there exist a
$(PS)_{c_{\mu}(a)}$-sequence $\{u_n\}\subset S_{r,a}$ satisfying
\eqref{e6.3}, \eqref{e6.4}, which  is bounded in $H^s_{rad}(\R^3)$,
and  there exists $u \in H^s_{rad}(\R^3)$ such that $
u_n\rightharpoonup u $ weakly in $H^s_{rad}(\R^3),$ $u_n\rightarrow
u$ strongly in  $L^p(\R^3)$, for $p\in (2,2^*_s).$ Moreover, by
Lemmas \ref{Lemma 6.1}-\ref{Lemma 6.4}, we have that
$\alpha_n\rightarrow \alpha<0$ as $ n\rightarrow+\infty$.   By the
weak convergence of $u_n\rightharpoonup u$ in $H^s_{rad}(\R^3)$,
\eqref{e6.3} and \eqref{e6.4}, we have that $u$ solves the equation
\begin{equation}\label{e6.11}
(-\Delta)^su+\phi^t_{u}u-\mu|u|^{q-2}u-|u|^{2^*_s-2}u=\alpha u.
\end{equation}
Therefore,  from   \eqref{e6.9}-\eqref{e6.11} and Lemma \ref{Lemma
6.5}, it follows that
\[\begin{split}
 \|(-\Delta )^\frac{s}{2}
u\|^2_2+\lambda\int_{\R^3}\phi^t_uu^2dx-\alpha\|u\|_2^2
  &=\mu\|u\|_q^q+\int_{\R^3}|u|^{2^*_s} dx\\
 &= \lim_{n\rightarrow\infty}\left[\mu\|u_n\|_q^q+\int_{\R^3}|u_n|^{2^*_s} dx\right]\\
&=\lim_{n\rightarrow\infty}[\|(-\Delta )^\frac{s}{2} u_n\|^2_2+\lambda\int_{\R^3}\phi^t_{u_n}u_n^2dx-\alpha_n\|u_n\|_2^2]\\
&=\lim_{n\rightarrow\infty}[\|(-\Delta )^\frac{s}{2}
u_n\|^2_2-\alpha_n\|u_n\|_2^2]+\lambda \int_{\R^3}\phi^t_uu^2dx.
\end{split}
\]
Since $\alpha<0$,   as in the proof of Lemma \ref{Lemma 4.3}, we can
derive as
\[
\lim_{n\rightarrow\infty}\|(-\Delta )^\frac{s}{2}
u_n\|^2_2=\|(-\Delta )^\frac{s}{2} u\|^2_2
~~\mbox{and}~~\lim_{n\rightarrow\infty}\|u_n\|_2^2=\|u\|_2^2.
\]
Therefore, $u_n\rightarrow u$ in $H^s_{rad}(\R^3)$ and $\|u\|_2=a$.
This completes the proof. \qed

\vskip0.2in

\vskip 10mm
\noindent\textbf{Conflict of interest.} The authors have no competing interests to declare for  this article.

\medskip
\noindent\textbf{Data availability statement.} We declare that the manuscript has no associated data.

\medskip


\bigskip
\medskip
\end{document}